\documentclass[pdflatex,sn-mathphys-num]{sn-jnl}

\makeatletter
\AtBeginDocument{%
  \newgeometry{
    paperwidth=210mm,paperheight=297mm,
    top=26mm, headsep=5.15mm,
    left=30mm,right=30mm,  
    bottom=25mm,
    marginparsep=5mm, marginparwidth=12mm,
    bindingoffset=0mm
  }%
}
\makeatother

\usepackage{graphicx}%
\usepackage{multirow}%
\usepackage{amsmath,amssymb,amsfonts}%
\usepackage{amsthm}%
\usepackage{mathrsfs}%
\usepackage[title]{appendix}%
\usepackage{xcolor}%
\usepackage{textcomp}%
\usepackage{manyfoot}%
\usepackage{booktabs}%
\usepackage{algorithm}%
\usepackage{algorithmicx}%
\usepackage{algpseudocode}%
\usepackage{listings}%

\usepackage{adjustbox}
\usepackage{subcaption}
\usepackage{booktabs}
\usepackage{todonotes}
\usepackage[
    capitalise,
    noabbrev,
  ]{cleveref}

  \usepackage[utf8]{inputenc}
\usepackage[T1]{fontenc}
\usepackage{tikz}
\usetikzlibrary{matrix, positioning, decorations.pathreplacing}

\theoremstyle{thmstyleone}%
\newtheorem{theorem}{Theorem}
\newtheorem{proposition}[theorem]{Proposition}%
\newtheorem{lemma}[theorem]{Lemma}
\newtheorem{corollary}[theorem]{Corollary}

\theoremstyle{thmstyletwo}%
\newtheorem{example}{Example}%
\newtheorem{remark}{Remark}%

\theoremstyle{thmstylethree}%

\let\div\relax		
\DeclareMathOperator{\cl}{cl}
\DeclareMathOperator{\div}{div}
\DeclareMathOperator{\vecdiv}{\boldsymbol{\div}}

\DeclareMathOperator{\grad}{\nabla}
\DeclareMathOperator{\supp}{supp}
\DeclareMathOperator{\tangent}{\nabla^{\perp}\!}

\newcommand{\ph}{\vec{p}^{h,*}\!}
\DeclareMathOperator{\vecgrad}{\boldsymbol{\nabla}\!}
\DeclareMathOperator*{\argmin}{argmin}

\def\loc{\mathrm{loc}}

\newcommand{\dx}{\;\mathrm{d}x}
\newcommand{\N}{\mathbb{N}}

\newcommand{\R}{\mathbb{R}}

\newcommand{\B}{\mathbf{B}}

\newcommand{\mcG}{\mathcal{G}}
\newcommand{\bfK}{\mathbf{K}}
\newcommand{\mcD}{\mathcal{D}}
\newcommand{\mcF}{\mathcal{F}}
\newcommand{\mcX}{\mathcal{X}}
\newcommand{\range}{\mathcal{R}} 
\newcommand{\Dom}{\operatorname{Dom}} 

\DeclareMathOperator{\Null}{\mathcal{N}} 

\DeclareMathOperator{\PSNR}{PSNR}
\DeclareMathOperator{\MSSIM}{MSSIM}
\DeclareMathOperator{\MSE}{MSE}
\DeclareMathOperator{\Perf}{Perf}

\newcommand{\gt}{\ensuremath{d_{gt}}}	
\newcommand{\TV}{TV}
\newcommand{\nf}{\vec{\tau}^\perp} 
\newcommand{\mcP}[1][]{\ensuremath{\mathcal{P}_{#1}}} 
\newcommand{\mcPh}[1][]{\ensuremath{\mathcal{P}^h_{#1}}} 

\begin{document}

\title[Analysis and Domain Decomposition for TV-Stokes]{Functional Analysis and Parallel Domain Decomposition for the TV-Stokes Model}


\author*[1]{\fnm{Andreas} \sur{Langer}}\email{andreas.langer@math.lth.se}

\author[1]{\fnm{Marc} \sur{Runft}}\email{runftmarc@gmail.com}

\author[2]{\fnm{Talal} \sur{Rahman}}\email{talal.rahman@hvl.no}
\author[3]{\fnm{Xue-Cheng} \sur{Tai}}\email{xtai@norceresearch.no}
\author[4]{\fnm{Bin} \sur{Wu}}\email{dr.bwu@ustc.edu}

\affil[1]{\orgdiv{Centre for Mathematical Sciences}, \orgname{Lund University}, \orgaddress{\street{Box 118}, \city{Lund}, \postcode{22100}, \country{Sweden}}}

\affil[2]{\orgdiv{Faculty of Engineering}, \orgname{Western Norway University of Applied Sciences}, \orgaddress{\street{Inndalsveien 28}, \city{Bergen}, \postcode{5063}, \country{Norway}}}

\affil[3]{\orgname{Norwegian Research Centre}, \orgaddress{\street{Nyg{\aa}rdsgaten 112}, \city{Bergen}, \postcode{5008}, \country{Norway}}}

\affil[4]{\orgname{EDInsights AS}, \orgaddress{\street{Hoffsveien 13}, \city{Oslo}, \postcode{0275}, \country{Norway}}}


\abstract{The TV-Stokes model is a two-step variational method for image denoising that combines the estimation of a divergence-free tangent field with total variation regularization in the first step and then uses that to reconstruct the image in the second step. 
Although effective in practice, its mathematical structure and potential for parallelization have remained unexplored. 
In this work, we establish a rigorous functional-analytic foundation for the TV-Stokes model. We formulate both steps in appropriate infinite-dimensional function spaces, derive their dual formulations, and analyze the compatibility and mathematical consistency of the coupled system. In particular, we identify analytical inconsistencies in the original formulation and demonstrate how an alternative model resolves them. We also examine the orthogonal projection onto the divergence-free subspace, proving its existence in a continuous setting and establishing consistency with its discrete counterpart. 

Building on this theoretical framework, we develop the first domain decomposition method for TV-Stokes by applying overlapping Schwarz-type iterations to the duals of both steps. Although the divergence-free constraint gives rise to a global projection operator in the continuous model, we show that it becomes locally computable in the discrete setting. This insight enables a fully parallelizable algorithm suitable for large-scale image processing in memory-constrained environments. Numerical experiments demonstrate the correctness of the domain decomposition approach and its usability in parallel image reconstruction.}




\maketitle

\section{Introduction}
Total variation (TV) minimization is a variational regularization technique, first introduced in \cite{ROF}, to address ill-posed inverse problems in image processing due to its ability to preserve discontinuities in the solution.
Let $\Omega\subset \R^2$ be an open, bounded and simply connected domain with Lipschitz boundary representing the image domain. 
For a vector $\vec{v}\in L^1(\Omega, \R^c)$, $c\in\N$, we define the total variation of $\vec{v}$ in $\Omega$ by
	\begin{equation*}
	\begin{split}
		TV(\vec{v}):=\int_{\Omega}|D\vec{v}|
		:=\sup\limits_{\vec{p}=(\vec{p}_1,\ldots,\vec{p}_c)\in \mathcal{C}_0^1(\Omega,\R^{2 \times c})}\Bigg\{\int_{\Omega} &\vec v\cdot\vecdiv\vec{p}\dx\colon\\ &|\vec{p}_i|\leq 1~\text{almost every  (a.e.) in}~\Omega,~i=1,...,c\Bigg\},
		\end{split}
	\end{equation*}
where $\vecdiv : \mathcal{C}_0^1(\Omega,\R^{2 \times c}) \to \mathcal{C}_0^1(\Omega,\R^{c})$ describes the column-wise divergence and $|\cdot|$ denotes the standard Euclidean vector norm. Here $D\vec{v}$ denotes the distributional gradient of $\vec{v}$ and $BV(\Omega,\R^c)$ defines the space of all $L^1$-functions with bounded variation, i.e.\ $BV(\Omega,\R^c) :=\{\vec{v}\in L^1(\Omega,\R^c)\colon TV(\vec{v})<\infty\}$. Equipped with the norm $\|\vec{v}\|_{BV}:= \|\vec{v}\|_{L^1} + TV(\vec{v})$ the space $BV(\Omega,\R^c)$ becomes a Banach space \cite[Thm. 10.1.1]{AtBuMi:14}. For a short overview of different ways to define the total variation for vector-valued functions, we refer the reader to \cite{GoStCr:12}.

While total variation is known to preserve discontinuities, it is also well-known that reconstructions obtained by total variation minimization may suffer from the so-called staircase effect. In the context of image restoration, this effect generates blocky and non-natural structures in the solution \cite{Jalalzai2016}. To overcome this limitation and achieve more natural reconstructions, several higher order regularization strategies have been proposed, such as the total generalized variation \cite{BreKunPoc} and second-order approaches \cite{PapafitsorosSchonlieb:14}. Another strategy is the TV-Stokes model \cite{RaTaOs2007} on which we will focus in this work. This model is a two-step variational method designed to mitigate the staircase effect while preserving edges in image denoising.
 In the first step, a divergent-free tangent field $\vec{\tau}$ is computed by solving
\begin{align}\label{eq:cTFS}
		\min_{\vec\tau\in BV(\Omega,\R^2)\cap L^{2}(\Omega,\R^2)}\left\{TV(\vec{\tau})+\frac{1}{2\delta}\|\vec{\tau}-\vec{\tau}_0\|^2_{L^2}\right\}\text{~~~~~~subject to $\div(\vec{\tau}) = 0$},
	\end{align}
where $\vec{\tau}_0\in L^2(\Omega,\R^2)$ is the given tangent field of an observed image $d_0$ and $\delta>0$ weighting the importance of the two terms. 
In the second step, the reconstructed image is recovered from the obtained divergent field $\vec{\tau}$ by solving
\begin{align}\label{eq:cIRP}
		\min_{d\in BV(\Omega,\R)\cap L^{2}(\Omega,\R)}TV(d) + \left\langle d, \div\frac{\nf}{|\nf|}\right\rangle_{L^2} \text{~~~~~~subject to $\|d - d_0\|_{L^2}^2 = \sigma^2$},
	\end{align}
where $\sigma > 0$ denotes the standard deviation of the noise present in the observed image $d_0$.

In \cite{RaTaOs2007}, the minimization problems \eqref{eq:cTFS} and \eqref{eq:cIRP} are solved numerically using an explicit time-marching scheme. To improve computational efficiency, dual formulations of these problems were later proposed in \cite{dualtvstokes:2009}, resulting in an iterative algorithm based on a variant of Chambolle's projection method \cite{Chambolle:2004}. While the dual approach significantly accelerates convergence and performs well for small- to medium-scale images, its applicability to large-scale problems is limited. This is primarily due to the iterative nature of the algorithms required to solve both steps of the TV-Stokes model, which leads to high computational costs when applied to high-resolution data. It should be noted that the dual formulations in \cite{dualtvstokes:2009} were presented without specifying the underlying function spaces. To address this gap, we include a careful derivation of the dual problem with explicit function space considerations in this paper.

To overcome the scalability limitations of these existing solvers, particularly in the context of high-resolution or large-scale imaging problems, it is natural to consider parallelization strategies. Domain decomposition methods provide a principled framework for this purpose, allowing the global problem to be reformulated as a collection of coupled subproblems defined on overlapping or non-overlapping subdomains.

One of the central challenges in developing domain decomposition algorithms for TV minimization stems from the intrinsic properties of the TV functional: it is both non-differentiable and lacks additivity over disjoint domain partitions. More precisely, if the domain $\Omega$ is split into disjoint subdomains  $\Omega_1$ and $\Omega_2$, then the total variation of a function $d$ over $\Omega$ satisfies the decomposition formula; cf.\ \cite[Theorem 3.84]{ambrosio:2000}:
\begin{equation*}
\int_\Omega |D(d_{\mid_{\Omega_1}}+d_{\mid_{\Omega_2}})| =
\int_{\Omega_1}|D(d_{\mid_{\Omega_1}})|+
\int_{\Omega_2}|D(d_{\mid_{\Omega_2}})|+ \int_{\partial \Omega_1 \cap \partial \Omega_2} |d_{\mid_{\Omega_1}}^+-d_{\mid_{\Omega_2}}^-|\text{ d} \mathcal{H}^{1}(x),
\end{equation*}
where $\mathcal{H}^{1}$ denotes the 1-dimensional Hausdorff measure, and $d^+$ and $d^-$ denote the traces of $d$ on the interface from the interior and exterior, respectively. This interface term captures the magnitude of jumps across subdomains, highlighting the fact that preserving continuity or controlled discontinuities across interfaces is essential when designing effective decomposition methods.

It is important to note that for widely used methods such as those in \cite{Car, ChaMat, TaiTse, TaiXu}, the question of convergence to a global minimizer remains open when applied to non-smooth and non-additive problems, as a general convergence theory is still lacking. Nonetheless, in \cite{CheTai} and \cite{XuTaiWan}, subspace correction techniques from \cite{TaiTse, TaiXu} have been successfully applied to smoothed approximations of total variation minimization problems.

The first domain decomposition techniques tailored to the minimization of TV appeared in \cite{DuanTai2012, ForKimLanSch, ForLanSch2010, ForSch, LanOshSch}, with the convergence of the energy and monotonicity properties proved. However, these early methods do not guarantee convergence to the global minimizer in general, as explicitly shown through counterexamples in \cite{Lan2021, LeeNam}. In \cite{HinLan2013, HinLan2014}, a posteriori bounds were introduced, quantifying the distance between the numerically computed solution and the true minimizer. These estimates provide practical assurance that these iterative methods approach the correct solution under suitable conditions.

To address the limitations imposed by non-smoothness and non-additivity of the TV term, dualization techniques have been employed. In particular, \cite{HinLan2015_1} proposed a convergent non-overlapping method for the discrete TV minimization, while \cite{ChaTaiWanYan} established convergence results and even rates  for overlapping decompositions in the continuous setting. These foundational contributions have enabled the development of both overlapping \cite{LangerGaspoz:19} and non-overlapping \cite{LeeNam} domain decomposition methods for the primal TV minimization problem, with theoretical results ensuring convergence to a global minimizer.

Since then, various splitting frameworks have been explored, including additive and accelerated schemes \cite{HilbLanger2022, LeeNamPark2019, LeeParkPark2019, Lee2019fast, LeePark2019, LiZhangChangDuan2021, park2020additive, park2020overlapping, park2021accelerated}. For a comprehensive introduction to domain decomposition methods in the context of TV minimization, the reader is referred to \cite{Lan2021, LeePark:20}.

\subsection*{Contribution and Outline}

The purpose of this paper is twofold. First, we establish a rigorous functional-analytic framework for the TV-Stokes model by formulating both variational subproblems and their duals in appropriate infinite-dimensional function spaces. While the original works \cite{RaTaOs2007, dualtvstokes:2009} considered the model in an infinite-dimensional setting, they did not specify the functional framework in detail, nor did they analyze whether the two variational steps are mathematically compatible or under what assumptions the model is well-defined. By revisiting the TV-Stokes model from a functional-analytic perspective, we clarify the mathematical consistency of the subproblems, the interpretation of the divergence-free constraint, and the mathematical role of projection operators in the dual formulation.

Second, we introduce a domain decomposition method for the TV-Stokes model, tailored for large-scale image processing applications. Although the first subproblem involves a global operator that prevents direct localization in the continuous setting, we demonstrate that a fully localized decomposition becomes possible at the discrete level. This enables the formulation of a domain decomposition method in which each subdomain problem can be solved independently on its respective subdomain. While we do not pursue a continuous domain decomposition formulation here, the analytic framework developed in the first part guides the construction of a consistent and parallelizable discrete algorithm. In particular, we propose an overlapping domain decomposition method for the dual formulation, constituting the first such extension of domain decomposition techniques to the TV-Stokes model.

Our specific contributions are as follows:

\begin{enumerate}
\item We identify two analytical inconsistencies that arise when coupling the two variational subproblems in the original TV-Stokes formulation. To address this, we revisit an alternative formulation previously proposed in \cite{LiRaTa2011}, and demonstrate that it provides a functionally compatible and mathematically consistent model. Our contribution lies in clarifying the need for this alternative from a functional-analytic perspective and in comparing both formulations through analytical arguments and numerical experiments.

\item We rigorously establish the equivalence between the constrained formulation \eqref{eq:cIRP} and its unconstrained counterpart. 
While the structure of the argument follows the general framework in \cite{ChambolleLions1997}, the presence of the term $\left\langle d, \div\frac{\nf}{|\nf|}\right\rangle_{L^2}$ introduces new analytical difficulties. 
In particular, the boundedness from below of the functional in \eqref{eq:cIRP} requires density arguments from \cite{HintermullerRautenberg:15}.

\item We derive the dual formulations of both \eqref{eq:cTFS} and the unconstrained counterpart of \eqref{eq:cIRP}, with a careful treatment of the functional analytic setting. While duality for problems of this type is standard in principle, the divergence-free constraint in \eqref{eq:cTFS} necessitates special attention. In contrast to \cite{dualtvstokes:2009}, where the dual formulations were presented without reference to the function spaces involved, we give a complete derivation in the appropriate variational setting.

\item We analyze the orthogonal projection \( \mcP[K] \) onto the divergence-free subspace 
\[ K:=\{\vec{\tau}\in L^2(\Omega,\R^2) \colon \div \vec\tau = 0\} \subset L^2(\Omega, \R^2). \]
 While an explicit discrete version of this projection was used in \cite{dualtvstokes:2009}, its existence and structure in the infinite-dimensional setting had not been addressed. We show that existence follows from classical results in inverse problems. Moreover, we demonstrate that for \( \mathcal{C}^1 \)-smooth vector fields, the continuous projection agrees with the discrete analogue.

\item  Based on the dual formulation of each step, we construct an overlapping domain decomposition method for the TV-Stokes model, following the framework introduced in  \cite{ChaTaiWanYan}. A key technical challenge is the global nature of the orthogonal projection $\mcP[K]$ onto the divergence-free subspace, which prevents a direct localization in the continuous setting. In the discrete setting, however, we show that this global projection can in fact be computed locally on each subdomain. This key observation enables the formulation of localized subproblems and forms the foundation of our parallel decomposition algorithm for TV-Stokes.
\end{enumerate}

The rest of the paper is organized as follows: In \cref{sec:Fundamentals} we recall frequently used notations and fundamental results relevant for the rest of the paper. The analytic discussion of the TV-Stokes model is presented in  \cref{Sec:DiscussionTVStokes}. In particular, we identify two analytical inconsistencies of the TV-Stokes model and present ways to deal with them in an analytical sound manner. Further we present the dual formulations of the optimization problems of both steps of the TV-Stokes model and analyze the orthogonal projection onto the divergence-free space $K$ in an infinite dimensional setting. In \cref{sec:NumImpl} we introduce a finite difference discretisation of the TV-Stokes model. An analytical and numerical comparison of the two different image reconstruction steps is presented in \cref{sec:ComparisonIR}. \Cref{sec:DD} is devoted to the domain decomposition approach of the TV-Stokes model. More precisely, a discrete overlapping domain decomposition method with local subspace iterations is presented together with numerical experiments, showing its usability. Finally, in \cref{sec:conclusions} we present the conclusion.

\section{Fundamentals}\label{sec:Fundamentals}
For a Banach space $X$ we denote its norm by $\|\cdot\|_X$ and its dual space by $X'$. If we apply a functional $x'\in X'$ on an element $x\in X$, we write $\langle x',x\rangle_{X',X}$. If $X$ is additionally a Hilbert space, we denote its inner product with $\langle\cdot,\cdot\rangle_X$. 
For $X=H^1_0(\Omega,\R)$ we use the shorthand notations $\|\cdot\|_{H^1_0}$ and $\langle\cdot,\cdot\rangle_{H^{-1},H_0^1}$, where $H^{-1}(\Omega,\R)$ denotes the dual space of $H^1_0(\Omega,\R)$. Similarly we write $\|\cdot\|_{H^{\div}}$, $\|\cdot\|_{L^2}$, $\|\cdot\|_{L^1}$, and $\|\cdot\|_{H^{-1}}$.

A function $f: X \to \overline{\R}:=\R \cup\{\pm\infty\}$ is called \emph{proper} if $f(u) < \infty$ for one $u \in X$ and $f(u) > -\infty$ for all $u \in X$. Further $f$ is called $X$-\emph{coercive}, if for any sequence $(v_n)_{n\in\N} \subset X$ we have
  \begin{align*}
    \|v_n\|_X \to \infty
    \implies
    F(v_n) \to \infty.
  \end{align*}
For a convex functional $f : X \to \overline{\R}$, we define the \emph{subdifferential} of $f$ at $v\in X$, as the set valued function $\partial f(v) = \emptyset$ if $f(v)=\infty$, and otherwise as
\[
\partial f(v) = \{v^* \in X' \ : \ \langle v^*, u-v\rangle_{X',X} + f(v) \leq f(u) \ \ \forall u\in X \}.
\]
For a proper convex function $f:X \to \overline{\R}$, the \emph{lower semicontinuous hull} (or \emph{closure}) of $f$, denoted by $\cl f$, is defined as the greatest lower semicontinuous function (not necessarily finite) that is majorized by $f$, i.e., satisfies $\cl f\leq f$.

For an operator $A:X\rightarrow Y$ between two Banach spaces $X$ and $Y$ we denote by $A^*: Y' \to X'$ its adjoint operator. Further the domain, the range and the null space of $A$ is denoted by $\Dom(A)$, $\range(A)$ and $\Null(A)$ respectively. 
If $X$ and $Y$ are Hilbert spaces and $A:X\rightarrow Y$ is a linear bounded operator, then the map $A^{\dagger} : \Dom(A^{\dagger})\subset Y \rightarrow X$, which maps every $g\in \Dom(A^{\dagger})$ to the unique element $f\in X$ with smallest possible norm fulfilling the equation $A^*A f=A^* g$, is called \textit{Moore-Penrose-Inverse} of $A$. The map $A^{\dagger}$ is well-defined for $\Dom(A^{\dagger})=\range(A)\oplus\range(A)^{\perp}.$ Here and in the following we denote the orthogonal complement of a space $M\subset Y$ by $M^\perp$ and by $\mcP[M]: Y \to M$ the orthogonal projection onto the closed subspace $M$. The element $f^{\dagger}=A^{\dagger} g$ is called \textit{minimum-norm-solution} of $A f=g$ \cite[Lemma 2.1.4 and Definition 2.1.5]{rieder:2003}. Further we have the following equivalency.

\begin{lemma}[{\cite[Satz 2.1.1]{rieder:2003}}]\label{thmproj}
Let $X$ and $Y$ be real Hilbert spaces and $A:X\rightarrow Y$ be a linear bounded operator. Then the following statements for $f\in X$ and $g\in Y$ are equivalent:
			\begin{enumerate}[(a)]
				\item $A f=\mcP[\overline{\range(A)}]g$,
				\item $\|A f-g\|_Y\leq \|A\phi-g\|_Y$ ~~for all $\phi\in X$,
				\item\label{thmproj:c} $A^{*}A f=A^{*} g$ ~~in $X'$.
			\end{enumerate}
\end{lemma}
Further two operators $A$ and $B$ with the same domain $\Dom(A)=\Dom(B)$ are said to be equal, written $A=B$, if $Af=Bf$ for all $f\in \Dom(A)=\Dom(B)$; cf.\ \cite[p.~99]{Kreyszig:1991}. 

Let $c\in\N$. We will frequently use the Hilbert space
\begin{equation*}
\begin{split}
H_0^{\div}(\Omega,\R^{2\times c}):=\{\vec{p}=(\vec{p}_1,\ldots, \vec{p}_c)\in L^2(\Omega,\R^{2\times c})\colon &\div\vec{p}_i\in L^2(\Omega,\R), \\
& \vec{p}_i \cdot \vec{e}\mid_{\partial \Omega} = 0 \ \text{for}\ i=1,\ldots,c\},
\end{split}
\end{equation*}
where $\vec{e}$ is the outward normal on $\partial \Omega$. Given a set $X\subset\{\vec p=(\vec p_1,...,\vec p_c) : \Omega\to\R^{2\times c}\}$ of vector-valued functions we define 
\begin{equation}\label{eq:DefB}
\B(X):=\{\vec{p}\in X \ : \ |\vec{p}_i|\leq 1 \text{ a.e.\ in } \Omega\text{~for~}i=1,...,c\}.
\end{equation}

\section{Discussion on TV-Stokes Model}\label{Sec:DiscussionTVStokes}
In this section, we revisit the TV-Stokes model and examine its two-step variational structure from a functional-analytic perspective. We start by analyzing the first step, which is concerned with computing a divergence-free tangent field using total variation minimization. The second step, which recovers the reconstructed image from the divergence-free field obtained in step one, is then treated separately in \cref{Sec:Step2}.
\subsection{Step 1 - Tangent Field Smoothing (TFS)}

\subsubsection{Analytic Discussion}
In the first step of the TV-Stokes model, it is required that the tangent field of the observed (noisy) image lies in $L^2(\Omega,\R^2)$, that is,
\begin{align*}
		\vec{\tau}_0 = \tangent d_0 = (-d_{0y}, d_{0x})^T\in L^2(\Omega,\R^2).
	\end{align*} 
To satisfy this requirement, we assume that $d_0 \in H^1(\Omega,\R)$, which ensures the existence of weak partial derivatives $d_{0x}, d_{0y} \in L^2(\Omega,\R)$, thereby guaranteeing that $\vec{\tau}_0\in L^2(\Omega,\R^2)$.
 Note that if $d_0\in L^2(\Omega,\R)$, then its tangent field might need to be understood in a distributional sense and $\vec{\tau}_0 \in H^{\operatorname{curl}}(\Omega,\R^2)'$, where  $H^{\operatorname{curl}}(\Omega,\R^2) = \{ \vec{n} \in L^2(\Omega,\R^2) \ : \ \operatorname{curl}\vec{n} \in L^2(\Omega,\R)\}$. For a modification of the first step that accommodates such more general situations, see \cite{LiRaTa2011}. Note that such a modification adds additional difficulties in deriving a solution process and hence this might be the reason why no algorithm for this modification has been presented in \cite{LiRaTa2011}. We assume $d_0 \in H^1(\Omega,\R)$, which may be interpreted as a smoothed version of a noisy image. Note that in a discrete setting, which is the relevant case in implementation and practical applications, this does not play any role and no smoothing is needed.

In Step~1 of the TV-Stokes model one is looking for a smoothed tangent field $\vec{\tau} \in BV(\Omega,\R^2) \cap L^2(\Omega,\R^2)$ by solving \eqref{eq:cTFS}. Thereby the constraint $\div(\vec{\tau})=0$ in \eqref{eq:cTFS} is understood in a distributional sense, i.e.,
	 \begin{align*}
			\langle\div\vec\tau,\phi\rangle_{H^{-1},H_0^1} ~:=~\langle\vec{\tau},-\grad\phi\rangle_{L^2}
		\end{align*}
		for $\phi\in H_0^1(\Omega,\R)$.
		Thus, $\div$ is a linear continuous operator from $L^2(\Omega,\R^2)$ into $H^{-1}(\Omega,\R)$ defined with the help of its adjoint operator $-\grad : H_0^1(\Omega,\R)\rightarrow L^2(\Omega,\R^2)$.
	
If $\vec{\tau}_0 \in L^2(\Omega,\R^2)$, then a unique solution of \eqref{eq:cTFS} is guaranteed in $BV(\Omega,\R^2)$ \cite[Theorem 3.2]{LiRaTa2011}. 
We note that if $\vec{\tau}$ is sufficiently smooth such that $\div(\vec{\tau})$ is defined in the classical way, then \eqref{eq:cTFS} has a solution in $H^{\div}(\Omega,\R^2)$, as the constraint $\div(\vec{\tau})=0$ ensures that $\div(\vec{\tau})\in L^2(\Omega,\R)$.

\subsubsection{Projection on Subspace with Zero Divergence}\label{sec:proj}
		To handle the constraint $\div\vec{\tau} = 0$ in the optimization problem \eqref{eq:cTFS} and to derive a respective dual formulation of \eqref{eq:cTFS}, the orthogonal projection $\mathcal{P}_{K}$ on the subspace $K:=\left\{\vec{\tau} \in L^2(\Omega,\R^2)\colon \div\vec{\tau}=0 \right\}$ is required, where $\div: L^2(\Omega,\R^2) \to H^{-1}(\Omega,\R)$. We have that
		\begin{equation}\label{eq:div_bounded}
		\begin{split}
					\|\div\| 
					&=
					\sup_{\vec{\tau}\in L^2(\Omega,\R^2)\setminus\{0\}}\frac{\|\div\vec{\tau}\|_{H^{-1}}}{\|\vec{\tau}\|_{L^2}}
					= \sup_{\vec{\tau}\in L^2(\Omega,\R^2)\setminus\{0\}}\sup_{\phi \in H_0^1(\Omega,\R)\setminus\{0\}}
					\frac{\big|\langle\div\vec{\tau},\phi\rangle_{H^{-1},H_0^1}\big|}{\|\vec{\tau}\|_{L^2}\|\phi\|_{H_0^1}}\\
					&= \sup_{\vec{\tau}\in L^2(\Omega,\R^2)\setminus\{0\}}\sup_{\phi \in H_0^1(\Omega,\R)\setminus\{0\}}\frac{\big|\langle\vec{\tau},-\grad\phi\rangle_{L^2}\big|}{\|\vec{\tau}\|_{L^2}\|\phi\|_{H_0^1}}\\
					&\leq \sup_{\vec{\tau}\in L^2(\Omega,\R^2)\setminus\{0\}}\sup_{\phi \in H_0^1(\Omega,\R)\setminus\{0\}} 
					\frac{\|\vec{\tau}\|_{L^2}\|\grad\phi\|_{L^2}}{\|\vec{\tau}\|_{L^2}\|\phi\|_{H_0^1}}
					\leq 1
					\end{split}
				\end{equation}
and hence $\div$ is linear and bounded. Since $K$ is the null space or kernel of $\div$, by \cite[2.7-10 Corollary, p.98]{Kreyszig:1991} it follows that $K$ is closed. Moreover any constant function lies in $K$, rendering $K$ non-empty. Hence \cite[3.3-1 Thm., p.144]{Kreyszig:1991} implies the existence of an orthogonal projection $\mcP[K]: L^2(\Omega,\R^2) \to K$.

\paragraph{Construction of the orthogonal projection}

Recall that $\grad: H_0^1(\Omega,\R) \to L^2(\Omega,\R^2)$ is adjoint to $-\div$. Let us assume that $\vec{w}\in H^{\div}(\Omega,\R^2)$. Then, by an application of the Lax-Milgram theorem, see e.g.\ \cite[Corollary 5.8, p.140]{Brezis:11}, the Laplace equation 
\begin{equation}\label{eq:NormalEq}
-\div(\grad) u= - \Delta u = \div \vec{w}
\end{equation}
 has a unique weak solution $u\in H^1(\Omega,\R)$. If we additionally assume homogeneous Dirichlet boundary conditions, then the problem has a weak solution $u\in H^1_0(\Omega,\R)$. Hence $\div \vec{w} \in \range{(-\Delta)}$ and there is a Moore-Penrose inverse $-\Delta^\dagger : \range(-\Delta) \oplus \range(-\Delta)^\perp \subseteq H^{-1}(\Omega,\R) \to H_0^1(\Omega,\R)$ such that 
 \begin{align}\label{eq:moorepenroseDivtau}
				u := (-\Delta)^{\dagger}\div\vec{w} = -\Delta^{\dagger}\div\vec{w}.
\end{align}
Now we apply \cref{thmproj} with $X=H_0^1(\Omega,\R)$, $Y=L^2(\Omega,\R^2)$, $A=-\grad$ and $A^{*}=\div$. Then \eqref{eq:NormalEq} is condition \ref{thmproj:c} in \cref{thmproj} and hence equivalent to
			\begin{align*}
				-\grad u=\mathcal{P}_{\overline{\range(-\grad)}}\vec{w}.
			\end{align*}
Plugging \eqref{eq:moorepenroseDivtau} into the latter equation yields
\begin{align*}
\mathcal{P}_{\overline{\range(-\grad)}}\vec{w}=\grad\Delta^{\dagger}\div\vec{w}.
\end{align*}
Note that while $\mcP[\overline{\range(-\grad)}]: L^2(\Omega,\R^2) \to \overline{\range(-\grad)}$ is indeed an orthogonal projection, $\grad\Delta^{\dagger}\div: H^{\div}(\Omega,\R^2) \to \range(-\grad)$ is not even a projection.
			Since for any linear bounded operator $A$ between Hilbert spaces we have $\Null(A^*)^{\perp}=\overline{\range(A)}$, see e.g.\ \cite[Prop. 4.9]{Morrison:2011}, and $\div : L^2(\Omega,\R^2)\rightarrow H^{-1}(\Omega,\R)$ is linear and bounded, we get for all $\vec{w}\in H^{\div}(\Omega,\R^2)$ the representation
			\begin{align*}
				\mathcal{P}_{\Null(\div)}\vec{w} =
				(I-\mathcal{P}_{\overline{\range(-\grad)}})\vec{w}
				= \vec{w} - \grad\Delta^{\dagger}\div\vec{w} \in L^2(\Omega,\R^2).
			\end{align*}
			As $\grad: H^1_0(\Omega,\R) \to L^2(\Omega,\R^2)$, we get $\mcP[\Null(\div)] \vec{w}\in \Null(\div)\subset L^2(\Omega,\R^2)$, but $\Null(\div)\not\subseteq H^{\div}(\Omega,\R^2)$. Note that $\grad\Delta^{\dagger}\div\vec{w}\in L^2(\Omega,\R^2)$ only and its divergence may be just understood in a weak sense. 
			
If $\vec{w}\in C^1(\Omega,\R^2)$ such that $\div\vec{w}\in C(\Omega,\R)$, that is $\div_{\mid_{C^1(\Omega,\R^2)}}$, i.e., $\div$ restricted to $C^1(\Omega,\R^2)$, maps into $C(\Omega,\R)$. Then it is well-known that \eqref{eq:NormalEq} has a solution $u\in C^2(\Omega,\R)$ and $\Delta^\dagger_{\mid_{C(\Omega,\R)}}$ maps onto $C^2(\Omega,\R)$, while $\grad_{\mid_{C^2(\Omega,\R)}}$ maps into $C^1(\Omega,\R^2)$. In this setting one can easily verify that $(I-\grad\Delta^{\dagger}\div)_{\mid_{C^1(\Omega,\R^2)}}$ is indeed an orthogonal projection.

\subsubsection{Dual Formulation}
Based on dualization techniques, see \cref{Sec:Dualization}, a dual formulation of \eqref{eq:cTFS} can be derived. 
\begin{corollary}[Dualization of Tangent Field Smoothing]\label{Cor:tfsdual}
Let $\vec\tau_0\in K$ and $\delta>0$. Then the solution $\vec{\tau}$ of \eqref{eq:cTFS} fulfills the equation
			\begin{align*}
			\vec\tau = \vec{\tau}_0-\delta\mcP[K] \vecdiv \vec{p},
			\end{align*}
where $\vec{p}$ is a solution of the optimization problem
\begin{align}\label{eq:tfs_dual}
\min\limits_{\vec p\in\B(H_0^{\div}(\Omega,\R^{2\times 2}))}\left\| \mcP[K] \vecdiv \vec{p}-\delta^{-1} \tau_0\right\|_{L^2}^2.
\end{align}
\end{corollary}
\begin{proof}
Applying \cref{Thm:Dual} with $c=2$, $\bfK:= K$, $\beta=\delta^{-1}$ and $l= -\beta \vec{\tau}_0$ yields the assertion. 
\end{proof}

Note that the duality between \eqref{eq:cTFS} and \eqref{eq:tfs_dual} only holds if $\vec{\tau}_0\in K$. If $\vec{\tau}_0 \in L^2(\Omega,\R^2)\setminus K$, then we simply exchange $\vec{\tau}_0$ in \eqref{eq:cTFS} and \cref{Cor:tfsdual} by $\mcP[K]\vec{\tau}_0$, which is obviously guaranteed to be in $K$, and the duality holds again, i.e., \eqref{eq:tfs_dual} becomes
\begin{align}\label{eq:tfs_dualPK}
\min\limits_{\vec p\in\B(H_0^{\div}(\Omega,\R^{2\times 2}))}\left\| \mcP[K] \vecdiv \vec{p}-\delta^{-1} \mcP[K]\tau_0\right\|_{L^2}^2
\end{align}
and 
\begin{align*}
			\vec\tau = \mcP[K]\vec{\tau}_0-\delta\mcP[K] \vecdiv \vec{p},
\end{align*}
where $\vec{p}$ is a solution of \eqref{eq:tfs_dualPK}, solves
\begin{align*}
		\min_{\vec\tau\in BV(\Omega,\R^2)\cap L^{2}(\Omega,\R^2)}\left\{TV(\vec{\tau})+\frac{1}{2\delta}\|\vec{\tau}-\mcP[K]\vec{\tau}_0\|^2_{L^2}\right\}\text{~~~~~~subject to $\div(\vec{\tau}) = 0$}.
	\end{align*}

To guarantee that $\vec{\tau}_0\in K$ holds without requiring an additional projection, one would have to impose the strong regularity assumption $d_0\in C^2(\Omega,\R)$. Under this condition, Clairaut's theorem justifies the interchange of mixed partial derivatives, implying that $\div\vec\tau_0=0$ and hence $\vec{\tau}_0\in K$. However, even slightly weakening this assumption, e.g., requiring only that $d_0$ is twice differentiable, may already fall outside the scope of Clairaut's theorem, and thus $\vec{\tau}_0\in K$ can no longer be guaranteed. This poses a conceptual issue, as it is generally unrealistic to model a noisy image with a function as smooth as $C^2$. Consequently, instead of \eqref{eq:tfs_dual} in the sequel we will consider \eqref{eq:tfs_dualPK}.

\subsection{Step 2 - Image Reconstruction} \label{Sec:Step2}
Given the tangent field $\vec{\tau}\in BV(\Omega,\R^2) \cap K$ obtained from Step 1 of the TV-Stokes model, the image is reconstructed by solving \eqref{eq:cIRP}.
There are two issues with formulation \eqref{eq:cIRP} on which we will comment next: 
First, unfortunately \eqref{eq:cIRP} is not well-defined, as there might exist $x\in\Omega$ such that $\nf(x)=0$. One may try to fix this by smoothing, i.e., for example, introducing
\begin{equation}\label{eq:defXi2}
	\vec{\xi}:= \frac{\nf}{|\nf|_\epsilon} := \frac{\nf}{\sqrt{\tau_1^2 + \tau_2^2 + \epsilon}}
	\end{equation}
with $\epsilon>0$ and rewrite \eqref{eq:cIRP} as
	\begin{align}\label{eq:cIRPXi}
		\min_{d\in BV(\Omega,\R)\cap L^2(\Omega,\R)}TV(d) + \left\langle d, \div\vec{\xi}\right\rangle_{L^2} \text{~~~~~~subject to $\|d - d_0\|_{L^2}^2 = \sigma^2$}.
	\end{align}
Note that for \eqref{eq:defXi2} we have $|\vec{\xi}(x)|\leq 1$ for almost all $x\in\Omega$ and hence $\vec{\xi}\in L^{\infty}(\Omega,\R^2)\subset L^2(\Omega,\R^2)$.  
	
Second, the normal field $\nf$ as well as $\vec{\xi}$ might not be smooth enough and thus $\div(\nf)$ and $\div\vec{\xi}$ might be only defined in a distributional sense. 
This would put more regularity on the solution $d$ in \eqref{eq:cIRP} and \eqref{eq:cIRPXi}, and would lead to a solution in $W^{1,1}(\Omega,\R)$, a subspace of $BV(\Omega,\R)$, cf.\ \cref{Rem:IRPW11} below. 
In this situation for the second term in \eqref{eq:cIRPXi} we use the identity $\left\langle d, \div\vec{\xi}\right\rangle_{L^2} := \left\langle -\grad d, \vec{\xi}\right\rangle_{L^2}$. 
If $\vec{\xi}$ is smooth enough, i.e., $\div\vec{\xi}$ exists in the classical sense, then we need $\div\vec{\xi} \in L^2(\Omega,\R)$ and hence $\vec{\xi} \in H^{\div}(\Omega,\R^2)$. For technical reasons in that situation we will even assume that $\vec{\xi} \in H_0^{\div}(\Omega,\R^2)$. The zero boundary condition seems even natural as the natural boundary condition on $d$ are homogeneous Neumann boundary conditions, leading to $\nf\cdot \vec{e} = 0$, where $\vec{e}$ is the outward normal.

We summarize the following assumption that ensure the TV-Stokes model is mathematically well-defined:
\begin{enumerate}
\item $d_0\in H^1(\Omega,\R)$ such that its tangent field $\vec{\tau}_0 \in L^2(\Omega,\R^2)$,
\item $\frac{\nf}{|\nf|}$ is approximated by $\vec{\xi}\in S$, where $S\in\{\B(H_0^{\div}(\Omega,\R^2)), \B(L^\infty(\Omega,\R^2))\}$.
\end{enumerate}

\subsubsection{Constrained versus Unconstrained}
Next we show that \eqref{eq:cIRPXi} can be equivalently written as an unconstrained optimization problem following the same strategy used in \cite{ChambolleLions1997, Langer2017, Langer2017_2}. 

\begin{theorem}\label{Thm:Step2Existence}
Let $d_0\in X:=\overline{BV(\Omega,\R) \cap L^2(\Omega,\R)}^{L^2}$ and $\vec{\xi}\in \B(H_0^{\div}(\Omega,\R^2))$. Then
\begin{align}\label{cIRP2}
		\min_{d\in BV(\Omega,\R)\cap L^2(\Omega,\R)}\left\{TV(d) + \left\langle d, \div\vec{\xi}\right\rangle_{L^2}\right\}\text{~~~~~~subject to $\|d - d_0\|_{L^2}^2 \leq \sigma^2$}
\end{align}
has a solution in $BV(\Omega,\R)\cap L^2(\Omega,\R)$. 
\end{theorem}
\begin{proof}
The assumption $d_0\in \overline{BV(\Omega,\R) \cap L^2(\Omega,\R)}^{L^2}$ guarantees that there is $d \in BV(\Omega,\R)$ such that $\|d-d_0\|_{L^2}^2 \leq \sigma^2$ and hence the feasible set $U :=\{d\in BV(\Omega,\R) \ : \ \|d-d_0\|_{L^2} \leq \sigma^2\}$ is non-empty. The rest of the proof is done in 3 steps:
\begin{enumerate}[1.)]
\item Show that $F(d) := TV(d) + \left\langle d, \div\vec{\xi} \right\rangle_{L^2}+ \|d - d_0\|_{L^2}^2$ is $BV$-coercive. Utilizing H\"older inequality and triangle inequality we obtain 
\begin{align*}
				F(d)
				&=TV(d) + \left\langle d, \div\xi\right\rangle_{L^2}+\|d-d_0\|^2_{L^2} \\
				&\geq TV(d) -\|d\|_{L^2}\|\div\xi\|_{L^2} + \left(\|d\|_{L^2}-\|d_0\|_{L^2}\right)^2	\nonumber \\
				&=TV(d) + \|d\|_{L^2}\left(\|d\|_{L^2}-2\|d_0\|_{L^2}-\|\div\xi\|_{L^2}\right)+\|d_0\|_{L^2}^2\\
				&\geq TV(d) + \underbrace{
				\frac{1}{|\Omega|^{1/2}}\|d\|_{L^1}\Big(\frac{1}{|\Omega|^{1/2}}\|d\|_{L^1}-2\|d_0\|_{L^2}-\|\div\xi\|_{L^2}\Big)}_{\to\infty\text{~for~}\|d\|_{L^1}\to\infty}+\frac{1}{2\mu}\|d_0\|_{L^2}^2.
				\end{align*}
If $\|d\|_{BV}\rightarrow\infty$, then at least one of $\|d\|_{L^1}$ or $TV(d)$ tends to infinity. Hence $F$ is $BV$-coercive.

\item Show that $TV(d) + \left\langle d, \div\xi\right\rangle_{L^2}$ is bounded from below by zero. Since $\B(C_0^{1}(\Omega,\R^2))$ is dense in the sense of $H_0^{\div}(\Omega,\R^2)$ in $\B(H_0^{\div}(\Omega,\R^2))$ \cite{HintermullerRautenberg:15} we obtain
\begin{equation}\label{eq:TV+_bounded}
\begin{split}
TV(d) + \left\langle d, \div \vec{\xi} \right\rangle_{L^2} &\geq TV(d) + \inf_{\vec{\xi} \in \B(H_0^{\div}(\Omega,\R^2))} \left\langle d, \div \vec{\xi} \right\rangle_{L^2}\\
&= TV(d) + \inf_{\vec{\xi} \in \B(H_0^{\div}(\Omega,\R^2))}  -\left\langle d, \div \vec{\xi} \right\rangle_{L^2}\\
&= TV(d) - \sup_{\vec{\xi} \in \B(H_0^{\div}(\Omega,\R^2))} \left\langle d, \div \vec{\xi} \right\rangle_{L^2} \\
&= TV(d) - \sup_{\vec{\xi} \in \B(C_0^{1}(\Omega,\R^2))} \left\langle d, \div \vec{\xi} \right\rangle_{L^2}= 0.
\end{split}
\end{equation}
for all $d\in L^2(\Omega,\R)$.

\item Since $TV(d) + \left\langle d, \div\xi\right\rangle_{L^2}$ is bounded from below, see \eqref{eq:TV+_bounded}, there is an infimal sequence $(d_k)_k \subset U$ of \eqref{eq:cIRPXi}. By the $BV$-coercivity we have that $(d_k)_k$ is bounded in $BV(\Omega,\R)$ and in $L^p(\Omega,\R)$, $1\leq p \leq \frac{N}{N-1}$, due to the Sobolev embedding; see e.g.\ \cite[Thm. 10.1.3]{AtBuMi:14}. Hence there exists a subsequence $(d_{k_\ell})_\ell$ which converges weakly in $L^2(\Omega,\R)$ to some $d^*\in L^2(\Omega,\R)$. Consequently we have $\lim_{\ell\to\infty}\left\langle d_{k_\ell}, \div \vec{\xi} \right\rangle_{L^2} = \left\langle d^*, \div \vec{\xi} \right\rangle_{L^2}$. The weak lower semi-continuity of the total variation with respect to the $L^2(\Omega,\R)$ topology \cite[Thm. 2.3]{AcarVogel:94} yields
\[
\liminf_{\ell \to \infty}TV(d_{k_\ell}) \geq TV(d^*)
\]
and hence $d^*\in BV(\Omega,\R)$. Further $(Dd_{k_\ell})_\ell$ converges weakly as a measure to $Dd^*$ \cite[Lemma 2.1]{AcarVogel:94}. Finally, since $\|\cdot - d_0\|_{L^2}^2$ is convex and continuous it is also weakly lower semi-continuous yielding
\begin{align*}
\|d^* - d_0\|_{L^2}^2 \leq \liminf_{\ell \to \infty} \|d_{k_\ell} - d_0\|_{L^2}^2 \leq \sigma^2.
\end{align*}
Thus, $d^* \in BV(\Omega,\R)\cap L^2(\Omega,\R)$ is a solution of \eqref{cIRP2}.
\end{enumerate}
\end{proof}

\begin{proposition}\label{prop:cIRP}
Assume that $\sigma\leq \|d_0- \int_\Omega d_0\|_{L^2}$ (cf.\ \cite[Remark 2]{ChambolleLions1997}), where $\int_\Omega d_0$ describes the average value of the function $d_0\in X$ in $\Omega$, and $\vec{\xi}\in \B(H_0^{\div}(\Omega,\R^2))$. 
Then the constrained problem \eqref{eq:cIRPXi} is equivalent to \eqref{cIRP2}.
\end{proposition}
\begin{proof}
Let $\tilde{d}$ be a solution of \eqref{cIRP2}. Then there exists $d\in BV(\Omega,\R)$ such that $\tilde{d} = d + \int_\Omega d_0$. We define the continuous function $f(s):= \|d_0 - (sd + \int_{\Omega} d_0) \|_{L^2}^2$ for $s\in[0,1]$. By the assumption $\|d_0- \int_\Omega d_0\|_{L^2} \geq \sigma$ we obtain that $f(0) \geq \sigma^2$ and $f(1) = \|d_0 - \tilde{d} \|_{L^2}^2\leq \sigma^2$. Since $f$ is continuous there exists some $s\in[0,1]$ such that $f(s)=\sigma^2$. Set $d'=sd + \int_\Omega d_0$ such that $\|d_0 - d'\|_{L^2}^2 = \sigma^2$. Then by using \eqref{eq:TV+_bounded} and the fact that for $\vec\xi\in H_0^{\div}(\Omega,\R^2)$ we have
\begin{align}\label{eq:OrthConstDiv}
\left\langle\int_{\Omega}d_0,\div\vec\xi\right\rangle_{L^2}
=-\left\langle\grad\int_{\Omega}d_0,\vec\xi\right\rangle_{L^2}=0,
\end{align}
 we get
\begin{align*}
TV(d') + \left\langle d', \div \vec{\xi} \right\rangle_{L^2} 
&\stackrel{\eqref{eq:OrthConstDiv}}{=} s TV(d) + s \left\langle d+ \int_\Omega d_0, \div \vec{\xi} \right\rangle_{L^2} = s TV(\tilde{d}) + s \left\langle \tilde{d}, \div \vec{\xi} \right\rangle_{L^2} \\
&\stackrel{w.l.s.c.}{\le} \liminf_{k\to\infty} s \left(TV(d_{k}) +  \left\langle d_{k}, \div \vec{\xi} \right\rangle_{L^2}\right) \\
&\stackrel{\eqref{eq:TV+_bounded}}{\le} \liminf_{k\to\infty}TV(d_{k}) +  \left\langle d_{k}, \div \vec{\xi} \right\rangle_{L^2},
\end{align*}
where $(d_{k})_k$ is a minimizing sequence of \eqref{cIRP2} weakly converging to $\tilde{d}$ with respect to the $L^2$-topology. Thus $d'$ is a solution of \eqref{cIRP2} and naturally also of \eqref{eq:cIRPXi}. Since $\tilde{d}$ is a solution of \eqref{cIRP2} we cannot have $TV(d') + \left\langle d', \div \vec{\xi} \right\rangle_{L^2} < TV(\tilde{d}) + \left\langle \tilde{d}, \div \vec{\xi} \right\rangle_{L^2}$. This yields $s=1$ and $d'=\tilde{d}$.
\end{proof}
An obvious consequence of this statement is that \eqref{eq:cIRPXi} has a solution which is the same as the solution of \eqref{cIRP2}, as clearly stated in the proof.
\begin{theorem}\label{Thm:Equiv}
Assume that $0< \sigma\leq \|d_0- \int_\Omega d_0\|_{L^2}$, $d_0\in X$ and $\vec\xi\in \B(H_0^{\div}(\Omega,\R^2))$. Then there exists $\alpha\geq 0$ such that the constrained optimization problem \eqref{eq:cIRPXi} is equivalent to the unconstrained problem 
\begin{align}\label{IRPXi}
		\min_{d\in BV(\Omega,\R)\cap L^2(\Omega,\R)}  TV(d) + \left\langle d, \div\vec{\xi}\right\rangle_{L^2}+\frac{\alpha}{2}\|d-d_0\|^2_{L^2}
	\end{align}
	and possesses a unique solution.
\end{theorem}
\begin{proof}
Set $R(d) = TV(d) + \left\langle d, \div\vec{\xi}\right\rangle_{L^2}$ and
\[
G(d)=\
\begin{cases}
+\infty  & \text{if } \|d-d_0\|_{L^2} >\sigma,\\
0         & \text{if } \|d-d_0\|_{L^2} \leq\sigma.
\end{cases}
\]
Notice that $R$ and $G$ are convex and l.s.c functions and problem \eqref{cIRP2} is equivalent to $\min_d {R}(d)+\mcG(d)$. Noting that $\div\vec{\xi}\in L^2(\Omega,\R)$ we have $\Dom(R) = BV(\Omega,\R) \cap L^2(\Omega,\R)$ and $\Dom (G)=\{u\in L^2(\Omega,\R) : G(u)<+\infty\}$. Since $d_0\in \overline{\Dom(R)}^{L^2}$, there exists $\tilde d\in \Dom(R)$ with $\|\tilde d - d_0\|_{L^2} \leq \sigma/2$. As $G$ is continuous at $\tilde{d}$ by \cite[Prop. 5.6, p. 26]{ekeland:1999} we obtain
\[
\partial ( R + G)(d) = \partial R(d) + \partial G(d) 
\]
for all $d$, where $\partial G(d)=\{0\}$ if $\|d-d_0\|_{L^2}<\sigma$ and $\partial \mcG(d)=\{\alpha (d-d_0), \alpha\geq 0\}$ if $\|d-d_0\|_{L^2}=\sigma$.

If $d$ is a solution of \eqref{cIRP2} and hence of \eqref{eq:cIRPXi}, then 
\[
0\in \partial (R +G)(d) =  \partial R(d) + \partial G(d).
\]
Since any solution of \eqref{eq:cIRPXi} satisfies $\|d-d_0\|_{L^2}=\sigma$, this shows that there exists an $\alpha\geq 0$ such that 
\[
0\in \partial R(d) + \alpha (d-d_0).
\]
Hence, for this $\alpha\geq 0$, $d$ is a minimizer of the problem in \eqref{IRPXi}. 

Conversely, a minimizer $d$ of \eqref{IRPXi} with the above $\alpha$ is obviously a solution of \eqref{eq:cIRPXi} with $\|d-d_0\|_{L^2}=\sigma$, which shows the equivalence.

Moreover, since the functional in \eqref{IRPXi} is strictly convex the minimizer is unique.
\end{proof}
A straightforward calculation shows that \eqref{IRPXi} is equivalent to 
\begin{align}\label{IRPXi_ROF}
		\min_{d\in BV(\Omega,\R)\cap L^2(\Omega,\R)}  TV(d) +\frac{\alpha}{2}\|d-d_0 + \frac{1}{\alpha}\div\vec{\xi}\|^2_{L^2}.
	\end{align}

\begin{remark}\label{Rem:IRPW11}
If $\vec{\xi}\in \B(L^\infty(\Omega,\R^2))$, then under the additional assumptions that $\|\vec{\xi}\|_{L^{\infty}}<1$ and $d_0 = \overline{W^{1,1}(\Omega,\R)\cap L^2(\Omega,\R)}^{L^2}$ we can show similar results as above. That is, similar to \cref{Thm:Step2Existence}, \cref{prop:cIRP} and \cref{Thm:Equiv}, the constrained optimization problem \eqref{eq:cIRPXi} is related to the unconstrained optimization problem  
\begin{align}\label{IRW11}
		\min_{d\in L^2(\Omega,\R)\cap W^{1,1}(\Omega,\R)}\left\{\|\grad d\|_{L^1} - \left\langle\grad d, \vec{\xi}\right\rangle_{L^2}+\frac{\alpha}{2}\|d-d_0\|^2_{L^2}\right\}
	\end{align}
and possesses a solution in $L^2(\Omega,\R)\cap W^{1,1}(\Omega,\R)$. This relies on the fact that $W^{1,1}(\Omega,\R)$ is compactly embedded in $L^1(\Omega,\R)$ and the functional in \eqref{IRW11} is bounded from below, $W^{1,1}$-coercive and weakly lower semicontinuous. 

Remark that the requirement $\|\vec{\xi}\|_{L^{\infty}}<1$ automatically holds, if $\vec{\xi}$ is chosen as in \eqref{eq:defXi2}. 
\end{remark}

Note that $H^1(\Omega,\R) \subset \overline{W^{1,1}(\Omega,\R)\cap L^2(\Omega,\R)}^{L^2} \subset \overline{BV(\Omega,\R) \cap L^2(\Omega,\R)}^{L^2}$ and hence \cref{Thm:Step2Existence,prop:cIRP,Thm:Equiv,Rem:IRPW11} especially hold for $d_0\in H^1(\Omega,\R)$. 

\subsubsection{Dual Formulation}

We next present dual formulations of the unconstrained problems \eqref{IRPXi} and \eqref{IRW11} within their appropriate functional-analytic frameworks. The general dualization strategy underlying these derivations is detailed in \cref{Sec:Dualization}.

\begin{corollary}[Dualization of Image Reconstruction]\label{Cor:irdual}
Let $\vec{\xi}\in\B(H_0^{\div}(\Omega,\R^2))$, $\alpha > 0$ and $d_0 \in L^2(\Omega,\R^2)$. Then the solution $d$ of \eqref{IRPXi} (and \eqref{IRPXi_ROF}) fulfills the equation
			\begin{align}\label{eq:tau_hat}
			d = d_0-\frac{1}{\alpha} (\div \vec{p} + \div \vec{\xi}),
			\end{align}
where $\vec{p}$ is the solution of the optimization problem
\begin{align}\label{eq:IRPXi_dual}
\min\limits_{\vec p\in\B(H_0^{\div}(\Omega,\R^{2}))}\left\| \div \vec{p}-(\alpha d_0 -  \div \vec{\xi})\right\|_{L^2}^2.
\end{align}
\end{corollary}
\begin{proof}
Applying \cref{Thm:Dual} with $c=1$, $\bfK=L^2(\Omega,\R)$, $\beta=\alpha$ and $l= -\beta d_0+\div\vec{\xi}$ yields the assertion. 
\end{proof}
We remark that \cref{Cor:irdual} resembles well-known dualization results of total variation minimization; see e.g.\ \cite{Chambolle:2004,  HiLaAl2023, HintermullerKunisch:04}.

\begin{proposition}
The dual problem of \eqref{IRW11} is given by
\begin{equation}\label{IRW11:dual}
\inf_{\vec{p} \in \tilde{\B}} \|\div \vec{p} - \alpha d_0\|_{L^2}^2, 
\end{equation}
where $\tilde{\B}:=\{\vec{p} \in H_0^{\div}(\Omega,\R^2) \colon |\vec{p}-\vec{\xi}| \leq 1 \text{ a.e.\ in } \Omega\}$. Further the solutions $d$ and $\vec{p}$ of \eqref{IRW11} and \eqref{IRW11:dual} respectively are connected by
\begin{align*}
 -\div \vec{p} = \alpha (d-d_0) &\\
 (-\vec{p} + \vec{\xi}) |\grad d|= \grad d, &\qquad |\vec{p} - \vec{\xi}| \leq 1.
 \end{align*}
\end{proposition}
\begin{proof}
The statement follows by applying Fenchel duality \cite[Remark III.4.2]{ekeland:1999}. The proper, convex and lower semicontinuous functions $\mcF: V \to \overline{\R}$ and $\mcG: W \to \overline{\R}$ are set as
\begin{align*}
\mcF(d) := \frac{\alpha}{2} \|d - d_0\|_{L^2}^2 \qquad \mcG(\grad u):= \int_{\Omega} |\grad u| - \langle \grad u , \vec{\xi}\rangle_{L^2}
\end{align*}
and the linear operator $\grad: V \to W$ with $V=L^2(\Omega,\R)$ and $W = H_0^{\div}(\Omega,\R^2)'$. Note that $\grad:~L^2(\Omega,\R) \to H_0^{\div}(\Omega,\R^2)'$ is bounded, since
\begin{align*}
\|\grad\| 
	&= \sup_{u \in L^2(\Omega,\R) \setminus\{0\}} \frac{\|\grad u \|_{{H_0^{\div}}'}}{\|u\|_{L^2}} 
	= \sup_{u \in L^2(\Omega,\R) \setminus\{0\}}\sup_{\phi \in H_0^{\div}(\Omega,\R^2) \setminus\{0\}}  \frac{|\langle \grad u,
	 \phi\rangle_{{H_0^{\div}}', {H_0^{\div}}}|}{\|u\|_{L^2}\|\phi\|_{H_0^{\div}}} \\
	&= \sup_{u \in L^2(\Omega,\R) \setminus\{0\}}\sup_{\phi \in H_0^{\div}(\Omega,\R^2) \setminus\{0\}}  \frac{|\langle  u,
	 \div\phi\rangle_{L^2}|}{\|u\|_{L^2}\|\phi\|_{H_0^{\div}}}\\
	&\leq \sup_{u \in L^2(\Omega,\R) \setminus\{0\}}\sup_{\phi \in H_0^{\div}(\Omega,\R^2) \setminus\{0\}}  \frac{\| u\|_{L^2} \|
	 \div\phi\|_{L^2}}{\|u\|_{L^2}\|\phi\|_{H_0^{\div}}}
	 \leq 1,	
\end{align*}
 with adjoint $\grad^* = -\div: H_0^{\div}(\Omega,\R^2) \to L^2(\Omega,\R)$. The convex conjugate of $\mcF$ and $\mcG$ can be computed as
 \begin{align*}
 &\mcF^*(u^*) = \frac{1}{\alpha} \|u^* + \alpha d_0\|_{L^2}^2 - \frac{\alpha}{2}\|d_0\|_{L^2}^2\\
 &\mcG^*(\vec{v}^*) = \sup_{\vec{v}\in W} \langle \vec{v},\vec{v}^* \rangle_{W,W'} - \mcG(\vec{v}) = \sup_{\vec{v}\in W} \int_{\Omega} \vec{v} \cdot (\vec{v}^* + \vec{\xi}) - |v| \dx = 
 \begin{cases}
 0 & \text{ if } |\vec{v}^*(x) + \vec{\xi}(x)| \leq 1,\\
 \infty & \text{ otherwise}.
 \end{cases}
 \end{align*}
 Then, according to \cite[Remark III.4.2]{ekeland:1999}, the dual problem is
 \begin{align*}
 \sup_{\vec{p}\in W'} - \mcF^*(-\div \vec{p}) - \mcG^*(-\vec{p}) = - \inf_{\vec{p}\in \tilde{\B}} \|-\div \vec{p} + \alpha d_0\|_{L^2}^2 - \frac{\alpha}{2}\|d_0\|_{L^2}^2
 \end{align*}
 with the optimality conditions $-\div \vec{p} \in \partial \mcF(d)$ and $-\vec{p}\in\partial \mcG(\grad u)$, whereby the first one can be written as
 \begin{equation*}
 -\div \vec{p} = \alpha (d-d_0).
 \end{equation*}
 The second one can be written in a point-wise way as
 \begin{align*}
 -\vec{p} + \vec{\xi} = \frac{\grad d}{|\grad d|} \qquad &\text{if } |\vec{p} - \vec{\xi}| = 1\\
 \grad d = 0 \qquad &\text{if } |\vec{p} - \vec{\xi}| < 1.
\end{align*}  
\end{proof}

\subsection{Alternative Formulation of Step 2 and its Dualization}\label{Sec:AlternativeStep2}

Since \eqref{eq:cTFS} and \eqref{eq:cIRP} do not fit together without any further in-between step, like described above, we consider a modification proposed in \cite{LiRaTa2011}: Let $\vec{\tau}\in L^2(\Omega,\R^2)$ be a solution of Step~1 and solve 
\begin{align}\label{modifiedStep2}
d^*\in \argmin_{d\in BV(\Omega,\R)\cap L^2(\Omega,\R)} TV(d-g) + \frac{\alpha}{2} \|d-d_0\|_{L^2}^2
\end{align}
where $\alpha>0$ and $g$ is such that $\vec{\tau} = (-\frac{\partial g}{\partial y}, \frac{\partial g}{\partial x})$,
i.e.\ $\grad g = \vec{\tau}^\perp$. Assuming that $\Omega$ is a bounded, simple-connected domain with Lipschitz-continuous boundary, which seems quite natural for image domains, the existence of such a $g \in H^1(\Omega,\R)$ is guaranteed and can be constructed with the help of the Fourier transform \cite[Theorem 3.1, p. 37]{GirRav2012}. The optimality condition of \eqref{modifiedStep2} is
\begin{align*}
0\in \partial TV(\cdot - g)(d^*) +  \alpha (d^* - d_0) \quad \Leftrightarrow \quad  \alpha (d_0 - d^*) \in \partial TV(\cdot - g)(d^*). 
\end{align*}
By the definition of the subdifferential this can be rewritten as
\begin{align*}
TV(v - g) \geq TV (d^* - g) + \langle\alpha (d_0 - d^*), v - d^*  \rangle_{V',V} \quad \forall v \in V,
\end{align*}
where $V= BV(\Omega,\R)\cap L^2(\Omega,\R)$. Since $g\in H^1(\Omega,\R) \subset V$ we have
\begin{align*}
TV(\tilde{v}) \geq TV (d^* - g) + \langle\alpha (d_0 - d^*), \tilde{v} + g - d^*  \rangle_{V',V} \quad \forall \tilde{v} \in V
\end{align*}
which means
\begin{align*}
 \alpha (d_0 - d^*) \in \partial TV(\cdot )(d^*-g). 
\end{align*}
Substituting $u^* := d^* - g$ into the latter inclusion yields
\begin{align*}
 \alpha (d_0 - (u^* + g)) \in \partial TV(\cdot )(u^*)
\end{align*}
which is the optimiality condition of 
\begin{align}\label{modifiedStep2_2}
\min_{u\in BV(\Omega,\R)} TV(u) + \frac{\alpha}{2} \|u-(d_0-g)\|_{L^2}^2.
\end{align}
Hence, if $u^*$ solves \eqref{modifiedStep2_2}, then $d^* = u^* + g$ solves \eqref{modifiedStep2}. Note that \eqref{modifiedStep2_2} is the well-known Rudin-Osher-Fatemi-model \cite{ROF} and its dual problem writes as
\begin{align}\label{modifiedStep2_2_dual}
\min_{\vec{p} \in \B(H_0^{\div}(\Omega,\R^2))} \|\div \vec{p} - \alpha(d_0-g)\|_{L^2}^2,
\end{align}
see \cite{HintermullerKunisch:04}, \cref{Cor:irdual} or \cref{Sec:Dualization}. The existence of a solution of \eqref{modifiedStep2_2} and \eqref{modifiedStep2_2_dual} is well-understood, see e.g., \cite{ChambolleLions1997,HintermullerKunisch:04} and \cite{HiLaAl2023} in a more general setting. Moreover, let $d^*$ be a solution of \eqref{modifiedStep2} and $\vec{p}^*$ be a solution of \eqref{modifiedStep2_2_dual}, then
\[
d^* = d_0 -\frac{1}{\alpha} \div \vec{p}^*.
\]
Remark that there exists $\alpha \geq 0$ such that \eqref{modifiedStep2} is equivalent to the constrained problem 
\begin{align*}
\min_{d\in BV(\Omega,\R)} TV(d-g)  \qquad \text{ subject to } \|d-d_0\|_{L^2}^2 \leq \sigma,
\end{align*}
if $0<\sigma \le \|d_0 - \int_\Omega d_0\|_{L^2}^2$ and $d_0 \in X$, which follows directly from \cite{ChambolleLions1997}.

To distinguish between the two formulations of the image reconstruction step within the TV-Stokes framework, we shall refer to the method introduced in this subsection as \textit{Image Reconstruction Variant~2 (IRV2)}, and the one presented in \cref{Sec:Step2} as \textit{Image Reconstruction Variant~1 (IRV1)}.

\section{Discretization}\label{sec:NumImpl}
\subsection{Notations}\label{Sec:Discrete:Notation}
Let $\Omega^h\subset \R^2$ be a discrete image domain consisting of $N_2 \times N_1$ pixels, where $N_1,N_2\in\N$. The pixel centers are denoted by $(\vec{x}_{i,j})_{i=1,...,N_2,~j=1,...,N_1}$, with $\vec{x}_{i,j}=(y_i,x_j)\in\Omega^h$. Here, $i$ and $j$ refers to the row and column indices, respectively. 
We define the discrete coordinate sets $\Omega_x^h=\{x_j\}_{j=1}^{N_1}$ and $\Omega_y^h=\{y_j\}_{j=1}^{N_2}$ as the horizontal and vertical grid points. The pixels are equidistant and the mesh size is given by $h = y_{i}- y_{i-1} = x_{j}- x_{j-1} > 0$ for all $i=2,...,N_2$ and $j=2,...,N_1$.  
Further, we define the extended domain $\tilde{\Omega}^h \subset \R^2$ consisting of $\tilde{N}_2 \times \tilde{N}_1$ points $(\tilde{\vec{x}}_{i,j})_{i=1,...,\tilde{N}_2,~j=1,...,\tilde{N}_1}$ where $\tilde{N}_1 := N_1+1$ and $\tilde{N}_2 := N_2+1$. Similarly as above let $\tilde{\Omega}^h = \tilde{\Omega}_y^h \times \tilde{\Omega}_x^h$ and note that $\Omega^h \subset \tilde{\Omega}^h$. For $\Gamma^h\in \{\Omega^h, \tilde{\Omega}^h\}$ we denote by $\Gamma^h_y $ and  $\Gamma^h_x$ the respective vertical and horizontal grid points such that  $\Gamma^h = \Gamma^h_y \times \Gamma^h_x$. 
Let $A^h$ be a set, then the cardinality (i.e., the number of elements) of the set $A^h$ is denoted by $\# A^h$, for example, $\# \tilde{\Omega}_y^h = N_2+1$.

We approximate functions $\vec{u}$ by a discrete function, denoted by $\vec{u}^h$ and for its evaluation at $\vec{x}_{i,j} \in \Gamma^h\in \{\Omega^h, \tilde{\Omega}^h\}$ we use the shorthand notation  $\vec{u}^h(\vec{x}_{i,j}) = \vec{u}^h_{i,j}$. The considered function spaces are
\begin{align*}
\mcX(\Gamma^h,c):=\{\vec u^h:\Gamma^h\to\R^c\} \quad \text{and} \quad \mcX(\Gamma^h,2\times c):=\{\vec p^h : \Gamma^h\to\R^{2\times c}\}
\end{align*} 
for $c\in\{1,2\}$ and $\Gamma^h\in \{\Omega^h, \tilde{\Omega}^h\}$, where $\mcX(\Gamma^h,2\times 1)$ is identified with $\mcX(\Gamma^h,2)$, with the norms
\[
\|d^h\|_{\mcX(\Omega^h,1)} = \left(h^2 \sum_{i=1}^{N_2}\sum_{j=1}^{N_1} |d^h_{i,j}|^2\right)^{1/2} \qquad \text{for all } d^h\in \mcX(\Omega^h,1),
\]
and 
\[
\|\vec{\tau}^h\|_{\mcX(\tilde\Omega^h,2)} = \left(h^2 \sum_{i=1}^{\tilde{N}_2}\sum_{j=1}^{\tilde{N}_1} |\tau^h_{1,i,j}|^2 + |\tau^h_{2,i,j}|^2\right)^{1/2} \qquad \text{for all } \vec{\tau}^h=(\tau^h_1,\tau^h_2)\in \mcX(\tilde\Omega^h,2),
\]
respectively. For $\Gamma^h\in \{\Omega^h, \tilde{\Omega}^h\}$, we define the scalar product of $u^h,v^h \in \mcX(\Gamma^h,1)$ and $\vec{p}^h=(p^h_1,p^h_2),\vec{q}^h=(q^h_1,q^h_2)\in \mcX(\Gamma^h,2)$, respectively, by
\[
\langle u^h,v^h \rangle_{\mcX(\Gamma^h,1)} = h^2 \sum_{i=1}^{\#\Gamma^h_y}\sum_{j=1}^{\#\Gamma^h_x} u^h_{i,j} v^h_{i,j}, \quad \text{and} \quad \langle \vec{p}^h,\vec{q}^h \rangle_{\mcX(\Gamma^h,2)} = \sum_{k=1}^2 \langle p_k^h, q_k^h \rangle_{\mcX(\Gamma^h,1)}.
\]
For $c\in\{1,2, 2\times 2\}$ and a subset $A^h\subset \Gamma^h$, we define the restriction operator $R_{A^h} : \mathcal{X}(\Gamma^h, c)\to\mathcal{X}(A^h,c)$ as
\begin{align*}
	\left(R_{A^h}\vec{u}^h\right)(x) &:= \vec{u}^h(x), \qquad \text{for all } x\in A^h.
\end{align*}
The forward differences $D_{x,\Omega^h}^+ : \mcX(\Omega^h,1)\to \mcX(\Omega^h,1)$ and $D_{y,\Omega^h}^+ : \mcX(\Omega^h,1)\to \mcX(\Omega^h,1)$ and the backward differences $D_{x,\Omega^h}^- : \mcX(\Omega^h,1)\to \mcX(\Omega^h,1)$ and $D_{y,\Omega^h}^- : \mcX(\Omega^h,1)\to \mcX(\Omega^h,1)$ shall for $i=1,...,N_2$ and $j=1,...,N_1$ be defined as
	\begin{equation*}
	\begin{split}
	(D_{x,\Omega^h}^+ d^h)_{i,j} &:= 
	\begin{cases}
	\frac{1}{h}(d^h_{i,j+1}-d^h_{i,j})& \text{if~} j<N_1, \\
	0 & \text{if~}j=N_1,
	\end{cases} \\
	(D_{y,\Omega^h}^+ d^h)_{i,j} &:=  
	\begin{cases}
	\frac{1}{h}(d^h_{i+1,j}-d^h_{i,j})& \text{if~} i<N_2, \\
	0 & \text{if~}i=N_2,
	\end{cases} \\
	(D_{x,\Omega^h}^- d^h)_{i,j} &:=  
	\begin{cases}
	\frac{1}{h} d^h_{i,j} & \text{if~} j=1, \\
	\frac{1}{h}(d^h_{i,j}-d^h_{i,j-1})& \text{if~} 1<j<N_1, \\
	-\frac{1}{h} d^h_{i,j-1} & \text{if~}j=N_1,
	\end{cases} \\
	(D_{y,\Omega^h}^- d^h)_{i,j} &:=
	\begin{cases}
	\frac{1}{h} d^h_{i,j} & \text{if~} i=1, \\
	\frac{1}{h}(d^h_{i,j}-d^h_{i-1,j})& \text{if~} 1<i<N_2, \\
	-\frac{1}{h} d^h_{i-1,j} & \text{if~}i=N_2.
	\end{cases}
	\end{split}
	\end{equation*}
	Further, we introduce discrete backward differences	with Neumann-boundaries $D_x^{neum-}:\mcX(\Omega^h,1)\to\mcX(\tilde{\Omega}^h,1)$ and $D_y^{neum-}:\mcX(\Omega^h,1)\to\mcX(\tilde{\Omega}^h,1)$ such that
\begin{equation}\label{def:backdiffNeumann}
\begin{split}
(D_x^{neum-}d^h)_{i,j} &:= \frac{1}{h}
\begin{cases}
0,	& j=1, N_1+1, \\
(d_{i,j}-d_{i,j-1}), & j=2,...,N_1, i=1,...,N_2, \\
(d_{i-1,j}-d_{i-1,j-1}), & j=2,...,N_1, i=N_2+1,
\end{cases} \\
(D_y^{neum-}d^h)_{i,j} &:= \frac{1}{h}
\begin{cases}
0,	& i=1, N_2+1, \\
(d_{i,j}-d_{i-1,j}), & i=2,...,N_2, j=1,...,N_1, \\
(d_{i,j-1}-d_{i-1,j-1}), & i=2,...,N_2, j=N_1+1.
\end{cases}
\end{split}
\end{equation}
Accordingly, the discrete gradient $\grad_{\Omega^h}^h: \mcX(\Omega^h,1) \to \mcX(\Omega^h,2)$ and the discrete divergence $\div_{\Omega^h}^h: \mcX(\Omega^h,2) \to \mcX(\Omega^h,1)$ are defined as
\begin{align}
	\grad_{\Omega^h}^h d^h&:=(D^+_{x,\Omega^h} d^h, D^+_{y,\Omega^h} d^h)^T,& \text{for~}d^h\in \mcX(\Omega^h,1), \label{defGradDis}\\
	\div_{\Omega^h}^h \vec{u}^h &= \div^h((u^h_1,u^h_2)^T):=D^-_{x,\Omega^h} u^h_1+ D^-_{y,\Omega^h} u^h_2, &\text{for~}\vec{u}^h\in \mcX(\Omega^h,2), \label{defDivDis}
	\end{align}
	which renders $\div_{\Omega^h}^h$ and $-\grad_{\Omega^h}^h$ adjoint to each other, i.e., $\div_{\Omega^h}^h = -(\grad_{\Omega^h}^h)^*$. 
Analogously, we define the discrete multi-gradient $\vecgrad_{\Omega^h}^{\;h}: \mcX(\Omega^h,2\times 1) \to \mcX(\Omega^h,2\times 2)$ and the discrete multi-divergence $\vecdiv_{\Omega^h}^h: \mcX(\Omega^h,2\times 2) \to \mcX(\Omega^h,2\times 1)$ as
\begin{align*}
\vecgrad_{\Omega^h}^{\;h} \vec{u}^h&:=(\grad^h_{\Omega^h} u_1^h, \grad^h_{\Omega^h} u_2^h)^T,& \text{for~}\vec{u}^h\in \mcX(\Omega^h,2\times 1), \\
\vecdiv_{\Omega^h}^h \vec{p}^h &= \vecdiv_{\Omega^h}^h((\vec{p_1}^h,\vec{p_2}^h)^T):=(\div^h_{\Omega^h}\vec{p_1}^h,\div^h_{\Omega^h}\vec{p_2}^h)^T, &\text{for~}\vec{p}^h\in \mcX(\Omega^h,2\times 2).
\end{align*}
Again, the identity $\vecdiv_{\Omega^h}^h = -(\vecgrad_{\Omega^h}^{\;h})^*$ follows directly from the above definitions of the discrete operators.
Moreover, we define the discrete Laplacian as $\Delta_{\Omega^h}^h: \mcX(\Omega^h,1) \to \mcX(\Omega^h,1)$ with $\Delta_{\Omega^h}^h = \div_{\Omega^h}^h \grad_{\Omega^h}^h$ and denote its Moore-Penrose inverse as $(\Delta_{\Omega^h}^h)^\dagger$.

The discrete analog of the orthogonal projection $\mcP[K]$ is denoted by $\mcPh[K^h]$, where $K^h:=\Null(\div_{\tilde{\Omega}^h}^h) = \{\vec{\tau}^h \in \mcX(\tilde{\Omega}^h,2) \colon \div_{\tilde{\Omega}^h}^h \vec{\tau}^h = 0\}$. In particular, one can show that $\mcPh[K^h]: \mcX(\tilde{\Omega}^h,2) \to K^h$ with
\begin{align}\label{formula:PKh}
\mcPh[K^h] = I^h - \grad_{\tilde{\Omega}^h}^h (\Delta_{\tilde{\Omega}^h}^h)^\dagger \div_{\tilde{\Omega}^h}^h, 
\end{align}
is indeed the unique orthogonal projection onto $K^h$, see \cref{sec:disorthproj}. Here, $I^h$ denotes the discrete identity operator, i.e., $I^h \vec{\tau}^h = \vec{\tau}^h$ for $\vec{\tau}^h\in \mcX(\tilde{\Omega}^h,2)$. 
In order to make the computation of the projection $\mcPh[K^h]$ efficient, we express the discrete Laplacian and its Moore-Penrose inverse in terms of the Discrete Cosine Transform (DCT). By the same considerations as in \cite{dualtvstokes:2009} we obtain that for $d^h\in \mcX(\tilde{\Omega}^h,1)$
\begin{align}\label{form:DiscInvLapl}
				(\Delta_{{\tilde\Omega}^h}^h)^\dagger d^h = (\mathcal{C}^h_{\tilde{\Omega}^h})^{-1} (\tilde{\Delta}^h_{\tilde{\Omega}^h})^\dagger \mathcal{C}^h_{\tilde{\Omega}^h} d^h,
\end{align}
where $(\tilde{\Delta}^h_{\tilde{\Omega}^h})^\dagger : \mcX(\tilde{\Omega}^h,1) \to \mcX(\tilde{\Omega}^h,1)$ is given by
\begin{align}\label{form:DiscInvLaplTild}
				((\tilde{\Delta}^h_{\tilde{\Omega}^h})^\dagger {d}^h)_{i,j} =
				\begin{cases}
				0 & \text{for~} i=j=1, \\
				-\frac{1}{\sigma_{\tilde{N}_{2},j}^2} {d}^h_{1,j} & \text{for~} i=1, j\neq 1, \\
				-\frac{1}{\sigma_{\tilde{N}_{1},i}^2} {d}^h_{i,1} & \text{for~} j=1, i\neq 1, \\
				-\frac{1}{\sigma_{\tilde{N}_{1},i}^2+\sigma_{\tilde{N}_{2},j}^2} {d}^h_{i,j}~~~~~~~~~ & \text{for~} i\neq 1, j\neq 1
				\end{cases}
\end{align}
with $\sigma_{N,j} = \frac{2}{h}\sin\left(\frac{(j-1)\pi}{2 N}\right)$~for $N\in\{\tilde{N}_1,\tilde{N}_2\}$ and 
$\mathcal{C}^h_{\tilde{\Omega}^h} : \mcX(\tilde{\Omega}^h,1)\to \mcX(\tilde{\Omega}^h,1)$ is the 2D DCT. The 2d DCT $\mathcal{C}^h_{\tilde{\Omega}^h}$ can be further represented as
\begin{align}\label{eq:DCT2D}
\mathcal{C}^h_{\tilde{\Omega}^h} d^h &= C_{\tilde{N}_2}~ d^h~ C_{\tilde{N}_1}^T \qquad \text{for any } d^h \in \mcX(\tilde{\Omega}^h,1),
\end{align}
where $C_{\tilde{N}_2}\in\R^{\tilde{N}_2\times \tilde{N}_2}$, $C_{\tilde{N}_1}\in\R^{\tilde{N}_1\times \tilde{N}_1}$ and
\begin{align} 
C_N &:=\sqrt{\frac{2}{N}}\begin{pmatrix}
1/\sqrt{2} & 1/\sqrt{2} & \hdotsfor{1} & 1/\sqrt{2} \\
\cos\big(\frac{\pi}{2 N}\big) & \cos\big(\frac{3\pi}{2 N}\big) & \hdotsfor{1} & 
\cos\Big(\frac{(2N - 1)\pi}{2 N}\Big) \\
\cos\big(\frac{2\pi}{2N}\big) & \cos\big(\frac{6\pi}{2N}\big) & \hdotsfor{1} & 
\cos\Big(\frac{2(2N-1)\pi}{2N}\Big) \\
\cos\big(\frac{3\pi}{2N}\big) & \cos\big(\frac{9\pi}{2N}\big) & \hdotsfor{1} & 
\cos\Big(\frac{3(2N-1)\pi}{2N}\Big) \\
\vdots & \vdots & \ddots & \vdots \\
\cos\Big(\frac{(N-1)\pi}{2N}\Big) & \cos\Big(\frac{(N-1)3\pi }{2N}\Big) & \hdotsfor{1} & 
\cos\Big(\frac{(N-1)(2N-1)\pi}{2N}\Big)
\end{pmatrix}\in\R^{N\times N}
\label{form:DCTmat}
\end{align}
for $N\in\{\tilde{N}_1,\tilde{N}_2\}$. 
Note that the orthogonality of $C_N$ gives us
\begin{align}\label{eq:DCT2Dinv}
(\mathcal{C}^h_{\tilde\Omega^h})^{-1}d^h = C_{\tilde{N}_2}^T~ d^h~ C_{\tilde{N}_1}
\end{align}
for all $d^h\in \mcX(\tilde\Omega^h,1)$.

Further, we define for a set $\mathcal{S}\subseteq \mcX(\Gamma^h,2\times c)$, $c=1,2$, $\Gamma^h\in\{\Omega^h, \tilde{\Omega}^h\}$, of vector-valued discrete images a discrete analogue to \eqref{eq:DefB} as 
\begin{align}\label{eq:DefDisB}
\B^h(\mathcal{S})&:=\{\vec{p}^h\in \mathcal{S} \ : \ |(\vec{p}^h_{k})_{i,j}|\leq 1 \ \text{ for all } \ i=1,...,\#\Gamma_y^h,\; j=1,...,\#\Gamma_x^h \; \text{and} \; k=1,...,c\} 
\end{align}

\subsection{Iteration for Discrete Dual Problems}

Note that all the dual formulations of the subproblems, i.e., \eqref{eq:tfs_dualPK}, \eqref{eq:IRPXi_dual} and \eqref{modifiedStep2_2_dual}, are of the the form
\begin{align}\label{eq:opt_dual}
\min\limits_{\vec{p}\in\B(H_0^{\div}(\Omega,\R^{2\times c}))} \left\{\mcD(\vec{p}) := \left\|\Lambda \vec{p}- \vec{f} \right\|_{L^2}^2\right\}, 
\end{align}
where $\vec{f}\in L^2(\Omega,\R^c)$, $\Lambda : H_0^{\div}(\Omega,\R^{2\times c}) \to L^2(\Omega,\R^c)$ is a bounded linear operator and $c\in\{1,2\}$.
In particular, for $c=2$, with $\Lambda = \mcP[K]\div$ and $\vec{f}=\delta^{-1}\mathcal{P}_K\vec{\tau}_0$ the problem in \eqref{eq:opt_dual} reduces to the dual problem \eqref{eq:tfs_dualPK} corresponding to Step 1 (TFS) of the TV-Stokes model. Since $\mcP[K]$ is an orthogonal projection and thus satisfying $\|\mcP[K]\|=1$ \cite[Proposition 8.4]{hunter2001applied} and $\div$ is bounded as well, see \eqref{eq:div_bounded}, the composition $\mcP[K]\div$ is indeed a bounded linear operator. 
For $c=1$, with $\Lambda = \div$, problem \eqref{eq:opt_dual} becomes the dual problem \eqref{eq:IRPXi_dual} of Step 2 (image reconstruction) in the TV-Stokes model when $\vec{f}=\alpha d_0 - \div\vec{\xi}$, and becomes \eqref{modifiedStep2_2_dual} when $\vec{f}=\alpha(d_0 - g)$.

With the notations of the previous subsection, we define the discrete analogue of \eqref{eq:opt_dual} as $\mcD^h : \mcX(\Gamma^h,2\times c)\to\R$ with $c\in\{1,2\}$ and $\Gamma^h$ being a discretization of $\Omega$, and write
\begin{align}\label{eq:opt_dual_discrete}
\min\limits_{\vec{p}^h\in\B^h(\mcX(\Gamma^h,2\times c))} \left\{\mcD^h(\vec{p}^h):= \left\|\Lambda^h\vec{p}^h-\vec{f}^h\right\|^2_{\mcX(\Gamma^h,c)} \right\},
\end{align}
where $\Lambda^h: \mcX(\Gamma^h,2\times c) \to \mcX(\Gamma^h,c)$ and $\vec{f}^h\in \mcX(\Gamma^h,c)$ are the discrete counterparts of $\Lambda$ and $\vec{f}$, respectively.
More precisely, given the discrete image $d_0^h\in \mcX(\Omega^h,1)$ the discrete dual problems for Step~1 and 2 of the TV-Stokes model read as follows:
\begin{align}
&\min\limits_{\vec{p}^h\in\B^h(\mcX(\tilde{\Omega}^h,2\times 2))} \left\{ \mcD^h_{\mathrm{TFS}}(\vec{p}^h) := \left\| \mcP[K^h]^h \vecdiv^h_{\tilde{\Omega}^h} \vec{p}^h-\delta^{-1} \mcP[K^h]^h\vec{\tau}^h_0\right\|_{\mcX(\tilde{\Omega}^h,2)}^2\right\} &\text{(TFS)}\label{eq:tfsdualdis}
\\
&\min\limits_{\vec p^h\in\B^h(\mcX(\Omega^h,2\times 1))} \left\{ \mcD^h_{\mathrm{IRV1}}(\vec{p}^h) :=\left\| \div^h \vec{p}^h-(\alpha d_0^h -  \div^h \vec{\xi}^h)\right\|_{\mcX(\Omega^h,1)}^2\right\} &\text{(IRV1)}\label{eq:ir1dualdis}
\\
&\min_{\vec{p}^h \in \B^h(\mcX(\Omega^h,2\times 1))} \left\{ \mcD^h_{\mathrm{IRV2}}(\vec{p}^h) := \left\|\div^h \vec{p}^h - \alpha(d_0^h-g^h)\right\|_{\mcX(\Omega^h,1)}^2\right\} &\text{(IRV2)}\label{eq:ir2dualdis}
\end{align}
where $\vec{\tau}_0^h = (-D_y^{neum-} d^h_0, D^{neum-}_x d^h_0)^T\in\mcX(\tilde{\Omega}^h,2)$, $\vec{\xi}^h\in\mcX(\Omega^h,2)$ a discretization of $\vec{\xi}$ and $g^h\in\mcX(\Omega^h,1)$ such that $(D^-_{x,\Omega^h},D^-_{y,\Omega^h})^T g^h = \left(R_{\Omega^h}\vec{\tau}^h\right)^\perp$ with $\vec{\tau}^h = \mcP[K^h]^h\vec{\tau}_0^h-\delta\mcP[K^h]^h \vecdiv^h_{\tilde{\Omega}^h} \ph\in\mcX(\tilde{\Omega}^h,2)$ and $\ph\in\mcX(\tilde{\Omega}^h,2\times 2)$ a solution of \eqref{eq:tfsdualdis}.
In particular, following (\ref{eq:defXi2}), we compute $\vec{\xi}^h\in\mcX(\Omega^h,2)$ from $\vec\tau^h\in\mcX(\tilde\Omega^h,2)$ by setting $\vec\tau^{h,\perp}_{\Omega^h}:=R_{\Omega^h}(\vec\tau^h_2,-\vec\tau^h_1)^T$ and
\begin{align}\label{eq:defXi2:discrete}
\vec{\xi}^h_{i,j} := \frac{\big(\vec\tau^{h,\perp}_{\Omega^h}\big)_{i,j}}{\big|\big(\vec\tau^{h,\perp}_{\Omega^h}\big)_{i,j}\big|_\epsilon} := \frac{\big(\vec\tau^{h,\perp}_{\Omega^h}\big)_{i,j}}{\sqrt{\big(\vec\tau^{h,\perp}_{\Omega^h,1}\big)_{i,j}^2 + \big(\vec\tau^{h,\perp}_{\Omega^h,2}\big)_{i,j}^2 + \epsilon}}
\end{align}
for $\epsilon>0$. 
Note that due to our definitions of the discrete gradient and divergence operators we cannot guarantee that $\vec{\tau}_0^h\in K^h$ in general. Hence also in our discrete setting the projection of $\vec{\tau}_0^h$ onto $K^h$ is needed, see \eqref{eq:tfsdualdis}.
Further, it is crucial to solve \eqref{eq:tfsdualdis} on the extended domain $\tilde{\Omega}^h$ as this allows to keep the boundaries outside the image domain, see \cref{figDualcoordinates}. In particular, this avoids artifacts (on the boundary) in the reconstructed image when using the restricted computed tangent field for solving \eqref{eq:ir1dualdis} and \eqref{eq:ir2dualdis}.
On top of that, specifically for IRV2, it is important to use backward differences to determine $\vec\tau_0^h$ from $d_0^h$ to make sure that the constructed $g^h$ will be in the same coordinate system as $d_0^h$ and $d^h$ (compare Figure \ref{figDualcoordinates}). Note also that such a specific $g^h$ can only be constructed, since $\div_{\tilde\Omega^h}^h\vec{\tau}^h=D^-_{x,\tilde\Omega^h}\tau_1^h+D^-_{y,\tilde\Omega^h}\tau_2^h=0$.  In our setting, if $\div_{\tilde\Omega^h}^h$ were defined by the forward-difference scheme introduced above, such $g^h$ may not exist.
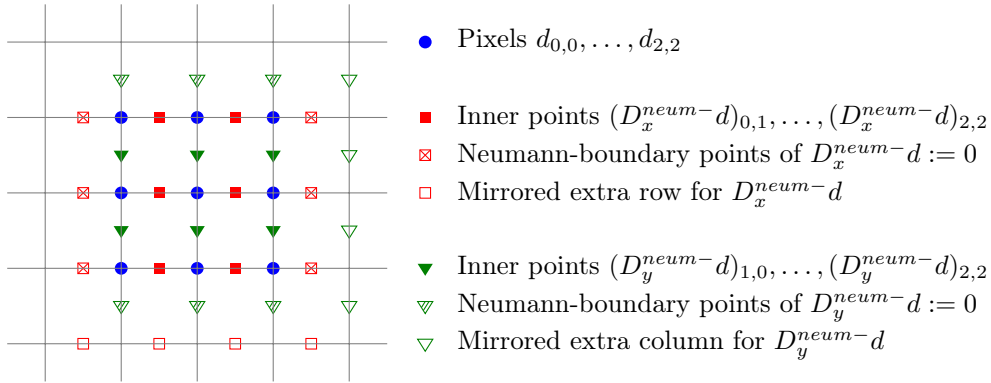
\begin{figure}[htbp]
	\centering
	\begin{tikzpicture}
	\definecolor{darkgreen}{rgb}{0.0, 0.5, 0.0}
	
	\def\xmax{5}
	\def\ymax{5}
	
	\foreach \x in {2,3,4} {
		\foreach \y in {2,3,4} {
			\fill[blue] (\x, \y) circle (2.4pt); 
		}
	}
	
	\foreach \y in {2,3,4} {
		\foreach \x in {2.5,3.5} {
			\fill[red] (\x-0.07, \y-0.07) rectangle ++(0.14,0.14);
		}
		\foreach \x in {1.5,4.5} {
			\draw[red] (\x-0.07, \y-0.07) rectangle 			++(0.14,0.14);
			\draw[red] (\x-0.07,\y-0.07) -- ++(0.14,0.14);
			\draw[red] (\x+0.07,\y-0.07) -- ++(-0.14,0.14);
		}
	}
	\foreach \x in {1.5,2.5,3.5,4.5}{
		\draw[red] (\x-0.07, 1-0.07) rectangle ++(0.14,0.14);
	}
	
	\foreach \x in {2,3,4} {
		\foreach \y in {2.5,3.5} {
			\fill[darkgreen] (\x-0.1, \y+0.06) -- ++(0.1,-0.16) -- ++(0.1,0.16) -- cycle; 
		}
		\foreach \y in {1.5,4.5} {
			\draw[darkgreen] (\x-0.1, \y+0.06) -- ++(0.1,-0.16) -- ++(0.1,0.16) -- cycle; 
			\draw[darkgreen] (\x+0.05,\y+0.06) -- ++(-0.075,-0.12);
			\draw[darkgreen] (\x-0.00,\y+0.06) -- ++(-0.05,-0.08);
		}	
	}
	\foreach \y in {1.5,2.5,3.5,4.5} {
		\draw[darkgreen] (5-0.1, \y+0.06) -- ++(0.1,-0.16) -- ++(0.1,0.16) -- cycle; 
	}
	\draw[gray, thin] (0.5,0.5) grid (\xmax+0.5,\ymax+0.5);
	\begin{scope}[xshift=6cm]
	\fill[blue] (0,5) circle (2.4pt);
	\node[right] at (0.3,5) {Pixels $d_{0,0},\ldots,d_{2,2}$};
	\fill[red] (0-0.07, 4.0-0.07) rectangle ++(0.14,0.14);
	\node[right] at (0.3,4.0) {Inner points $(D_x^{neum-}d)_{0,1},\ldots,(D_x^{neum-}d)_{2,2}$};
	\draw[red] (0-0.07, 3.5-0.07) rectangle ++(0.14,0.14);
	\draw[red] (-0.07,3.5-0.07) -- ++(0.14,0.14);
	\draw[red] (-0.07+0.14,3.5-0.07) -- ++(-0.14,0.14);
	\node[right] at (0.3,3.5) {Neumann-boundary points of $D_x^{neum-}d:=0$};
	\draw[red] (0-0.07, 3.0-0.07) rectangle ++(0.14,0.14);
	\node[right] at (0.3,3.0) {Mirrored extra row for $D_x^{neum-}d$};
	
	\fill[darkgreen] (-0.1, 2.0+0.06) -- ++(0.1,-0.16) -- ++(0.1,0.16) -- cycle; 
	\node[right] at (0.3,2.0) {Inner points $(D_y^{neum-}d)_{1,0},\ldots,(D_y^{neum-}d)_{2,2}$};
	\draw[darkgreen] (0-0.1,1.5+0.06) -- ++(0.1,-0.16) -- ++(0.1,0.16) -- cycle;
	\draw[darkgreen] (0+0.05,1.5+0.06) -- ++(-0.075,-0.12);
	\draw[darkgreen] (0-0.00,1.5+0.06) -- ++(-0.05,-0.08);
	\node[right] at (0.3,1.5) {Neumann-boundary points of $D_y^{neum-}d:=0$};
	\draw[darkgreen] (-0.1, 1.0+0.06) -- ++(0.1,-0.16) -- ++(0.1,0.16) -- cycle; 
	\node[right] at (0.3,1.0) {Mirrored extra column for $D_y^{neum-}d$};
	\end{scope}
	\end{tikzpicture}
	\caption{Primal and dual coordinate systems for $N_2=N_1=3$.
	}\label{figDualcoordinates}
\end{figure}

According to \cite{dualtvstokes:2009, HilbLanger2022} the discrete problem \eqref{eq:opt_dual_discrete} and consequently \eqref{eq:tfsdualdis}-\eqref{eq:ir2dualdis} can be solved with the semi-implicit algorithm presented in \cite{Chambolle:2004}, which we denote by Chambolle's algorithm in the sequel. While for \eqref{eq:ir1dualdis} and \eqref{eq:ir2dualdis} Chambolle's algorithm can be directly applied and its convergence is guaranteed due to \cite[Theorem~3.1]{Chambolle:2004}, a slight modification of the algorithm, see \cref{alg:chambTfs} for the modified variant, is needed to handle \eqref{eq:tfsdualdis} due to the present of the multi-divergence operator and the projection $\mathcal{P}^h_{K^h}$. Similarly as for the algorithm in \cite{Chambolle:2004} the convergence of \cref{alg:chambTfs} can be guaranteed.

\begin{algorithm}[htbp]
	\newcommand{\Break}{\State \textbf{break} }
	\caption{Chambolle's algorithm for Dual Tangent Field Smoothing}\label{alg:chambTfs}
	\begin{algorithmic}[1]
		\Require $\vec{\tau}^h_0 \in \mcX(\tilde{\Omega}^h,2)$, $\delta > 0$, and $t\in(0,\frac{1}{8}]$
		\State $\vec{p}^{h,0}=(\vec{p}_1^{h,0},\vec{p}_2^{h,0})  \gets 0\in \mcX(\tilde{\Omega}^h,2\times 2)$
		\State{$\vec{\psi}^{h,0} = (\vec{\psi}_1^{h,0},\vec{\psi}_2^{h,0})\gets 0\in \mcX(\tilde{\Omega}^h,2\times 2) $}
		\For{$n = 0,1,2,...,max\_it$}
		\State $\vec{\psi}^{h,n} = (\vec{\psi}_1^{h,n},\vec{\psi}_2^{h,n}) \gets \vecgrad^{\;h}_{\tilde{\Omega}^h}(\mathcal{P}^h_{K^h}\vecdiv^h_{\tilde{\Omega}^h} \vec{p}^{h,n} -\delta^{-1}\vec{\tau}^h_0)$
		\State
		$\vec{p}^{h,n+1} = (\vec{p}_1^{h,n+1}, \vec{p}_2^{h,n+1}) \gets \left(\dfrac{\vec{p}_1^{h,n}+t\vec{\psi}_1^{h,n}}{1+t\left|\vec{\psi}_1^{h,n}\right|}, \dfrac{\vec{p}_2^{h,n}+t\vec{\psi_2}^{h,n}}{1+t\left|\vec{\psi}_2^{h,n}\right|} \right) $		
		\If{stop\_criteria}
		\Break
		\EndIf
		\EndFor
	\end{algorithmic}
\end{algorithm}
\begin{theorem}[Convergence of Chambolle's algorithm for Tangent Field Smoothing]
	Let $0<t\leq \frac{1}{8}$. Then Algorithm \ref{alg:chambTfs} generates a sequence $\left(\vec{p}^{h,n}\right)_n \subset\mcX(\tilde{\Omega}^h,2\times 2)$ such that 
	\[\lim\limits_{n\to\infty} \mcP[K^h]^h \vecdiv^h_{\tilde{\Omega}^h} \vec{p}^{h,n} =\mathcal{P}^h_{K^h} \vecdiv^h_{\tilde{\Omega}^h} \vec{p}^{h,*},\]
	where $\vec{p}^{h,*}$ is a minimizer of (\ref{eq:tfsdualdis}).
\end{theorem}
\begin{proof}
Since $\mcP[K^h]^h$ is an orthogonal projection, see \cref{prop:DisProj}, we have $\|\mcP[K^h]^h\|^2 = 1$, where
\begin{align*}
\|\mcP[K^h]^h\|:=\sup\limits_{\|\vec v\|_{\mcX(\tilde\Omega^h,2)}=1}\|\mcP[K^h]^h\vec v\|_{\mcX(\tilde\Omega^h,2)} 
\end{align*}
denotes the induced operator norm. Hence $\|\mcP[K^h]^h \vecdiv^h_{\tilde{\Omega}^h}\|^2 \leq \|\vecdiv^h_{\tilde{\Omega}^h}\|^2 \leq 8$. The rest follows along the lines of the proof of \cite[Theorem 3.1]{Chambolle:2004}. 
\end{proof}

\section{Comparison of IRV1 and IRV2}\label{sec:ComparisonIR}
In this section we revisit the two different approaches on how to solve Step 2 of the TV-Stokes model, i.e., the image reconstruction step, and compare them.

\subsection{Analytic Comparison}
\begin{itemize}
\item \textbf{IRV1:} From a functional-analytic perspective we have observed in \cref{Sec:Step2} that \eqref{eq:cTFS} and \eqref{eq:cIRP} do not fit together and a modification step is required leading to \eqref{IRPXi_ROF}. Hence, in this variant the solution of \eqref{eq:cTFS} cannot be directly used and needs to be modified, for example as discussed in \cref{Sec:Step2}. 

If we assume that $d$ is sufficiently smooth and $\vec\xi= \frac{\vec\tau^\perp}{|\vec\tau^\perp|_{\epsilon}}$ for a sufficiently small $\epsilon>0$, then we can make similar considerations as in \cite[(8)]{LysOshTai2004} for the first two terms of the functional in \eqref{IRPXi}. More precisely, we have
\begin{align*}
	TV(d) + \left\langle d, \div\vec\xi\right\rangle_{L^2}
	&=\int_{\Omega}|\grad d|\dx-\int_{\Omega}\grad d\cdot\vec\xi\dx\\
	&=\int_{\Omega}|\grad d|\dx-\int_{\Omega}|\grad d|   |\xi|\cos\angle(\grad d,\vec\xi)\dx\\
	&= \int_{\Omega}|\grad d|\left(1-\underbrace{|\vec\xi|}_{\approx 1}\cos\angle(\grad d,\vec\tau^\perp)\right)\dx.
	\end{align*}
When minimizing \eqref{IRPXi}, one enforces that $\grad d$ is parallel to the normal field $\vec\xi$ and that $|\grad d|$ is small, since $|\vec\xi| < 1$ a.e.\ in $\Omega$. The latter produces a smoothing effect, which is independent of the tangent field $\vec{\tau}$ computed in Step 1. So even if the determined tangent field is very noisy, the result $d$ can be very smooth since $|\grad d|$ gets forced to be small. This makes IRV1 robust to noise.

\item \textbf{IRV2:} This variant presented in \cref{Sec:AlternativeStep2} suggests to solve \eqref{modifiedStep2_2}, which is chosen such that the solution of Step 1, i.e., \eqref{eq:cTFS}, can be directly used without any modifications. 

Here, we minimize
	\begin{align*}
	TV(d-g)+\frac{\alpha}{2}\|d-d_0\|_{L^2}^2.
	\end{align*}
This means that for $\alpha\to 0$ the minimum would tend to $d=g$. On the other hand, as $\alpha\to\infty$, the solution tends towards $d=d_0$, which corresponds to the noisy observation. Since $g$ is the result of TFS, it should ideally be less noisy than $d_0$. Therefore, regardless of the choice of $\alpha$, the reconstructed image $d$ can only be as denoised as $g$. This implies that a poor choice of the regularization parameter $\delta$ in the TFS step may result in a still noisy $g$, and consequently a noisy $d$. In this sense, the effectiveness of the reconstruction step directly depends on the success of TFS.	
\end{itemize}

Note that in practice both variants rely on a numerical solution of Step~1. Since this approximate solution enters the formulation of Step~2 in both variants, the numerical errors propagate from Step~1 to Step~2 and may accumulate, thereby affecting the reconstruction. In IRV1, small errors in $\tau$ can locally lead to large errors in $\xi$, particularly if $\tau\approx 0$ and $\epsilon$ is small, which can locally amplify the error in the final result of Step~2. In IRV2, small global errors in $\tau$ may accumulate when computing $g$ from $\tau$, potentially leading to large deviations in $g$. Nevertheless, we observed that the reconstructions obtained in practice seem to remain reasonable. These effects must be taken into account when comparing and evaluating the quality of different methods for solving the steps in the TV-Stokes model. We will explicitly consider them when assessing our domain decomposition approach later in the paper; see \cref{sec:DDNumVal} and \cref{fig:dd_energy} below.

\subsection{Numerical Experiments}

In the following we numerically compare the performance of IRV1 and IRV2 within the TV-Stokes model. To this end, we evaluate SNR, PSNR, and MSSIM metrics across a set of 12 test images, comprising 2 phantom images (see \cref{fig:phantomimg}) and 10 real-world images (see \cref{fig:realworldimg}).
\begin{figure}[htb]
	\renewcommand{\arraystretch}{1.2}
		\centering
		\normalsize
		\begin{tabular}{|c|c|}
			\hline
			\textbf{ID} & \textbf{Phantom test images}   \\\hline
			00 and 01 & \parbox[c][5.5cm][c]{0.65\textwidth}{
			\includegraphics[width=5cm]{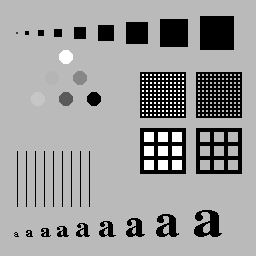}
			\includegraphics[width=5cm]{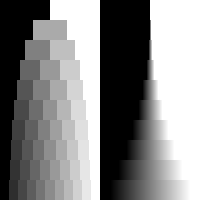}
			}
			\\\hline
		\end{tabular}
	\caption{Phantom images used for tests. Image 00 (left) and Image 01 (right).}\label{fig:phantomimg}
\end{figure}
\begin{figure}[htbp]
		\centering
		\normalsize
		\begin{tabular}{|c|c|}
			\hline
			\textbf{ID} & \textbf{Test images} (270x480)
			\\\hline
			10 and 11 & \parbox[c][3cm][c]{0.65\textwidth}{
			\includegraphics[width=5cm]{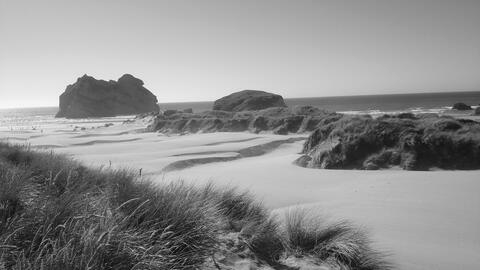}
			\includegraphics[width=5cm]{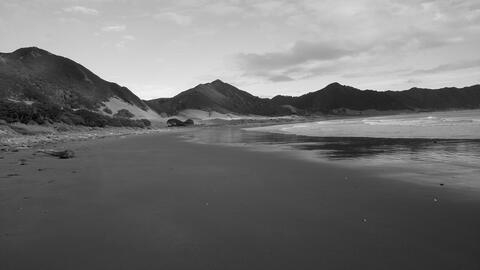}
			}
			\\\hline
			12 and 13 & \parbox[c][3cm][c]{0.65\textwidth}{
			\includegraphics[width=5cm]{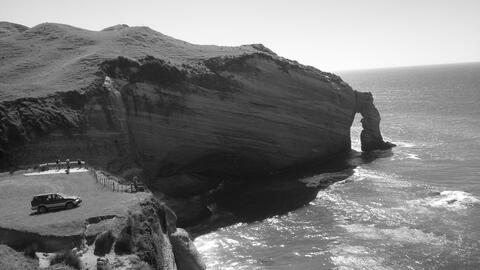}
			\includegraphics[width=5cm]{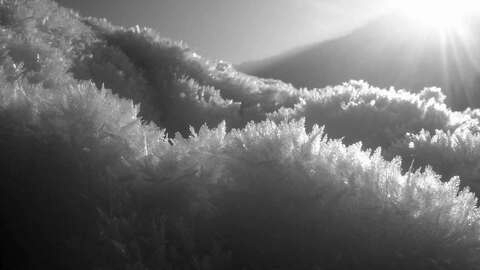}
			}\\\hline
			14 and 15 & \parbox[c][3cm][c]{0.65\textwidth}{
			\includegraphics[width=5cm]{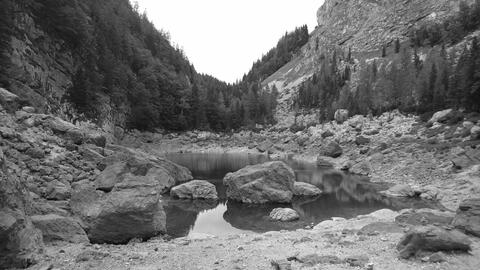}
			\includegraphics[width=5cm]{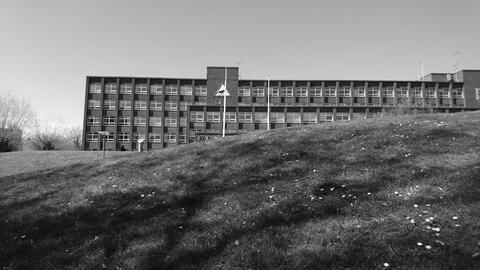}
			}\\\hline
			16 and 17 & \parbox[c][3cm][c]{0.65\textwidth}{
			\includegraphics[width=5cm]{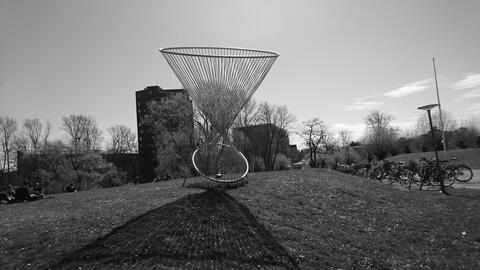}
			\includegraphics[width=5cm]{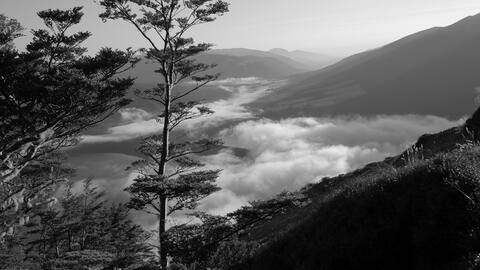}
			}\\\hline
			18 and 19 & \parbox[c][3cm][c]{0.65\textwidth}{
			\includegraphics[width=5cm]{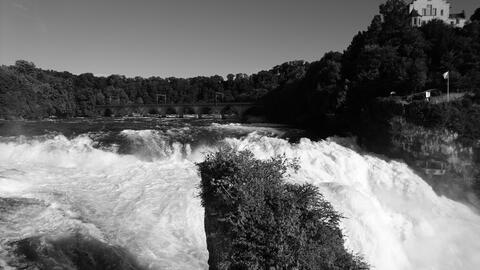}
			\includegraphics[width=5cm]{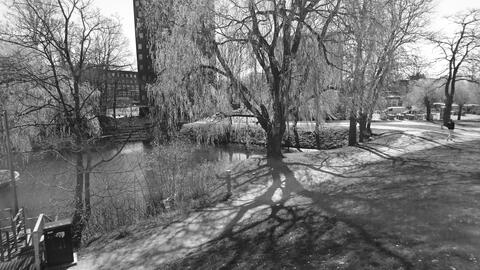}
			}\\\hline
		\end{tabular}
	\caption{Real world images used for tests, which are ordered left to right according to the IDs listed in the first column.}\label{fig:realworldimg}
\end{figure}
In all experiments, we chose $h=1.0$. Let $\gt^h\in \mcX(\Omega^h,1)$ be the discrete ground truth image, $d^h\in \mcX(\Omega^h,1)$ its reconstruction and $\eta^h := d^h-\gt^h$ the reconstruction error. 
The peak signal-to-noise ratio (PSNR) is defined as:
\begin{align*}
\PSNR(d^h,\gt^h)
&=-10\cdot\log_{10}\big(\MSE(d^h,\gt^h)\big), \quad
\text{with }\MSE=\frac{1}{N_1 N_2}\sum_{i,j=1}^{N_2,N_1}\left( d^h_{i,j}-(\gt^h)_{i,j}\right)^2. 
\end{align*}
Let $\overline{d^h}$ and $\overline{\gt^h}$ denote the mean pixel values of $d^h$ and $\gt^h$, respectively.
The mean structural similarity index (MSSIM) is computed by:
\begin{align*}
&\MSSIM(d^h,\gt^h)=\frac{(2\overline{d^h}\overline{\gt^h}+c_1)(2\sigma_{d^h\gt^h}+c_2)}{(\overline{d^h}^2+\overline{\gt^h}^2+c_1)(\sigma_{d^h}^2+\sigma_{\gt^h}^2+c_2)},
\end{align*}
with the following variances and covariance:
\begin{align*}
&\sigma_{d^h}^2=\frac{1}{N_1 N_2}\sum_{i,j=1}^{N_2,N_1}(d^h_{i,j}-\overline{d^h})^2,~~~
\sigma_{\gt^h}^2=\frac{1}{N_1 N_2}\sum_{i,j=1}^{N_2,N_1}\left((\gt^h)_{i,j}-\overline{\gt^h}\right)^2,\\
&\sigma_{d^h\gt^h}=\frac{1}{N_1 N_2}\sum_{i,j=1}^{N_2,N_1}(d^h_{i,j}-\overline{d^h})\left((\gt^h)_{i,j}-\overline{\gt^h}\right).
\end{align*}
The constants are set as $c_1=(k_1 R)^2$, $c_2=(k_2 R)^2$ with $k_1=0.01$, $k_2=0.03$, and $R$ denoting the maximal possible pixel value. Since we only consider grayscale images with intensity range $[0,1]$, this yields $R=1$.
For each experimental setting, Gaussian noise is added to the ground truth image and the resulting noisy image is denoised using the TV-Stokes model under varying parameter configurations. Four noise levels were considered, corresponding to variances $\sigma^2\in\{0.0001,0.0025,0.01,0.09\}$. 
For each noise level and each of the 12 test images, we solve the discrete TFS problem \eqref{eq:tfsdualdis} using \cref{alg:chambTfs}. The regularization parameter $\delta$ is varied over the set 
\begin{align*}
\delta\in\{0.001, 0.002, 0.005, 0.01, 0.02, 0.04, 0.08, 0.15, 0.3, 0.6, 1.2, 2.5, 5.0, 10.0\}.
\end{align*}
\Cref{alg:chambTfs} is performed with step size $t=0.125$ and $max\_it~=10^6$.

As stop criteria, we used
\begin{align*}
\left|\mcD^h_{\mathrm{TFS}}(\vec{p}^{h,n})^{1/2}-\mcD^h_{\mathrm{TFS}}(\vec{p}^{h,n+1})^{1/2}\right| < (2|\tilde{\Omega}^h|)^{1/2}~T_{\mathcal{D}^h},
\end{align*}
with $T_{\mathcal{D}^h}=10^{-7}$.
We compute the ground truth tangent field $\vec\tau^h_{gt}=(\tau^h_{gt,1},\tau^h_{gt,2})$ from the ground truth image and compare it component-wise to the estimated tangent field $(\tau^h_1, \tau^h_2)$ obtained from solving \eqref{eq:tfsdualdis}. For each component, we evaluate the reconstruction quality using both PSNR and MSSIM.
To select the tangent field $\vec{\tau}^h$ corresponding to the parameter $\delta$ that yields the best overall performance for each noise level and test image, we define a composite performance measure balancing both metrics:
\begin{align*}
\Perf_{\vec\tau^h} =  \sum_{i=1}^2\frac{1}{2}\PSNR(\tau^h_i,\tau^h_{gt,i})
~+~20.0\cdot \frac{1}{2}\MSSIM(\tau^h_i,\tau^h_{gt,i}).
\end{align*}
Here, larger values of $\Perf_{\vec\tau^h}$ indicate higher overall reconstruction quality. Hence, the $\vec{\tau}^h$ maximizing $\Perf_{\vec\tau^h}$ is selected for use in IRV1 and IRV2.

We perform IRV1 by solving \eqref{eq:ir1dualdis}, where $\vec{\xi}^h$ is defined according to \eqref{eq:defXi2:discrete}, using Chambolle's algorithm. The algorithm is executed for all combinations of the parameters
\begin{align*}
\frac{1}{\mu}:= \alpha\in\left\{\frac{1}{3}\cdot 10^{-1}, 10^{-1},\frac{1}{3}\cdot 10^{0},10^{0},\ldots,10^{3}, \frac{1}{3}\cdot 10^{4}\right\}
\end{align*}
and
\begin{align*}
\epsilon\in\{10^{1}, 10^{0}, 10^{-1}, 10^{-2}, 10^{-3}, 10^{-4}, 10^{-7}, 10^{-10}, 10^{-13}\}.
\end{align*}
The iteration is run with step size $t=0.125$ and $max\_it~=10^6$.
As stop criteria, we used
\begin{align*}
\left|\mcD^h_{\mathrm{IRV1}}(\vec{p}^{h,n})^{1/2}-\mcD^h_{\mathrm{IRV1}}(\vec{p}^{h,n+1})^{1/2}\right| < |\Omega^h|^{1/2}~T_{\mcD^h}
\end{align*}
with $T_{\mcD^h}=10^{-7}$.

For IRV2 we solve
\eqref{eq:ir2dualdis} using the Chambolle iteration for
\begin{align*}
\alpha \in \{10^{-6}, 3\cdot 10^{-6}, 10^{-5}, 3\cdot 10^{-5},\ldots, 3\cdot 10^2, 10^3\}.
\end{align*}
As in the previous experiments, we use a step size of $t=0.125$ and set $max\_it~=10^6$.
Convergence is assessed using the stopping criterion
\begin{align*}
\left|\mcD^h_{\mathrm{IRV2}}(\vec{p}^{h,n})^{1/2}-\mcD^h_{\mathrm{IRV2}}(\vec{p}^{h,n+1})^{1/2}\right| < |\Omega^h|^{1/2}~T_{\mcD^h},
\end{align*}
where the tolerance is set to $T_{\mcD^h}=10^{-7}$.

\subsubsection*{Results}
The best results for IRV1 and IRV2, both per test image and averaged over all test images, are reported in \cref{tab:expsig001} for $\sigma^2=0.0001$, \cref{tab:expsig005} for $\sigma^2=0.0025$, \cref{tab:expsig01} for $\sigma^2=0.01$ and \cref{tab:expsig03} for $\sigma^2=0.09$. 

Each table indicates the value of $\delta$ for which TFS achieved the best performance (and was therefore used in the reconstruction), as well as the parameters for which IRV1 and IRV2 performed best with respect to PSNR and MSSIM. The final row in each table presents the average over all 12 test images, providing a global comparison across methods and parameter configurations.

\begin{table}[htbp]
	\centering
    \caption{Performance comparison of the image reconstruction variants for noise level $\sigma^2=0.0001$.} \label{tab:expsig001}
    {\footnotesize{
	\begin{tabular}{|c|c|c|c|c|c|}
		\hline
		\textbf{Image}& \textbf{TFS Parameter}  & \textbf{Metric}& \textbf{IRV} & \textbf{Best IRV Parameters}  & \textbf{Metric Value} \\ \hline 
		\multirow{4}{*}{00} &\multirow{4}{*}{$\delta=0.01$} & \multirow{2}{*}{PSNR} & 1 &$(\mu,\epsilon)=(0.01,10.0)$&44.8481\\\cmidrule{4-6}
		&&& \textbf{2} & $\alpha=100.0$ &\textbf{45.0686}\\\cmidrule{3-6}
		&& \multirow{2}{*}{MSSIM} & \textbf{1} & $(\mu,\epsilon)=(0.01,1.0\cdot 10^{-7})$&\textbf{0.998253}\\\cmidrule{4-6}
		&&& 2 & $\alpha=100.0$&0.998192\\ \hline 
		\multirow{4}{*}{01} &\multirow{4}{*}{$\delta=0.01$} & \multirow{2}{*}{PSNR} & \textbf{1} &$(\mu,\epsilon)=(0.01,10.0)$&\textbf{46.5367}\\\cmidrule{4-6}
		&&& 2 & $\alpha=100.0$ &45.617\\\cmidrule{3-6}
		&& \multirow{2}{*}{MSSIM} & 1 & $(\mu,\epsilon)=(0.01,0.1)$&0.988337\\\cmidrule{4-6}
		&&& \textbf{2} & $\alpha=0.003$&\textbf{0.988531}\\ \hline 
		\multirow{4}{*}{10} &\multirow{4}{*}{$\delta=0.005$} & \multirow{2}{*}{PSNR} & 1 &$(\mu,\epsilon)=(0.003,10.0)$&41.9494\\\cmidrule{4-6}
		&&& \textbf{2} & $\alpha=100.0$ &\textbf{42.1273}\\\cmidrule{3-6}
		&& \multirow{2}{*}{MSSIM} & 1 & $(\mu,\epsilon)=(0.003,0.0001)$&0.991629\\\cmidrule{4-6}
		&&& \textbf{2} & $\alpha=30.0$&\textbf{0.992383}\\ \hline 
		\multirow{4}{*}{11} &\multirow{4}{*}{$\delta=0.01$} & \multirow{2}{*}{PSNR} & 1 &$(\mu,\epsilon)=(0.003,10.0)$&43.0851\\\cmidrule{4-6}
		&&& \textbf{2} & $\alpha=0.003$ &\textbf{44.0418}\\\cmidrule{3-6}
		&& \multirow{2}{*}{MSSIM} & 1 & $(\mu,\epsilon)=(0.01,10.0)$&0.992007\\\cmidrule{4-6}
		&&& \textbf{2} & $\alpha=1.0\cdot 10^{-6}$&\textbf{0.994176}\\ \hline 
		\multirow{4}{*}{12} &\multirow{4}{*}{$\delta=0.005$} & \multirow{2}{*}{PSNR} & 1 &$(\mu,\epsilon)=(0.003,10.0)$&41.6801\\\cmidrule{4-6}
		&&& \textbf{2} & $\alpha=0.01$ &\textbf{41.8891}\\\cmidrule{3-6}
		&& \multirow{2}{*}{MSSIM} & 1 & $(\mu,\epsilon)=(0.003,0.01)$&0.99044\\\cmidrule{4-6}
		&&& \textbf{2} & $\alpha=0.03$&\textbf{0.991273}\\ \hline 
		\multirow{4}{*}{13} &\multirow{4}{*}{$\delta=0.005$} & \multirow{2}{*}{PSNR} & 1 &$(\mu,\epsilon)=(0.003,0.0001)$&42.9194\\\cmidrule{4-6}
		&&& \textbf{2} & $\alpha=0.03$ &\textbf{43.2712}\\\cmidrule{3-6}
		&& \multirow{2}{*}{MSSIM} & 1 & $(\mu,\epsilon)=(0.003,0.0001)$&0.991477\\\cmidrule{4-6}
		&&& \textbf{2} & $\alpha=0.03$&\textbf{0.992131}\\ \hline 
		\multirow{4}{*}{14} &\multirow{4}{*}{$\delta=0.002$} & \multirow{2}{*}{PSNR} & 1 &$(\mu,\epsilon)=(0.001,1.0\cdot 10^{-7})$&40.5863\\\cmidrule{4-6}
		&&& \textbf{2} & $\alpha=0.01$ &\textbf{40.6096}\\\cmidrule{3-6}
		&& \multirow{2}{*}{MSSIM} & \textbf{1} & $(\mu,\epsilon)=(0.003,1.0)$&\textbf{0.992162}\\\cmidrule{4-6}
		&&& 2 & $\alpha=1.0\cdot 10^{-6}$&0.992089\\ \hline 
		\multirow{4}{*}{15} &\multirow{4}{*}{$\delta=0.005$} & \multirow{2}{*}{PSNR} & 1 &$(\mu,\epsilon)=(0.003,10.0)$&41.1716\\\cmidrule{4-6}
		&&& \textbf{2} & $\alpha=0.003$ &\textbf{41.2664}\\\cmidrule{3-6}
		&& \multirow{2}{*}{MSSIM} & 1 & $(\mu,\epsilon)=(0.003,10.0)$&0.991733\\\cmidrule{4-6}
		&&& \textbf{2} & $\alpha=0.03$&\textbf{0.992049}\\ \hline 
		\multirow{4}{*}{16} &\multirow{4}{*}{$\delta=0.005$} & \multirow{2}{*}{PSNR} & 1 &$(\mu,\epsilon)=(0.003,10.0)$&41.2678\\\cmidrule{4-6}
		&&& \textbf{2} & $\alpha=0.01$ &\textbf{41.3697}\\\cmidrule{3-6}
		&& \multirow{2}{*}{MSSIM} & 1 & $(\mu,\epsilon)=(0.003,10.0)$&0.990896\\\cmidrule{4-6}
		&&& \textbf{2} & $\alpha=0.03$&\textbf{0.991308}\\ \hline 
		\multirow{4}{*}{17} &\multirow{4}{*}{$\delta=0.005$} & \multirow{2}{*}{PSNR} & 1 &$(\mu,\epsilon)=(0.003,10.0)$&41.3502\\\cmidrule{4-6}
		&&& \textbf{2} & $\alpha=0.003$ &\textbf{41.4572}\\\cmidrule{3-6}
		&& \multirow{2}{*}{MSSIM} & 1 & $(\mu,\epsilon)=(0.003,1.0)$&0.991676\\\cmidrule{4-6}
		&&& \textbf{2} & $\alpha=0.01$&\textbf{0.992344}\\ \hline 
		\multirow{4}{*}{18} &\multirow{4}{*}{$\delta=0.005$} & \multirow{2}{*}{PSNR} & 1 &$(\mu,\epsilon)=(0.003,10.0)$&41.3533\\\cmidrule{4-6}
		&&& \textbf{2} & $\alpha=0.003$ &\textbf{41.5118}\\\cmidrule{3-6}
		&& \multirow{2}{*}{MSSIM} & 1 & $(\mu,\epsilon)=(0.003,0.1)$&0.99097\\\cmidrule{4-6}
		&&& \textbf{2} & $\alpha=0.01$&\textbf{0.991633}\\ \hline 
		\multirow{4}{*}{19} &\multirow{4}{*}{$\delta=0.002$} & \multirow{2}{*}{PSNR} & 1 &$(\mu,\epsilon)=(0.001,0.1)$&40.2967\\\cmidrule{4-6}
		&&& \textbf{2} & $\alpha=300.0$ &\textbf{40.3515}\\\cmidrule{3-6}
		&& \multirow{2}{*}{MSSIM} & 1 & $(\mu,\epsilon)=(0.001,0.0001)$&0.993443\\\cmidrule{4-6}
		&&& \textbf{2} & $\alpha=1.0\cdot 10^{-6}$&\textbf{0.993588}\\ \hline 
		\multirow{4}{*}{Average} &\multirow{4}{*}{different} & \multirow{2}{*}{PSNR} & 1 &different&42.2537\\\cmidrule{4-6}
		&&& \textbf{2} & different &\textbf{42.3817}\\\cmidrule{3-6}
		&& \multirow{2}{*}{MSSIM} & 1 & different&0.991919\\\cmidrule{4-6}
		&&& \textbf{2} & different&\textbf{0.992475}\\ \hline 
	\end{tabular}}}
	\end{table}
\begin{table}[htbp]
	\centering
    \caption{Performance comparison of the image reconstruction variants for noise level $\sigma^2=0.0025$.}
	\label{tab:expsig005}
    {\footnotesize{
	\begin{tabular}{|c|c|c|c|c|c|}
		\hline
		\textbf{Image}& \textbf{TFS Parameter}  & \textbf{Metric}& \textbf{IRV} & \textbf{Best IRV Parameters}  & \textbf{Metric Value} \\ \hline 
		\multirow{4}{*}{00} &\multirow{4}{*}{$\delta=0.04$} & \multirow{2}{*}{PSNR} & \textbf{1} &$(\mu,\epsilon)=(0.03,10.0)$&\textbf{32.4891}\\\cmidrule{4-6}
		&&& 2 & $\alpha=30.0$ &32.1852\\\cmidrule{3-6}
		&& \multirow{2}{*}{MSSIM} & 1 & $(\mu,\epsilon)=(0.1,10.0)$&0.9624\\\cmidrule{4-6}
		&&& \textbf{2} & $\alpha=30.0$&\textbf{0.96353}\\ \hline 
		\multirow{4}{*}{01} &\multirow{4}{*}{$\delta=0.04$} & \multirow{2}{*}{PSNR} & \textbf{1} &$(\mu,\epsilon)=(0.03,10.0)$&\textbf{33.5982}\\\cmidrule{4-6}
		&&& 2 & $\alpha=30.0$ &33.3289\\\cmidrule{3-6}
		&& \multirow{2}{*}{MSSIM} & \textbf{1} & $(\mu,\epsilon)=(0.03,0.01)$&\textbf{0.86145}\\\cmidrule{4-6}
		&&& 2 & $\alpha=0.003$&0.859822\\ \hline 
		\multirow{4}{*}{10} &\multirow{4}{*}{$\delta=0.04$} & \multirow{2}{*}{PSNR} & 1 &$(\mu,\epsilon)=(0.03,1.0)$&31.6541\\\cmidrule{4-6}
		&&& \textbf{2} & $\alpha=30.0$ &\textbf{31.8468}\\\cmidrule{3-6}
		&& \multirow{2}{*}{MSSIM} & 1 & $(\mu,\epsilon)=(0.03,0.1)$&0.933091\\\cmidrule{4-6}
		&&& \textbf{2} & $\alpha=10.0$&\textbf{0.936612}\\ \hline 
		\multirow{4}{*}{11} &\multirow{4}{*}{$\delta=0.04$} & \multirow{2}{*}{PSNR} & 1 &$(\mu,\epsilon)=(0.03,0.01)$&34.5462\\\cmidrule{4-6}
		&&& \textbf{2} & $\alpha=10.0$ &\textbf{34.6197}\\\cmidrule{3-6}
		&& \multirow{2}{*}{MSSIM} & \textbf{1} & $(\mu,\epsilon)=(0.1,0.0001)$&\textbf{0.960308}\\\cmidrule{4-6}
		&&& 2 & $\alpha=3.0$&0.958834\\ \hline 
		\multirow{4}{*}{12} &\multirow{4}{*}{$\delta=0.04$} & \multirow{2}{*}{PSNR} & 1 &$(\mu,\epsilon)=(0.03,1.0)$&31.7329\\\cmidrule{4-6}
		&&& \textbf{2} & $\alpha=30.0$ &\textbf{31.7912}\\\cmidrule{3-6}
		&& \multirow{2}{*}{MSSIM} & 1 & $(\mu,\epsilon)=(0.03,0.1)$&0.931671\\\cmidrule{4-6}
		&&& \textbf{2} & $\alpha=30.0$&\textbf{0.932952}\\ \hline 
		\multirow{4}{*}{13} &\multirow{4}{*}{$\delta=0.04$} & \multirow{2}{*}{PSNR} & 1 &$(\mu,\epsilon)=(0.03,0.01)$&33.9986\\\cmidrule{4-6}
		&&& \textbf{2} & $\alpha=30.0$ &\textbf{34.0041}\\\cmidrule{3-6}
		&& \multirow{2}{*}{MSSIM} & 1 & $(\mu,\epsilon)=(0.1,0.0001)$&0.953444\\\cmidrule{4-6}
		&&& \textbf{2} & $\alpha=0.03$&\textbf{0.954953}\\ \hline 
		\multirow{4}{*}{14} &\multirow{4}{*}{$\delta=0.04$} & \multirow{2}{*}{PSNR} & 1 &$(\mu,\epsilon)=(0.03,1.0)$&29.2559\\\cmidrule{4-6}
		&&& \textbf{2} & $\alpha=0.0003$ &\textbf{29.6477}\\\cmidrule{3-6}
		&& \multirow{2}{*}{MSSIM} & 1 & $(\mu,\epsilon)=(0.03,0.1)$&0.908757\\\cmidrule{4-6}
		&&& \textbf{2} & $\alpha=0.0003$&\textbf{0.917949}\\ \hline 
		\multirow{4}{*}{15} &\multirow{4}{*}{$\delta=0.04$} & \multirow{2}{*}{PSNR} & 1 &$(\mu,\epsilon)=(0.03,10.0)$&30.2435\\\cmidrule{4-6}
		&&& \textbf{2} & $\alpha=0.0003$ &\textbf{30.378}\\\cmidrule{3-6}
		&& \multirow{2}{*}{MSSIM} & 1 & $(\mu,\epsilon)=(0.03,10.0)$&0.919035\\\cmidrule{4-6}
		&&& \textbf{2} & $\alpha=30.0$&\textbf{0.920187}\\ \hline 
		\multirow{4}{*}{16} &\multirow{4}{*}{$\delta=0.04$} & \multirow{2}{*}{PSNR} & 1 &$(\mu,\epsilon)=(0.03,10.0)$&30.3721\\\cmidrule{4-6}
		&&& \textbf{2} & $\alpha=30.0$ &\textbf{30.5222}\\\cmidrule{3-6}
		&& \multirow{2}{*}{MSSIM} & 1 & $(\mu,\epsilon)=(0.03,10.0)$&0.912018\\\cmidrule{4-6}
		&&& \textbf{2} & $\alpha=30.0$&\textbf{0.913924}\\ \hline 
		\multirow{4}{*}{17} &\multirow{4}{*}{$\delta=0.04$} & \multirow{2}{*}{PSNR} & 1 &$(\mu,\epsilon)=(0.03,10.0)$&30.443\\\cmidrule{4-6}
		&&& \textbf{2} & $\alpha=0.0003$ &\textbf{30.5559}\\\cmidrule{3-6}
		&& \multirow{2}{*}{MSSIM} & 1 & $(\mu,\epsilon)=(0.03,0.1)$&0.919369\\\cmidrule{4-6}
		&&& \textbf{2} & $\alpha=0.003$&\textbf{0.920215}\\ \hline 
		\multirow{4}{*}{18} &\multirow{4}{*}{$\delta=0.04$} & \multirow{2}{*}{PSNR} & 1 &$(\mu,\epsilon)=(0.03,10.0)$&30.7996\\\cmidrule{4-6}
		&&& \textbf{2} & $\alpha=0.0003$ &\textbf{30.9465}\\\cmidrule{3-6}
		&& \multirow{2}{*}{MSSIM} & 1 & $(\mu,\epsilon)=(0.03,0.01)$&0.920268\\\cmidrule{4-6}
		&&& \textbf{2} & $\alpha=0.001$&\textbf{0.921744}\\ \hline 
		\multirow{4}{*}{19} &\multirow{4}{*}{$\delta=0.02$} & \multirow{2}{*}{PSNR} & 1 &$(\mu,\epsilon)=(0.01,0.01)$&27.9548\\\cmidrule{4-6}
		&&& \textbf{2} & $\alpha=0.01$ &\textbf{28.2612}\\\cmidrule{3-6}
		&& \multirow{2}{*}{MSSIM} & 1 & $(\mu,\epsilon)=(0.01,0.001)$&0.904475\\\cmidrule{4-6}
		&&& \textbf{2} & $\alpha=0.01$&\textbf{0.910227}\\ \hline 
		\multirow{4}{*}{Average} &\multirow{4}{*}{different} & \multirow{2}{*}{PSNR} & 1 &different&31.424\\\cmidrule{4-6}
		&&& \textbf{2} & different &\textbf{31.5073}\\\cmidrule{3-6}
		&& \multirow{2}{*}{MSSIM} & 1 & different&0.923857\\\cmidrule{4-6}
		&&& \textbf{2} & different&\textbf{0.925913}\\ \hline 
	\end{tabular}}}
	
\end{table}
\begin{table}[htbp]
	\centering
    \caption{Performance comparison of the image reconstruction variants for noise level $\sigma^2=0.01$.}
	\label{tab:expsig01}
    {\footnotesize{
	\begin{tabular}{|c|c|c|c|c|c|}
		\hline
		\textbf{Image}& \textbf{TFS Parameter}  & \textbf{Metric}& \textbf{IRV} & \textbf{Best IRV Parameters}  & \textbf{Metric Value} \\ \hline 
		\multirow{4}{*}{00} &\multirow{4}{*}{$\delta=0.08$} & \multirow{2}{*}{PSNR} & 1 &$(\mu,\epsilon)=(0.1,10.0)$&25.3134\\\cmidrule{4-6}
		&&& \textbf{2} & $\alpha=10.0$ &\textbf{25.5897}\\\cmidrule{3-6}
		&& \multirow{2}{*}{MSSIM} & \textbf{1} & $(\mu,\epsilon)=(0.1,10.0)$&\textbf{0.946662}\\\cmidrule{4-6}
		&&& 2 & $\alpha=10.0$&0.911459\\ \hline 
		\multirow{4}{*}{01} &\multirow{4}{*}{$\delta=0.08$} & \multirow{2}{*}{PSNR} & \textbf{1} &$(\mu,\epsilon)=(0.1,10.0)$&\textbf{28.212}\\\cmidrule{4-6}
		&&& 2 & $\alpha=10.0$ &27.8378\\\cmidrule{3-6}
		&& \multirow{2}{*}{MSSIM} & \textbf{1} & $(\mu,\epsilon)=(0.1,0.1)$&\textbf{0.77214}\\\cmidrule{4-6}
		&&& 2 & $\alpha=10.0$&0.757025\\ \hline 
		\multirow{4}{*}{10} &\multirow{4}{*}{$\delta=0.08$} & \multirow{2}{*}{PSNR} & \textbf{1} &$(\mu,\epsilon)=(0.1,0.01)$&\textbf{28.5348}\\\cmidrule{4-6}
		&&& 2 & $\alpha=10.0$ &28.4239\\\cmidrule{3-6}
		&& \multirow{2}{*}{MSSIM} & \textbf{1} & $(\mu,\epsilon)=(0.1,0.001)$&\textbf{0.880882}\\\cmidrule{4-6}
		&&& 2 & $\alpha=10.0$&0.867789\\ \hline 
		\multirow{4}{*}{11} &\multirow{4}{*}{$\delta=0.15$} & \multirow{2}{*}{PSNR} & \textbf{1} &$(\mu,\epsilon)=(0.1,0.0001)$&\textbf{32.5011}\\\cmidrule{4-6}
		&&& 2 & $\alpha=10.0$ &32.131\\\cmidrule{3-6}
		&& \multirow{2}{*}{MSSIM} & \textbf{1} & $(\mu,\epsilon)=(0.1,0.0001)$&\textbf{0.945975}\\\cmidrule{4-6}
		&&& 2 & $\alpha=10.0$&0.943373\\ \hline 
		\multirow{4}{*}{12} &\multirow{4}{*}{$\delta=0.08$} & \multirow{2}{*}{PSNR} & \textbf{1} &$(\mu,\epsilon)=(0.1,0.001)$&\textbf{28.4023}\\\cmidrule{4-6}
		&&& 2 & $\alpha=10.0$ &28.3683\\\cmidrule{3-6}
		&& \multirow{2}{*}{MSSIM} & \textbf{1} & $(\mu,\epsilon)=(0.1,0.001)$&\textbf{0.882216}\\\cmidrule{4-6}
		&&& 2 & $\alpha=10.0$&0.876387\\ \hline 
		\multirow{4}{*}{13} &\multirow{4}{*}{$\delta=0.08$} & \multirow{2}{*}{PSNR} & \textbf{1} &$(\mu,\epsilon)=(0.1,0.01)$&\textbf{31.1257}\\\cmidrule{4-6}
		&&& 2 & $\alpha=0.01$ &30.2484\\\cmidrule{3-6}
		&& \multirow{2}{*}{MSSIM} & \textbf{1} & $(\mu,\epsilon)=(0.1,0.01)$&\textbf{0.93104}\\\cmidrule{4-6}
		&&& 2 & $\alpha=0.03$&0.903424\\ \hline 
		\multirow{4}{*}{14} &\multirow{4}{*}{$\delta=0.08$} & \multirow{2}{*}{PSNR} & 1 &$(\mu,\epsilon)=(0.1,0.001)$&25.6394\\\cmidrule{4-6}
		&&& \textbf{2} & $\alpha=0.001$ &\textbf{26.1311}\\\cmidrule{3-6}
		&& \multirow{2}{*}{MSSIM} & 1 & $(\mu,\epsilon)=(0.1,0.0001)$&0.809588\\\cmidrule{4-6}
		&&& \textbf{2} & $\alpha=0.0003$&\textbf{0.831324}\\ \hline 
		\multirow{4}{*}{15} &\multirow{4}{*}{$\delta=0.08$} & \multirow{2}{*}{PSNR} & 1 &$(\mu,\epsilon)=(0.1,0.01)$&26.3656\\\cmidrule{4-6}
		&&& \textbf{2} & $\alpha=0.003$ &\textbf{26.9656}\\\cmidrule{3-6}
		&& \multirow{2}{*}{MSSIM} & 1 & $(\mu,\epsilon)=(0.1,0.001)$&0.841147\\\cmidrule{4-6}
		&&& \textbf{2} & $\alpha=0.01$&\textbf{0.843202}\\ \hline 
		\multirow{4}{*}{16} &\multirow{4}{*}{$\delta=0.08$} & \multirow{2}{*}{PSNR} & 1 &$(\mu,\epsilon)=(0.1,0.001)$&26.5624\\\cmidrule{4-6}
		&&& \textbf{2} & $\alpha=10.0$ &\textbf{27.183}\\\cmidrule{3-6}
		&& \multirow{2}{*}{MSSIM} & 1 & $(\mu,\epsilon)=(0.1,0.001)$&0.825525\\\cmidrule{4-6}
		&&& \textbf{2} & $\alpha=10.0$&\textbf{0.830345}\\ \hline 
		\multirow{4}{*}{17} &\multirow{4}{*}{$\delta=0.08$} & \multirow{2}{*}{PSNR} & 1 &$(\mu,\epsilon)=(0.1,0.001)$&26.3882\\\cmidrule{4-6}
		&&& \textbf{2} & $\alpha=0.001$ &\textbf{26.7634}\\\cmidrule{3-6}
		&& \multirow{2}{*}{MSSIM} & \textbf{1} & $(\mu,\epsilon)=(0.1,0.001)$&\textbf{0.82184}\\\cmidrule{4-6}
		&&& 2 & $\alpha=0.003$&0.82154\\ \hline 
		\multirow{4}{*}{18} &\multirow{4}{*}{$\delta=0.08$} & \multirow{2}{*}{PSNR} & 1 &$(\mu,\epsilon)=(0.1,0.001)$&26.9971\\\cmidrule{4-6}
		&&& \textbf{2} & $\alpha=0.001$ &\textbf{27.3432}\\\cmidrule{3-6}
		&& \multirow{2}{*}{MSSIM} & 1 & $(\mu,\epsilon)=(0.1,0.001)$&0.834812\\\cmidrule{4-6}
		&&& \textbf{2} & $\alpha=0.003$&\textbf{0.838988}\\ \hline 
		\multirow{4}{*}{19} &\multirow{4}{*}{$\delta=0.08$} & \multirow{2}{*}{PSNR} & 1 &$(\mu,\epsilon)=(0.03,0.01)$&23.7134\\\cmidrule{4-6}
		&&& \textbf{2} & $\alpha=0.0001$ &\textbf{24.7005}\\\cmidrule{3-6}
		&& \multirow{2}{*}{MSSIM} & 1 & $(\mu,\epsilon)=(0.03,0.01)$&0.780341\\\cmidrule{4-6}
		&&& \textbf{2} & $\alpha=0.0001$&\textbf{0.817587}\\ \hline 
		\multirow{4}{*}{Average} &\multirow{4}{*}{different} & \multirow{2}{*}{PSNR} & 1 &different&27.4796\\\cmidrule{4-6}
		&&& \textbf{2} & different &\textbf{27.6405}\\\cmidrule{3-6}
		&& \multirow{2}{*}{MSSIM} & \textbf{1} & different&\textbf{0.856014}\\\cmidrule{4-6}
		&&& 2 & different&0.853537\\ \hline 
	\end{tabular}}}
	
\end{table}
\begin{table}[htbp]
\centering
\caption{Performance comparison of the image reconstruction variants for noise level $\sigma^2=0.09$.}
	\label{tab:expsig03}
{\footnotesize{	
	\begin{tabular}{|c|c|c|c|c|c|}
		\hline
		\textbf{Image}& \textbf{TFS Parameter}  & \textbf{Metric}& \textbf{IRV} & \textbf{Best IRV Parameters}  & \textbf{Metric Value} \\ \hline 
		\multirow{4}{*}{00} &\multirow{4}{*}{$\delta=0.3$} & \multirow{2}{*}{PSNR} & \textbf{1} &$(\mu,\epsilon)=(0.3,1.0\cdot 10^{-7})$&\textbf{17.9774}\\\cmidrule{4-6}
		&&& 2 & $\alpha=10.0$ &17.2941\\\cmidrule{3-6}
		&& \multirow{2}{*}{MSSIM} & \textbf{1} & $(\mu,\epsilon)=(0.3,0.0001)$&\textbf{0.820006}\\\cmidrule{4-6}
		&&& 2 & $\alpha=3.0$&0.782291\\ \hline 
		\multirow{4}{*}{01} &\multirow{4}{*}{$\delta=0.3$} & \multirow{2}{*}{PSNR} & \textbf{1} &$(\mu,\epsilon)=(0.3,1.0\cdot 10^{-7})$&\textbf{19.8108}\\\cmidrule{4-6}
		&&& 2 & $\alpha=3.0$ &19.2286\\\cmidrule{3-6}
		&& \multirow{2}{*}{MSSIM} & \textbf{1} & $(\mu,\epsilon)=(0.3,0.0001)$&\textbf{0.673876}\\\cmidrule{4-6}
		&&& 2 & $\alpha=3.0$&0.671168\\ \hline 
		\multirow{4}{*}{10} &\multirow{4}{*}{$\delta=10.0$} & \multirow{2}{*}{PSNR} & \textbf{1} &$(\mu,\epsilon)=(0.3,0.0001)$&\textbf{24.1472}\\\cmidrule{4-6}
		&&& 2 & $\alpha=0.001$ &23.8306\\\cmidrule{3-6}
		&& \multirow{2}{*}{MSSIM} & \textbf{1} & $(\mu,\epsilon)=(0.3,0.0001)$&\textbf{0.791067}\\\cmidrule{4-6}
		&&& 2 & $\alpha=3.0$&0.785287\\ \hline 
		\multirow{4}{*}{11} &\multirow{4}{*}{$\delta=0.3$} & \multirow{2}{*}{PSNR} & \textbf{1} &$(\mu,\epsilon)=(0.3,1.0\cdot 10^{-7})$&\textbf{25.706}\\\cmidrule{4-6}
		&&& 2 & $\alpha=3.0$ &25.2301\\\cmidrule{3-6}
		&& \multirow{2}{*}{MSSIM} & \textbf{1} & $(\mu,\epsilon)=(0.3,0.0001)$&\textbf{0.903165}\\\cmidrule{4-6}
		&&& 2 & $\alpha=3.0$&0.900032\\ \hline 
		\multirow{4}{*}{12} &\multirow{4}{*}{$\delta=10.0$} & \multirow{2}{*}{PSNR} & \textbf{1} &$(\mu,\epsilon)=(0.3,1.0\cdot 10^{-7})$&\textbf{22.4471}\\\cmidrule{4-6}
		&&& 2 & $\alpha=0.001$ &22.1343\\\cmidrule{3-6}
		&& \multirow{2}{*}{MSSIM} & \textbf{1} & $(\mu,\epsilon)=(0.3,0.0001)$&\textbf{0.776834}\\\cmidrule{4-6}
		&&& 2 & $\alpha=3.0$&0.764316\\ \hline 
		\multirow{4}{*}{13} &\multirow{4}{*}{$\delta=10.0$} & \multirow{2}{*}{PSNR} & \textbf{1} &$(\mu,\epsilon)=(0.3,0.0001)$&\textbf{24.8922}\\\cmidrule{4-6}
		&&& 2 & $\alpha=0.003$ &24.7564\\\cmidrule{3-6}
		&& \multirow{2}{*}{MSSIM} & \textbf{1} & $(\mu,\epsilon)=(0.3,0.0001)$&\textbf{0.868593}\\\cmidrule{4-6}
		&&& 2 & $\alpha=3.0$&0.867148\\ \hline 
		\multirow{4}{*}{14} &\multirow{4}{*}{$\delta=10.0$} & \multirow{2}{*}{PSNR} & 1 &$(\mu,\epsilon)=(0.3,1.0\cdot 10^{-7})$&21.1228\\\cmidrule{4-6}
		&&& \textbf{2} & $\alpha=0.0003$ &\textbf{21.3355}\\\cmidrule{3-6}
		&& \multirow{2}{*}{MSSIM} & 1 & $(\mu,\epsilon)=(0.3,1.0\cdot 10^{-7})$&0.608897\\\cmidrule{4-6}
		&&& \textbf{2} & $\alpha=0.0003$&\textbf{0.63372}\\ \hline 
		\multirow{4}{*}{15} &\multirow{4}{*}{$\delta=0.3$} & \multirow{2}{*}{PSNR} & \textbf{1} &$(\mu,\epsilon)=(0.3,1.0\cdot 10^{-7})$&\textbf{21.7513}\\\cmidrule{4-6}
		&&& 2 & $\alpha=0.0003$ &21.6606\\\cmidrule{3-6}
		&& \multirow{2}{*}{MSSIM} & \textbf{1} & $(\mu,\epsilon)=(0.3,1.0\cdot 10^{-7})$&\textbf{0.687155}\\\cmidrule{4-6}
		&&& 2 & $\alpha=0.001$&0.672922\\ \hline 
		\multirow{4}{*}{16} &\multirow{4}{*}{$\delta=10.0$} & \multirow{2}{*}{PSNR} & \textbf{1} &$(\mu,\epsilon)=(0.3,0.0001)$&\textbf{22.4571}\\\cmidrule{4-6}
		&&& 2 & $\alpha=0.001$ &22.1759\\\cmidrule{3-6}
		&& \multirow{2}{*}{MSSIM} & \textbf{1} & $(\mu,\epsilon)=(0.3,0.0001)$&\textbf{0.713072}\\\cmidrule{4-6}
		&&& 2 & $\alpha=3.0$&0.7037\\ \hline 
		\multirow{4}{*}{17} &\multirow{4}{*}{$\delta=10.0$} & \multirow{2}{*}{PSNR} & 1 &$(\mu,\epsilon)=(0.3,1.0\cdot 10^{-7})$&20.9042\\\cmidrule{4-6}
		&&& \textbf{2} & $\alpha=0.0003$ &\textbf{20.9212}\\\cmidrule{3-6}
		&& \multirow{2}{*}{MSSIM} & \textbf{1} & $(\mu,\epsilon)=(0.3,1.0\cdot 10^{-7})$&\textbf{0.632838}\\\cmidrule{4-6}
		&&& 2 & $\alpha=0.001$&0.623967\\ \hline 
		\multirow{4}{*}{18} &\multirow{4}{*}{$\delta=0.3$} & \multirow{2}{*}{PSNR} & \textbf{1} &$(\mu,\epsilon)=(0.3,1.0\cdot 10^{-7})$&\textbf{21.4898}\\\cmidrule{4-6}
		&&& 2 & $\alpha=0.0003$ &21.2022\\\cmidrule{3-6}
		&& \multirow{2}{*}{MSSIM} & \textbf{1} & $(\mu,\epsilon)=(0.3,1.0\cdot 10^{-7})$&\textbf{0.681267}\\\cmidrule{4-6}
		&&& 2 & $\alpha=3.0$&0.669294\\ \hline 
		\multirow{4}{*}{19} &\multirow{4}{*}{$\delta=10.0$} & \multirow{2}{*}{PSNR} & 1 &$(\mu,\epsilon)=(0.3,1.0\cdot 10^{-7})$&19.6342\\\cmidrule{4-6}
		&&& \textbf{2} & $\alpha=0.0001$ &\textbf{20.0933}\\\cmidrule{3-6}
		&& \multirow{2}{*}{MSSIM} & 1 & $(\mu,\epsilon)=(0.3,1.0\cdot 10^{-7})$&0.511152\\\cmidrule{4-6}
		&&& \textbf{2} & $\alpha=0.0001$&\textbf{0.572295}\\ \hline 
		\multirow{4}{*}{Average} &\multirow{4}{*}{different} & \multirow{2}{*}{PSNR} & \textbf{1} &different&\textbf{21.8617}\\\cmidrule{4-6}
		&&& 2 & different &21.6552\\\cmidrule{3-6}
		&& \multirow{2}{*}{MSSIM} & \textbf{1} & different&\textbf{0.722327}\\\cmidrule{4-6}
		&&& 2 & different&0.720512\\ \hline 
	\end{tabular}}}
	
\end{table}
\Cref{tab:expcounts} summarizes how often each variant of the image reconstruction method outperformed the other with respect to the performance measures PSNR and MSSIM, across all noise levels. The results indicate that IRV1 tends to perform slightly better at higher noise levels, whereas IRV2 shows a slight advantage at lower noise levels. An additional row for each metric provides the total count aggregated over all noise levels.
\begin{table}[htbp]
	\centering
    \caption{Overview of how often each image reconstruction variant performed best across the evaluation set.}
	\label{tab:expcounts}
	\begin{tabular}{|c|c|c|c|}
		\hline
		\textbf{Metric} & \textbf{Noise Variance} &   \textbf{IRV1 better}  & \textbf{IRV2 better} \\ \hline 
		\multirow{5}{*}{PSNR}& 0.0001 &1 & 11\\\cmidrule{2-4}
		& 0.0025 &2 & 10\\\cmidrule{2-4}
		& 0.01 &5 & 7\\\cmidrule{2-4}
		& 0.09 &9 & 3\\\cmidrule{2-4}
		& \textbf{Total} & \textbf{17} & \textbf{31}\\\hline 
		\multirow{5}{*}{MSSIM}& 0.0001 &2 & 10\\\cmidrule{2-4}
		& 0.0025 &2 & 10\\\cmidrule{2-4}
		& 0.01 &7 & 5\\\cmidrule{2-4}
		& 0.09 &10 & 2\\\cmidrule{2-4}
		& \textbf{Total} & \textbf{21} & \textbf{27}\\\hline 
	\end{tabular}

\end{table}
Considering \cref{tab:expcounts} together with \cref{tab:expsig001,tab:expsig005,tab:expsig01,tab:expsig03}, we observe that IRV2 performs slightly better under low noise conditions, both in terms of PSNR and MSSIM, while IRV1 shows marginally better results under higher noise levels, given the tested parameter configurations. Overall, the performance of the two approaches is comparable, and depending on the application context, one may offer a slight advantage over the other. From a visual perspective, the difference in quality is minimal at low noise levels (e.g., $\sigma^2=0.0025$). For instance, when comparing the best reconstructions of Image 18, no substantial visual difference is noticeable between the two methods (see \cref{fig:schaff_00025_v1} and \cref{fig:schaff_00025_v2}).

%
%
%

 \begin{figure}[htbp]
 	\centering
 	\begin{subfigure}[t]{0.49\textwidth}
 		\centering
 		\includegraphics[width=\linewidth]{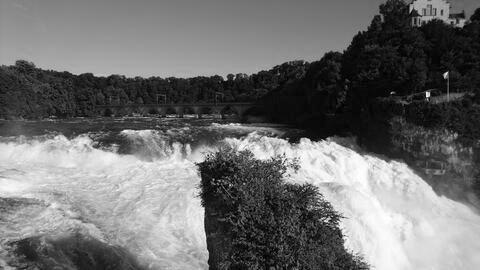}
 		\subcaption{Ground Truth, Image 18}
 		\label{fig:schaff_gt}
 	\end{subfigure}
 	\hfill
 	\begin{subfigure}[t]{0.49\textwidth}
 		\centering
 		\includegraphics[width=\linewidth]{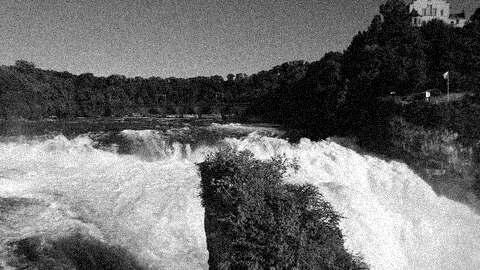}
 		\subcaption{Image 18, noisy, $\sigma^2=0.0025$}
 		\label{fig:schaff_00025}
 	\end{subfigure}
 	\hfill
 	\begin{subfigure}[t]{0.49\textwidth}
 		\centering
 		\includegraphics[width=\linewidth]{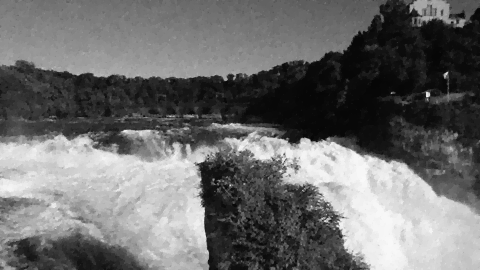}
 		\subcaption{Denoised from $\sigma^2=0.0025$ with IRV1, $\mu=0.03$, $\epsilon=0.01$ (best MSSIM), MSSIM $=0.920268$, PSNR $=30.6543$}
 		\label{fig:schaff_00025_v1}
 	\end{subfigure}
 	\hfill
 	\begin{subfigure}[t]{0.49\textwidth}
 		\centering
 		\includegraphics[width=\linewidth]{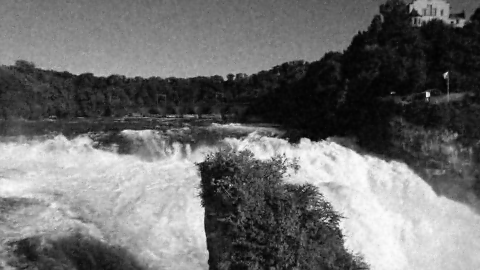}
 		\subcaption{Denoised from $\sigma^2=0.0025$ with IRV2, $\alpha=0.001$ (best MSSIM), MSSIM $= 0.921744$, PSNR $= 30.9195$}
 		\label{fig:schaff_00025_v2}
 	\end{subfigure}
 	\caption{Denoising example for light noise, $\sigma^2=0.0025$.}
 	\label{fig:schaff_denoising_00025}
 \end{figure}
For high noise levels ($\sigma^2=0.09$), the choice of $\alpha$ in IRV2 appears to have a significant impact on reconstruction quality. In some cases, the best performance is achieved at widely differing $\alpha$ values. For example, in \cref{tab:expsig03}, Image 18 achieves the highest PSNR at $\alpha=0.0003$, whereas the best MSSIM is attained at $\alpha=3.0$. The resulting reconstructions also differ visually (see \cref{fig:schaff_009_v2_00003,fig:schaff_009_v2_3}). In contrast, IRV1 exhibits more stable behavior with respect to the parameter $\mu$, with optimal values varying less across metrics and images (see again \cref{tab:expsig03}).
\begin{figure}[htbp]
  \centering

  \subfloat[Image 18, noisy, $\sigma^2=0.09$\\~\\~\label{fig:schaff_009}]{
    \includegraphics[width=0.48\linewidth]{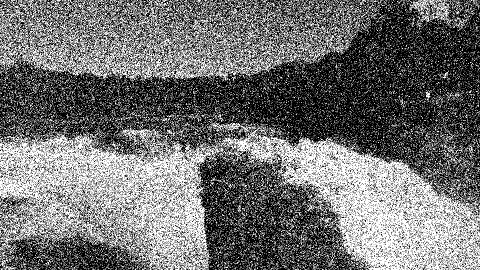}
  }\hfill
  \subfloat[Denoised from $\sigma^2=0.09$ with IRV1, $\mu=0.3$, $\epsilon=10^{-7}$ (best PSNR and MSSIM), MSSIM $= 0.681267$, PSNR $= 21.4898$\label{fig:schaff_009_v1}]{
    \includegraphics[width=0.48\linewidth]{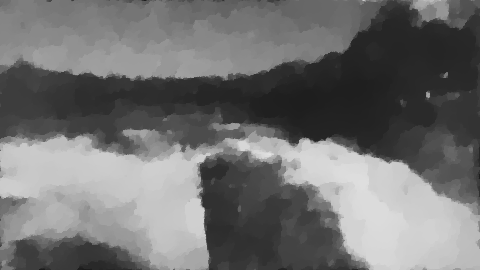}
  }\\[1ex]
  \subfloat[Denoised from $\sigma^2=0.09$ with IRV2, $\alpha=3.0$ (best MSSIM), MSSIM $= 0.669294$, PSNR $= 20.9540$\label{fig:schaff_009_v2_3}]{
    \includegraphics[width=0.48\linewidth]{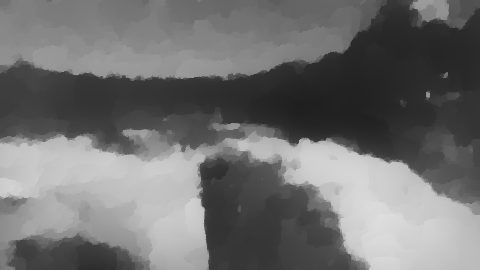}
  }\hfill
  \subfloat[Denoised from $\sigma^2=0.09$ with IRV2, $\alpha=0.0003$ (best PSNR), MSSIM $= 0.645003$, PSNR $= 21.2022$\label{fig:schaff_009_v2_00003}]{
    \includegraphics[width=0.48\linewidth]{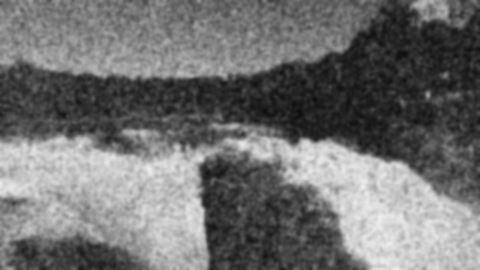}
  }

  \caption{Denoising example for heavy noise, $\sigma^2=0.09$}
  \label{fig:schaff_denoising_009}
\end{figure}
A general observation from \cref{tab:expsig001,tab:expsig005,tab:expsig01,tab:expsig03} is that the selected parameter $\delta$ with the best performance is getting bigger with increasing $\sigma^2$. This makes sense, since the smoothing term should be weighted more when the noise level is large. Similarly, we see that the best performing $\mu$ is growing. In a similar way, one would expect that for IRV2 the best performing $\alpha$ should be decreasing. This cannot be observed, as in some cases, the denoised image yielding the largest PSNR is obtained for very low $\alpha$, cf.\  \cref{fig:schaff_009_v2_3,fig:schaff_009_v2_00003}.

\section{Domain Decomposition for TV-Stokes}\label{sec:DD}

\subsection{Discrete Domain Decomposition}

We utilize the overlapping domain decomposition algorithm presented in \cite{ChaTaiWanYan} to decompose the TV-Stokes model into problems of smaller size.
To this end, we consider the discrete rectangular domains $\Omega^h$ and $\tilde \Omega^h$ 
%
and decompose them into $M_2\times M_1$ overlapping rectangular domains, where $M_1,M_2\in\N$ denote the number of subdomains in $y$ (vertical) and $x$ (horizontal) directions, respectively. For $\Gamma^h \in \{\Omega^h,\tilde\Omega^h\}$ each subdomain is denoted by $\Gamma^h_m=\Gamma^h_{m,y}\times\Gamma^h_{m,x}$, $m=(m_2,m_1)=(1,1),\ldots,(M_2,M_1)$, such that 
\[
\Gamma^h = \bigcup_{m=(1,1)}^{(M_2,M_1)} \Gamma^h_m \quad \text{and} \quad \Gamma^h_m\cap \Gamma^h_{l} \not=\emptyset \ \text{for} \ m\not=l.
\]
By $s>0$ we denote the overlapping size of subdomains which refers to the width of the stripe, measured in grid points, by which one subdomain extends into its neighboring domain.
We introduce a partition of unity $\{\theta^h_m\}_m$ satisfying the following properties:
\begin{enumerate}[(i)]
	\item $\sum\limits_{m=(1,1)}^{(M_2,M_1)} \theta^h_m = 1$, $\theta^h_m \geq 0$ for all $m=(1,1),...,(M_2,M_1)$;
	\item $\theta^h_m\in\mcX(\Gamma^h,1)$ with $\operatorname{supp}\theta^h_m \subset \Gamma^h_m$;
	\item $\|\nabla^h \theta^h_m\|_\infty \leq \frac{C_0}{s}$, with $C_0>0$ independent of the overlapping size $s$.
\end{enumerate}
Utilizing the partition of unity functions $\{\theta^h_m\}_m$
we define for a set $\mathcal{S}\subseteq \mcX(\Gamma^h,2\times c)$, $c=1,2$, the space 
\begin{align*}
\B^h_m(\mathcal{S})&:=\{\vec{p}^h\in \mathcal{S} \ : \ |(\vec{p}^h_{k})_{i,j}|\leq (\theta^h_m)_{i,j} \text{ for all } i=1,...,\#\Gamma_y^h,\; j=1,...,\#\Gamma_x^h \; \text{and} \; k=1,...,c\}.
\end{align*}
Building on these elements, the discrete parallel domain decomposition method of \eqref{eq:opt_dual_discrete} is outlined in \cref{alg:disalgsubsp}. 

\begin{algorithm}[H]
	\newcommand{\Break}{\State \textbf{break} }
	\caption{Discrete parallel domain decomposition}\label{alg:disalgsubsp}
	\begin{algorithmic}
		\Require $M=(M_2,M_1)$, $\widehat\alpha\in\big(0,1\big]$
		\For{$n = 0,1,2,...,max\_it$}
		\For{$m=(1,1),...,(M_2,M_1)$}
		\State $\vec{\hat{q}}^{h,n}_m \gets \argmin\limits_{\vec v^h\in \B_m(\mcX(\Gamma^h,2\times c))}\mcD^h\left(\vec{v}^h+\sum\limits_{l\neq m}\theta^h_l\vec{p}^{h,n}\right)$ ~~~~~~~~~~~from Algorithm \ref{alg:tfsparinner}
		\EndFor
		\State $\vec{p}^{h,n+1} \gets (1-\widehat\alpha)\vec p^{{h,n}}+\widehat\alpha \sum\limits_{l=(1,1)}^{(M_2,M_1)}\hat{\vec q}_l^{h,n}$
		\If{stop\_criteria}
		\Break
		\EndIf
		\EndFor
	\end{algorithmic}
\end{algorithm}

Note that, using the \textit{coloring technique} (see, e.g., \cite{TosWid}), the value of $\widehat\alpha$ can be optimized based on the minimal number of colors required to partition the index set $\{(1,1), \ldots, (M_2, M_1)\}$ such that the corresponding subdomains $(\Gamma^h_m)_m$ with the same color are mutually non-overlapping.  Consequently, we choose $\widehat\alpha = 1$ for $M_1 = M_2 = 1$, $\widehat\alpha = 0.5$ for $M_1 = 1$ or $M_2 = 1$, and $\widehat\alpha = 0.25$ for $M_1 > 1$ and $M_2 > 1$.

Note that in \cref{alg:disalgsubsp} one sets $\Gamma^h = \tilde{\Omega}^h$ for Step~1 (Tangent Field Smoothing, \eqref{eq:tfsdualdis}) and $\Gamma^h = {\Omega}^h$ for Step~2 (Image Reconstruction, \eqref{eq:ir1dualdis}-\eqref{eq:ir2dualdis}). 
Although our framework allows different decompositions in each step of the TV-Stokes model, it seems reasonable to use the same splitting for both steps in the sense that $\Omega_m^h = \tilde \Omega_m^h\cap{\Omega^h}$ for all $m=(1,1),\ldots,(M_2,M_1)$. 

By \cite{ChaTaiWanYan, HilbLanger2022} \cref{alg:disalgsubsp} generates a sequence $(\vec{p}^{h,n})_n$ which converges with order $\mathcal{O}(n^{-1/2})$ to a solution $\vec{p}^{h,*}$ of \eqref{eq:opt_dual_discrete}. This convergence order is also ensured if the subproblems in \cref{alg:disalgsubsp} are only solved approximately \cite{HilbLanger2022}.
A solution strategy for solving the respective subproblems is the semi-implicit dual multiplier method presented in \cite{HilbLanger2022}, which is a generalization of the algorithm proposed in \cite{Chambolle:2004}.
An explicit formulation of the algorithm is given in \cref{alg:tfsparinner}.
\begin{algorithm}[htbp]
	\newcommand{\Break}{\State \textbf{break} }
	\caption{Discrete parallel domain decomposition: Inner loop (Chambolle's algorithm	)}\label{alg:tfsparinner}
	\begin{algorithmic}
		\Require $m \in \{(1,1),\ldots,(M_2,M_1)\}, \vec{p}^{h,n},\vec{\hat{q}}_m^{h,n}\in \mcX(\Gamma^h, 2\times c),\vec{f}^h\in\mcX(\Gamma^h,c), t\in(0,\frac{1}{8}]$ 
		\State $\vec{v}^{h,0}\gets\vec{\hat{q}}_m^{h,n}$
		\For{$\nu = 0,1,...,max\_inner\_it$}
		\State $\vec\psi^{h,\nu} = \left(\vec{\psi}_1^{h,\nu},\ldots, \vec{\psi}_c^{h,\nu}\right) \gets (\Lambda^{h})^*\left(-\Lambda^h\left(\vec{v}^{h,\nu}+\sum\limits_{l\neq m}\theta^h_l\vec{p}^{h,n}\right)+\vec{f}^h\right)$
		\State
		$\vec{v}^{h,\nu+1} = \left( \vec{v}_1^{h,\nu+1}, \ldots, \vec{v}_c^{h,\nu+1}\right)\gets 
		\left(\dfrac{\theta_m^h\vec{v}_1^{h,\nu}+t\theta_m^h \vec\psi_1^{h,\nu}}{\theta_m^h+t|\vec\psi_1^{h,\nu}|},\ldots, \dfrac{\theta_m^h\vec{v}_c^{h,\nu}+t\theta_m^h \vec\psi_c^{h,\nu}}{\theta_m^h+t|\vec\psi_c^{h,\nu}|}\right) 
		$
		\If{inner\_stop\_criteria}
		\State $\vec{\hat{q}}_m^{h,n+1} \gets \vec{v}^{h,\nu+1}$
		\Break
		\EndIf
		\EndFor
		\State\Return $\vec{\hat{q}}_m^{h,n+1}$
	\end{algorithmic}
\end{algorithm}

Although \cref{alg:disalgsubsp}, which relies on \cref{alg:tfsparinner}, performs well in practice, it is important to note that all expressions still reside on the full set $\Gamma^h$. The key motivation behind our formulation, however, is to design the loop in \cref{alg:tfsparinner} to operate only on the subset $\Gamma^h_m$, thereby limiting memory usage.
For IRV1 and IRV2, this strategy can be implemented relatively smoothly, aside from minor complications at subdomain boundaries, since $\div^h_{\Omega^h}$ and $(\Lambda^h)^*=-\grad^h_{\Omega^h}$ are not completely local, cf.\ \cite{LangerGaspoz:19}). In the case of TFS, however, the situation is more involved: the operator $\Lambda^h$ includes $\mathcal{P}^h_{K^h}$, which is completely global and becomes entangled with $\vec{f}^h$ during the iteration, both of which are defined on $\tilde\Omega^h$. Resolving this issue and deriving a localized variant of \cref{alg:tfsparinner} for the TFS 
case is the focus of the following subsection.

\subsection{Local Subspace Iterations}\label{Sec:LocalSubspaceIteration}
To formulate the discrete subspace iterations for the parallel TV-Stokes-algorithm in its local variant, we introduce further notation. In what follows, we focus solely on the TFS step and therefore restrict our attention to the domain $\tilde{\Omega}^h$.

\subsubsection{Further Useful Notations}
Here and in the sequel, we assume that all subsets of $\tilde{\Omega}^h$ are, like $\tilde{\Omega}^h$ itself, rectangular domains. This allows us to define finite difference operators on these subsets as in \cref{Sec:Discrete:Notation}.

For $c\in\{1,2, 2\times 2\}$ and a subset $A^h\subset\tilde{\Omega}^h=\tilde{\Omega}_y^h\times\tilde{\Omega}_x^h$, we define 
the extension operator $E_{A^h} : \mathcal{X}(A^h,c)\to \mathcal{X}(\tilde{\Omega}^h,c)$ as
\begin{align*}
	E_{A^h}\vec{u}^h (x) &:= \begin{cases}
	\vec{u}^h(x), & x\in {A^h}, \\
	0, & x\in \tilde{\Omega}^h~\!\backslash~\! {A^h}.
	\end{cases}
\end{align*}
Furthermore, since $A^h$ is rectangular, we can write it as $A^h=A^h_y\times A^h_x$ with $A^h_x\subset \tilde{\Omega}^h_x$, $A^h_y\subset\tilde{\Omega}^h_y$.
Let the rectangular domain \( A^h \subset \tilde{\Omega}^h \) be written as
\(
A^h = \{ \vec{x}_{i,j} \in \tilde{\Omega}^h : i_1 \leq i \leq i_2,\ j_1 \leq j \leq j_2 \},
\)
where \( 1 \leq i_1 < i_2 \leq \tilde{N}_2 \) and \( 1 \leq j_1 < j_2 \leq \tilde{N}_1 \) determine the position and size of \( A^h \) within \( \tilde{\Omega}^h \). Based on this setup, we define the first vertical and horizontal 1-pixel-wide stripes immediately outside and adjacent to \( A^h \), namely to its right and below, as follows:
\[
S_r = \left\{ \vec{x}_{i,j} \in \tilde{\Omega}^h : i = i_2 + 1,\ j_1 \leq j \leq j_2 \right\}, \quad
S_b = \left\{ \vec{x}_{i,j} \in \tilde{\Omega}^h : i_1 \leq i \leq i_2,\ j = j_2 + 1 \right\}.
\]
Similarly we define the last vertical $1$-pixel-wide stripe within $A^h$ at its right boundary, and the last horizontal $1$-pixel-wide stripe within $A^h$ at its bottom, by 
\[
S_{r,-1} = \left\{ \vec{x}_{i,j} \in \tilde{\Omega}^h : i = i_2,\ j_1 \leq j \leq j_2 \right\}, \quad
S_{b,-1} = \left\{ \vec{x}_{i,j} \in \tilde{\Omega}^h : i_1 \leq i \leq i_2,\ j = j_2 \right\}.
\]

This allows us to define
\begin{align}
A^h_{+} &:= \begin{cases}
A^h&\text{~if~} i_2=\tilde{N}_2 \text{~and~} j_2=\tilde{N}_1\ (\text{i.e.,~} x_{\tilde{N}_2,\tilde{N}_1} \in A^h),\\ 
A^h\cup~\!S_{r}, &  \text{~if~} i_2=\tilde{N}_2 \text{~and~} j_2<\tilde{N}_1 ,\\ 
A^h\cup~\!S_{b}, &   \text{~if~} i_2<\tilde{N}_2 \text{~and~} j_2=\tilde{N}_1 ,\\ 
A^h\cup~\!S_{r}\cup~\!S_{b},
 &   \text{~if~} i_2<\tilde{N}_2 \text{~and~} j_2<\tilde{N}_1 ,\\
\end{cases}
\nonumber\\
A^h_{-} &:= \begin{cases}
A^h&\text{~if~} i_2=\tilde{N}_2 \text{~and~} j_2=\tilde{N}_1\\ 
A^h~\!\backslash~\!S_{r,-1}, &  \text{~if~} i_2=\tilde{N}_2 \text{~and~} j_2<\tilde{N}_1 ,\\ 
A^h~\!\backslash~\!S_{b,-1}, & \text{~if~} i_2<\tilde{N}_2 \text{~and~} j_2=\tilde{N}_1 ,\\ 
A^h~\!\backslash~\!S_{r,-1}~\!\backslash~\!S_{b,-1},
&  \text{~if~} i_2<\tilde{N}_2 \text{~and~} j_2<\tilde{N}_1 .\\
\end{cases}
\end{align}
These notations are introduced to better describe the behavior of the operators $\grad^h_{A^h}$ and $\div^h_{A^h}$, which turn out to be almost local when embedded in $\tilde{\Omega}^h$. In particular, one easily sees that 
\begin{align}\label{resDiv}
		\div^h_{\tilde{\Omega}^h}(E_{A^h_-}\vec{u}^h)
		=(E_{A^h}\div^h_{A^h}R_{A^h})(E_{A^h_-}\vec{u}^h)
\end{align}
for all $\vec{u}^h\in \mcX(A^h_-,2)$ and 
\begin{align}\label{resGrad}
		R_{A^h_-}\grad^h_{\tilde{\Omega}^h} d^h
		&=R_{A^h_-}(E_{A^h}\grad^h_{A^h}R_{A^h}) d^h
		\end{align}
for all $d^h\in \mcX(\tilde{\Omega}^h,1)$.

\subsubsection{Decomposing Linear Operators}
Our goal is to localize the action of global linear operators by formulating their restriction on smaller overlapping subdomains. In particular, we consider the projection operator $\mcP[K^h]^h$,  which is global by definition. To reduce memory consumption and enable efficient subdomain solvers within the domain decomposition method, we study how this operator acts on locally supported data. For this, we analyze the composition $(R_{\tilde{\Omega}^h_{m,+}}\mathcal{P}^h_{K^h}E_{\tilde{\Omega}^h_{k,+}})\vec{w}^h$ for $\vec{w}^h\in\mcX(\tilde{\Omega}^h_k,2)$ for all $k,m\in\{(1,1),...,(M_2,M_1)\}$. The reason we derive a formula for the extended domains $\tilde{\Omega}^h_{m,+}$ and $\tilde{\Omega}^h_{k,+}$ instead of $\tilde{\Omega}^h_m$ and $\tilde{\Omega}^h_k$ is 
due to the involved divergence and gradient operators in the context of subspace iterations (see \cref{alg:tfsparinnertfs,alg:tfsparinnertfsloc}).
Although $\mcP[K^h]^h$ itself is a global operator, we will show that its action on locally supported functions can be computed in a fully local manner. This observation forms the foundation for a memory-aware implementation of the domain decomposition scheme.

Let $c_1,c_2\in\{1,2\}$ and let $\mathcal{T}^h: \mcX(\tilde{\Omega}^h,c_1)\to\mcX(\tilde{\Omega}^h,c_2)$ be a linear operator. For fixed $k$ and $m$ we choose disjoint decompositions $(A^h_\kappa)_{\kappa=(1,1),...,(M_2,M_1)}$, $(\tilde{A}^h_\lambda)_{\lambda=(1,1),...,(M_2,M_1)}$ and $(B^h_\mu)_{\mu=(1,1),...,(M_2,M_1)}$ such that 
\begin{align}\label{conditionAB}
\tilde{\Omega}^h_{k,+}=A^h_{k,-},\qquad \tilde{\Omega}^h_{m,+}=B^h_{m,-}
\end{align}
and
\begin{align}\label{def:2DDs}
\tilde{\Omega}^h~=~\dot\bigcup_{\kappa=(1,1)}^{(M_2,M_1)}A^h_{\kappa}, 
\qquad
\tilde{\Omega}^h~=~\dot\bigcup_{\lambda=(1,1)}^{(M_2,M_1)}\tilde A^h_{\lambda},
\qquad
\tilde{\Omega}^h~=~\dot\bigcup_{\mu=(1,1)}^{(M_2,M_1)}B^h_{\mu},
\end{align}
where all $A^h_\kappa$, $\tilde A^h_{\lambda}$ and $B^h_\mu$ are rectangular, so that they can be represented as $A^h_\kappa=A^h_{y,\kappa_2}\times A^h_{x,\kappa_1}$, $\tilde A^h_\lambda=\tilde A^h_{y,\lambda_2}\times \tilde A^h_{x,\lambda_1}$ and $B^h_\mu=B^h_{y,\mu_2}\times B^h_{x,\mu_1}$.
We now decompose $\mathcal{T}^h$, which, by linearity, can be expressed in terms of its sub-operators associated with $(A_\kappa^h)_\kappa$ and $(B_\mu^h)_\mu$ by
\begin{equation}
\begin{split}
\mathcal{T}^h
&~ = ~ \mathcal{T}^h\left(\sum\limits_{\kappa=(1,1)}^{(M_2,M_1)} E_{A^h_{\kappa}} R_{A^h_{\kappa}}\right)
~ = ~ \sum\limits_{\mu=(1,1)}^{(M_2,M_1)} E_{B^h_\mu} R_{B^h_\mu}\mathcal{T}^h\left(\sum\limits_{\kappa=(1,1)}^{(M_2,M_1)} E_{A^h_{\kappa}} R_{A^h_{\kappa}}\right)\\
&~ = ~ \sum\limits_{\mu=(1,1)}^{(M_2,M_1)}\sum\limits_{\kappa=(1,1)}^{(M_2,M_1)}E_{B^h_\mu}\left(R_{B^h_\mu} \mathcal{T}^h E_{A^h_{\kappa}}\right) R_{A^h_{\kappa}}.
\end{split}
\end{equation}
This allows us to find a simple formula for the composition of two decomposed operators:
\begin{lemma}\label{lmChainOp}
	Let $c_1,c_2,c_3 \in\{1,2\}$ and let $\mathcal{T}^h :  \mcX(\tilde{\Omega}^h,c_1)\to\mcX(\tilde{\Omega}^h,c_2)$ and $\mathcal{U}^h : \mcX(\tilde{\Omega}^h,c_2)\to\mcX(\tilde{\Omega}^h,c_3)$ be linear operators and $\mathcal{U}^h \mathcal{T}^h : \mcX(\tilde{\Omega}^h,c_1)\to\mcX(\tilde{\Omega}^h,c_3)$ its composition.
Let	$(A_\kappa^h)_\kappa$, $(\tilde{A}^h_\lambda)_\lambda$ and $(B_\mu^h)_\mu$ be disjoint decompositions of $\tilde{\Omega}^h$, as in \eqref{def:2DDs}.
	 Then we have that
	\begin{align*}
	R_{B^h_\mu} (\mathcal{U}^h \mathcal{T}^h) E_{A^h_{\kappa}} 
	= \sum\limits_{\lambda=(1,1)}^{(M_2,M_1)} (R_{B^h_\mu} \mathcal{U}^h E_{\tilde{A}^h_\lambda}) (R_{\tilde{A}^h_\lambda}\mathcal{T}^h E_{A^h_\kappa}).
	\end{align*}
\end{lemma}
\begin{proof}
A straightforward calculation shows 
	\begin{align*}
	 R_{B^h_\mu} (\mathcal{U}^h\mathcal{T}^h) E_{A^h_\kappa} = R_{B^h_\mu} \mathcal{U}^h \left(\sum\limits_{\lambda=(1,1)}^{(M_2,M_1)} E_{\tilde{A}^h_\lambda} R_{\tilde{A}^h_\lambda}(\mathcal{T}^h E_{A^h_\kappa})\right)
	= \sum\limits_{\lambda=(1,1)}^{(M_2,M_1)} (R_{B^h_\mu} \mathcal{U}^h E_{\tilde{A}^h_\lambda}) (R_{\tilde{A}^h_\lambda}\mathcal{T}^h E_{A^h_\kappa}).
	\end{align*}
	{\ }
\end{proof}

\paragraph{Formula for the discrete global projection on local domain}
We are now ready to fully localize the action of $R_{\tilde{\Omega}^h_{m,+}}\mcPh[K^h] E_{\tilde{\Omega}^h_{k,+}}$.
Let $k,m\in\{(1,1),...,(M_2,M_1)\}$ be fixed and $\vec{\tau}^h_{k,+}:=R_{\tilde{\Omega}^h_{k,+}}\vec{\tau}^h$. 
Using the notations from above and the representation \eqref{formula:PKh}  of $\mcPh[K^h]$, we obtain
\begin{equation}\label{locglobprojpart1}
\begin{split}
(R_{\tilde{\Omega}^h_{m,+}}&\mcPh[K^h] E_{\tilde{\Omega}^h_{k,+}})\vec{\tau}^h_{k,+}\\
& \underset{\eqref{formula:PKh}}{=} R_{\tilde{\Omega}^h_{m,+}}E_{\tilde{\Omega}^h_{k,+}}\vec{\tau}^h_{k,+}~-~R_{\tilde{\Omega}^h_{m,+}}\grad^h_{\tilde{\Omega}^h} (\Delta^h_{\tilde{\Omega}^h})^\dagger\div^h_{\tilde{\Omega}^h}(E_{\tilde{\Omega}^h_{k,+}}\vec{\tau}^h_{k,+})  \\
& \underset{\eqref{conditionAB}}{=} R_{\tilde{\Omega}^h_{m,+}}E_{\tilde{\Omega}^h_{k,+}}\vec{\tau}^h_{k,+}~-~R_{B^h_{m,-}}\grad^h_{\tilde{\Omega}^h}(\Delta^h_{\tilde{\Omega}^h})^\dagger\div^h_{\tilde{\Omega}^h}(E_{A^h_{k,-}}\vec{\tau}^h_{k,+})  \\
& \underset{(\ref{resDiv})}{=} R_{\tilde{\Omega}^h_{m,+}}E_{\tilde{\Omega}^h_{k,+}}\vec{\tau}^h_{k,+}~-~R_{B^h_{m,-}}\grad^h_{\tilde{\Omega}^h}(\Delta^h_{\tilde{\Omega}^h})^\dagger (E_{A^h_k}\div^h_{A^h_k}R_{A^h_k})(E_{A^h_{k,-}}\vec{\tau}^h_{k,+})  \\
& \underset{(\ref{resGrad})}{=} R_{\tilde{\Omega}^h_{m,+}}E_{\tilde{\Omega}^h_{k,+}}\vec{\tau}^h_{k,+}~-~R_{B^h_{m,-}}(E_{B^h_m}\grad^h_{B_m^h}R_{B_m^h})(\Delta^h_{\tilde{\Omega}^h})^\dagger (E_{A^h_k}\div^h_{A^h_k}R_{A^h_k})(E_{A^h_{k,-}}\vec{\tau}^h_{k,+}) \\
&~=~R_{\tilde{\Omega}^h_{m,+}}E_{\tilde{\Omega}^h_{k,+}}\vec{\tau}^h_{k,+}~-~R_{B^h_{m,-}}E_{B^h_m}\grad^h_{B_m^h}\big(R_{B_m^h}(\Delta^h_{\tilde{\Omega}^h})^\dagger E_{A^h_k}\big)\div^h_{A^h_k}R_{A^h_k}E_{A^h_{k,-}}\vec{\tau}^h_{k,+}.
\end{split}
\end{equation}
It remains to derive a formula for $R_{B_m^h}(\Delta^h_{\tilde{\Omega}^h})^\dagger E_{A^h_k}$.
Using \eqref{form:DiscInvLapl} and \cref{lmChainOp} with the decompositions $(A^h_\kappa)_\kappa$, $(\tilde{A}^h_\lambda)_\lambda$ and $(B^h_\mu)_\mu$, we derive 
\begin{equation}\label{locglobprojpart3}
\begin{split}
&\left(R_{B_m^h}(\Delta^h_{\tilde{\Omega}^h})^\dagger E_{A^h_k}\right) d_k^h\\
&\phantom{\hspace{2.5cm}} ~=~ R_{B_m^h}\left((\mathcal{C}^h_{\tilde{\Omega}^h})^{-1} (\tilde{\Delta}^h_{\tilde{\Omega}^h})^\dagger \mathcal{C}^h_{\tilde{\Omega}^h}\right)E_{A^h_k} d_k^h \\
&\phantom{\hspace{2.5cm}} ~=~ \sum\limits_{\lambda=(1,1)}^{(M_2,M_1)} ~\left(R_{B_m^h}\big((\mathcal{C}^h_{\tilde{\Omega}^h})^{-1} (\tilde{\Delta}^h_{\tilde{\Omega}^h})^\dagger\big) E_{\tilde{A}^h_\lambda}\right)\left(R_{\tilde{A}^h_\lambda}\mathcal{C}^h_{\tilde{\Omega}^h}E_{A^h_k}\right) d_k^h\\
&\phantom{\hspace{2.5cm}} ~=~ \sum\limits_{\lambda'=(1,1)}^{(M_2,M_1)}\sum\limits_{\lambda=(1,1)}^{(M_2,M_1)} \left(R_{B_m^h}(\mathcal{C}^h_{\tilde{\Omega}^h})^{-1}E_{{\tilde{A}}^h_{\lambda'}}\right)  \left(R_{{\tilde{A}}^h_{\lambda'}}(\tilde{\Delta}^h_{\tilde{\Omega}^h})^\dagger E_{\tilde{{A}}^h_{\lambda}}\right) \left(R_{\tilde{{A}}^h_{\lambda}}\mathcal{C}^h_{\tilde{\Omega}^h}E_{A^h_k}\right) d_k^h \\
&\phantom{\hspace{2.5cm}} ~=~ \sum\limits_{\lambda=(1,1)}^{(M_2,M_1)}  \left(R_{B_m^h}(\mathcal{C}^h_{\tilde{\Omega}^h})^{-1}E_{{\tilde{A}}^h_{\lambda}}\right)  \left(R_{{\tilde{A}}^h_{\lambda}}(\tilde{\Delta}^h_{\tilde{\Omega}^h})^\dagger E_{\tilde{{A}}^h_{\lambda}}\right) \left(R_{\tilde{{A}}^h_{\lambda}}\mathcal{C}^h_{\tilde{\Omega}^h}E_{A^h_k}\right) d_k^h.
\end{split}
\end{equation}
for all $d_k^h\in\mcX(A_k^h,1)$.
Note that $R_{{\tilde{A}}^h_{\lambda'}}(\tilde{\Delta}^h_{\tilde{\Omega}^h})^\dagger E_{\tilde{{A}}^h_{\lambda}}=0$ for $\lambda\neq\lambda'$, since $(\tilde{\Delta}^h_{\tilde{\Omega}^h})^\dagger$ is a completely local operator; see \eqref{form:DiscInvLaplTild}. 

Since $m=(m_2,m_1)$, $\lambda=(\lambda_2,\lambda_1)$ and $k=(k_2,k_1)$ are double indices, a more precise, expanded notation can be derived:
Remember that from (\ref{eq:DCT2D}) and (\ref{eq:DCT2Dinv}), we can represent the 2D-DCT as
\begin{align*}
\mathcal{C}^h_{\tilde{\Omega}^h} d^h &= C_{\tilde{N}_2}~ d^h~ C_{\tilde{N}_1}^T, \qquad
(\mathcal{C}^h_{\tilde{\Omega}^h})^{-1}d^h = C_{\tilde{N}_2}^T~ d^h~ C_{\tilde{N}_1},
\end{align*}
for any $d^h \in \mcX(\Omega^h,1)$
with $C_{\tilde{N}_2}\in\R^{\tilde{N}_2\times \tilde{N}_2}$ and $C_{\tilde{N}_1}\in\R^{\tilde{N}_1\times \tilde{N}_1}$. 
Note that the multiplication with $C_{\tilde{N}_1}^T$ corresponds to a 1D-DCT in $x$-direction and the multiplication with $C_{\tilde{N}_2}$ corresponds to a 1D-DCT in $y$-direction. 
We can localize $\mathcal{C}^h_{\tilde{\Omega}^h}$ and $(\mathcal{C}^h_{\tilde{\Omega}^h})^{-1}$ by only evaluating the part of $d^h$ that we need. 
So, for instance if we want to apply $\mathcal{C}^h_{\tilde{\Omega}^h}$ on a local $d_k^h\in\mcX(A_k^h,1)$ and restrict the result on $\tilde{A}^h_\lambda$, it is enough to evaluate the corresponding matrix blocks $[C_{\tilde{N}_1}]^{k_1}_{\lambda_1}\in\R^{(\#\tilde{A}^h_{x,\lambda_1})\times(\#A^h_{x,k_1})}$ within $C_{\tilde{N}_1}$ and $[C_{\tilde{N}_2}]^{k_2}_{\lambda_2}\in\R^{(\#\tilde{A}^h_{y,\lambda_2})\times(\#A^h_{y,k_2})}$ within $C_{\tilde{N}_2}$ (see Figure \ref{fig:matvisualization} for a visualization of the blocks).

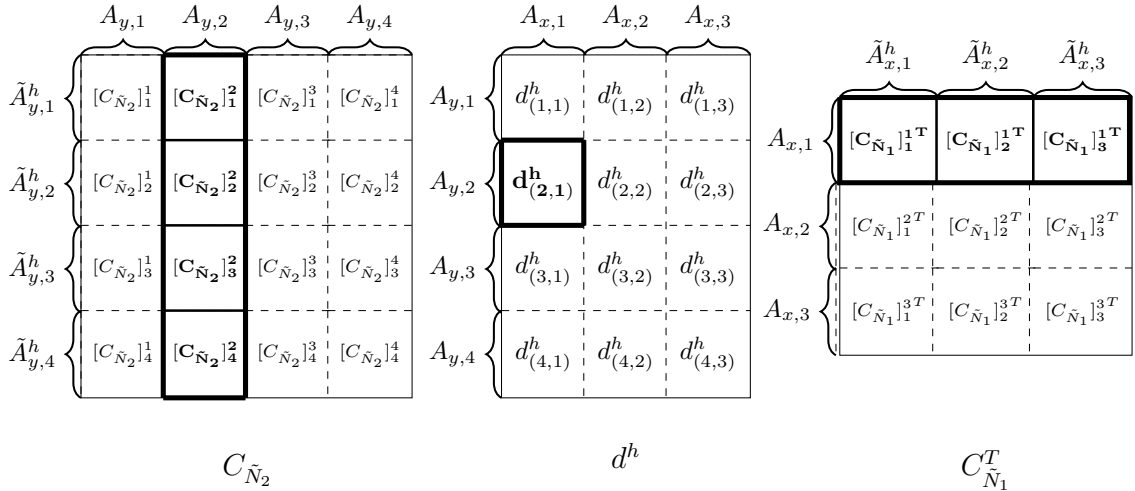
\begin{figure}[h]
\begin{center}
\begin{adjustbox}{max width=\textwidth}
\begin{tikzpicture}
\matrix (m) [matrix of nodes,
nodes in empty cells,
nodes={minimum width=1.15cm, minimum height=1.2cm, anchor=center},
column sep=-\pgflinewidth, row sep=-\pgflinewidth] 
{
	$\scriptstyle{[C_{\tilde N_2}]_1^{1}}$ & $\mathbf{\scriptstyle{[C_{\tilde N_2}]_1^{2}}}$ & $\scriptstyle{[C_{\tilde N_2}]_1^{3}}$ & $\scriptstyle{[C_{\tilde N_2}]_1^{4}}$ \\
	$\scriptstyle{[C_{\tilde N_2}]_2^{1}}$ & $\mathbf{\scriptstyle{[C_{\tilde N_2}]_2^{2}}}$ & $\scriptstyle{[C_{\tilde N_2}]_2^{3}}$ & $\scriptstyle{[C_{\tilde N_2}]_2^{4}}$ \\
	$\scriptstyle{[C_{\tilde N_2}]_3^{1}}$ & $\mathbf{\scriptstyle{[C_{\tilde N_2}]_3^{2}}}$ & $\scriptstyle{[C_{\tilde N_2}]_3^{3}}$ & $\scriptstyle{[C_{\tilde N_2}]_3^{4}}$ \\
	$\scriptstyle{[C_{\tilde N_2}]_4^{1}}$ & $\mathbf{\scriptstyle{[C_{\tilde N_2}]_4^{2}}}$ & $\scriptstyle{[C_{\tilde N_2}]_4^{3}}$ & $\scriptstyle{[C_{\tilde N_2}]_4^{4}}$\\
};

\draw(m-1-1.north west) rectangle (m-4-4.south east);
\draw[line width=2.0pt](m-1-2.north west) -- (m-1-3.north west);
\draw[line width=2.0pt](m-4-2.south west) -- (m-4-3.south west);
\draw[line width=2.0pt](m-1-2.north west) -- (m-4-2.south west);
\draw[line width=2.0pt](m-1-3.north west) -- (m-4-3.south west);
\draw[line width=1pt](m-2-2.north west) -- (m-2-3.north west);
\draw[line width=1pt](m-3-2.north west) -- (m-3-3.north west);
\draw[line width=1pt](m-4-2.north west) -- (m-4-3.north west);
\foreach \i in {1,...,4} {
	\foreach \j in {1,...,4} {
		\draw[dashed](m-\i-\j.north west) -- (m-\i-\j.south west);
		\draw[dashed](m-\i-\j.north west) -- (m-\i-\j.north east);
	}
}

\foreach \i/\label in {1/$\tilde A_{y,1}^h$, 2/$\tilde A_{y,2}^h$, 3/$\tilde A_{y,3}^h$, 4/$\tilde A_{y,4}^h$}{
	\draw[decorate, decoration={brace, amplitude=6pt, mirror}, thick]
	([yshift=0.0cm]m-\i-1.north west) -- ([yshift=0.0cm]m-\i-1.south west)
	node[midway,left=6pt] {\label};
}

\foreach \j/\label in {1/$A_{y,1}$, 2/$A_{y,2}$, 3/$A_{y,3}$, 4/$A_{y,4}$} {
	\draw[decorate, decoration={brace, amplitude=6pt}, thick]
	([xshift=0.0cm]m-1-\j.north west) -- ([xshift=0.0cm]m-1-\j.north east)
	node[midway,above=6pt] {\label};
}

\node[below=1.0cm of m-4-2.south east, anchor=center, font=\large] (f)
{$C_{\tilde N_2}$};

\matrix (n) [matrix of nodes,
nodes in empty cells,
nodes={minimum width=1.15cm, minimum height=1.2cm, anchor=center},
column sep=-\pgflinewidth, row sep=-\pgflinewidth,
right=1cm of m] 
{
	$d^h_{(1,1)}$ & $d^h_{(1,2)}$ & $d^h_{(1,3)}$\\
	$\mathbf{d^h_{(2,1)}}$ & $d^h_{(2,2)}$ & $d^h_{(2,3)}$\\
	$d^h_{(3,1)}$ & $d^h_{(3,2)}$& $d^h_{(3,3)}$ \\
	$d^h_{(4,1)}$ & $d^h_{(4,2)}$& $d^h_{(4,3)}$ \\
};

\draw(n-1-1.north west) rectangle (n-4-3.south east);
\draw[line width=2.0pt](n-2-1.north west) -- (n-3-1.north west);
\draw[line width=2.0pt](n-2-1.north west) -- (n-2-2.north west);
\draw[line width=2.0pt](n-3-1.north west) -- (n-3-2.north west);
\draw[line width=2.0pt](n-2-2.north west) -- (n-3-2.north west);
\foreach \i in {2,...,4} {
	\draw[dashed](n-\i-1.north west) -- (n-\i-3.north east);
}
\foreach \j in {2,...,3} {
	\draw[dashed](n-1-\j.north west) -- (n-4-\j.south west);
}

\foreach \i/\label in {1/$A_{y,1}$, 2/$A_{y,2}$, 3/$A_{y,3}$, 4/$A_{y,4}$} {
	\draw[decorate, decoration={brace, amplitude=6pt, mirror}, thick]
	([yshift=0.0cm]n-\i-1.north west) -- ([yshift=0.0cm]n-\i-1.south west)
	node[midway,left=6pt] {\label};
}

\foreach \j/\label in {1/$A_{x,1}$, 2/$A_{x,2}$,3/$A_{x,3}$} {
	\draw[decorate, decoration={brace, amplitude=6pt}, thick]
	([xshift=0.0cm]n-1-\j.north west) -- ([xshift=0.0cm]n-1-\j.north east)
	node[midway,above=6pt] {\label};
}

\node[below=0.8cm of n-4-2.south, anchor=center, font=\large] (f)
{$d^h$};


\matrix (o) [matrix of nodes,
nodes in empty cells,
nodes={minimum width=1.35cm, minimum height=1.2cm, anchor=center},
column sep=-\pgflinewidth, row sep=-\pgflinewidth,
right=1cm of n] 
{
	$\mathbf{\scriptstyle{[C_{\tilde N_1}]^1_1}^T} $ & $\mathbf{\scriptstyle{[C_{\tilde N_1}]^1_2}^T} $ & $\mathbf{\scriptstyle{[C_{\tilde N_1}]^1_3}^T} $\\
	$\scriptstyle{[C_{\tilde N_1}]^2_1}^T $ & $\scriptstyle{[C_{\tilde N_1}]^2_2}^T $ & $\scriptstyle{[C_{\tilde N_1}]^2_3}^T $\\
	$\scriptstyle{[C_{\tilde N_1}]^3_1}^T $ & $\scriptstyle{[C_{\tilde N_1}]^3_2}^T $ & $\scriptstyle{[C_{\tilde N_1}]^3_3}^T $\\
};

\draw(o-1-1.north west) rectangle (o-3-3.south east);
\draw[line width=2.0pt](o-1-1.north west) -- (o-1-3.north east);
\draw[line width=2.0pt]([xshift=-1pt]o-2-1.north west) -- (o-2-3.north east);
\draw[line width=2.0pt](o-1-1.north west) -- (o-1-1.south west);
\draw[line width=2.0pt]([xshift=-1pt]o-1-3.north east) -- ([xshift=-1pt]o-1-3.south east);
\draw[line width=1pt](o-1-2.north west) -- (o-1-2.south west);
\draw[line width=1pt](o-1-3.north west) -- (o-1-3.south west);
\foreach \i in {1} {
	\foreach \j in {1,...,3} {
		\draw[dashed](o-\i-\j.north west) -- (o-\i-\j.south west);
		}
		}
\foreach \i in {2,...,3} {
	\foreach \j in {1,...,3} {
		\draw[dashed]([xshift=-1.5pt]o-\i-\j.north west) -- ([xshift=-1.5pt]o-\i-\j.south west);
		\draw[dashed](o-\i-\j.north west) -- (o-\i-\j.north east);
	}
}

\foreach \i/\label in {1/$A_{x,1}$} {
	\draw[decorate, decoration={brace, amplitude=6pt, mirror}, thick]
	([yshift=0.0cm]o-\i-1.north west) -- ([yshift=0.0cm]o-\i-1.south west)
	node[midway,left=6pt] {\label};
}
\foreach \i/\label in {2/$A_{x,2}$, 3/$A_{x,3}$} {
	\draw[decorate, decoration={brace, amplitude=6pt, mirror}, thick]
	([yshift=0.0cm,xshift=-1.5pt]o-\i-1.north west) -- ([yshift=0.0cm,xshift=-1.5pt]o-\i-1.south west)
	node[midway,left=6pt] {\label};
}

\foreach \j/\label in {1/$\tilde A_{x,1}^h$, 2/$\tilde A_{x,2}^h$, 3/$\tilde A_{x,3}^h$} {
	\draw[decorate, decoration={brace, amplitude=6pt}, thick]
	([xshift=0.0cm]o-1-\j.north west) -- ([xshift=0.0cm]o-1-\j.north east)
	node[midway,above=6pt] {\label};
}

\node[below=1.6cm of o-3-2.south, anchor=center, font=\large] (f)
{$C_{\tilde N_1}^T$};
\end{tikzpicture}
\end{adjustbox}
\end{center}
\caption{Visualization of $[C_{\tilde N_2}]_{\lambda_2}^{2} d^h_{(2,1)}\big([C_{\tilde N_1}]_{\lambda_1}^{1}\big)^T$ for all $\lambda=(\lambda_2,\lambda_1)=(1,1),...,(4,3)$ within the bigger matrix $C_{\tilde N_2} d^hC_{\tilde N_1}^T$ for $M_1=3$, $M_2=4$.}\label{fig:matvisualization}
\end{figure}

Equivalently, if we want to apply $(\mathcal{C}^{-1}_{\tilde{\Omega}^h})^h$ on a local $d_\lambda^h\in\mcX(A_\lambda^h,1)$ and restrict the result on $B^h_m$, it is enough to evaluate the corresponding matrix blocks $[C_{\tilde{N}_1}]^{\lambda_1}_{m_1}\in\R^{(\#B^h_{x,m_1})\times(\#\tilde{A}^h_{x,\lambda_1})}$ within $C_{\tilde{N}_1}$ and  $[C_{\tilde{N}_2}]^{\lambda_2}_{m_2}\in\R^{(\#B^h_{y,m_2})\times(\#\tilde{A}^h_{y,\lambda_2})}$ within $C_{\tilde{N}_2}$.
So, for the local operations $R_{\tilde{{A}}^h_{\lambda}}\mathcal{C}^h_{\tilde{\Omega}^h}E_{A^h_k}$ and $R_{B_m^h}(\mathcal{C}^h_{\tilde{\Omega}^h})^{-1}E_{{\tilde{A}}^h_{\lambda}}$ we get
\begin{align*}
&\left(R_{\tilde{{A}}^h_{\lambda}}\mathcal{C}^h_{\tilde{\Omega}^h}E_{A^h_k}\right)d_k^h=[C_{\tilde{N}_2}]^{k_2}_{\lambda_2}~d_k^h~ ([C_{\tilde{N}_1}]^{k_1}_{\lambda_1})^T,\\
&\left( R_{B_m^h}(\mathcal{C}^h_{\tilde{\Omega}^h})^{-1}E_{{\tilde{A}}^h_{\lambda}}\right)d_\lambda^h=([C_{\tilde{N}_2}]^{\lambda_2}_{m_2})^T~d_\lambda^h~[C_{\tilde{N}_1}]^{\lambda_1}_{m_1},
\end{align*}
for all $d_k^h\in\mcX(A_k^h,1)$ and $d_\lambda^h\in\mcX(\tilde{A}_\lambda^h,1)$.
Inserting this into \eqref{locglobprojpart3} 
yields the matrix formula
\begin{equation}\label{invLaplLowCapLong}
\begin{split}
&\left(R_{B_m^h}(\Delta^h_{\tilde{\Omega}^h})^\dagger E_{A^h_k}\right) d_k^h\\
&\phantom{\hspace{2.5cm}}=\sum\limits_{\lambda_2=1}^{M_2}\sum\limits_{\lambda_1=1}^{M_1} ([C_{\tilde{N}_2}]^{\lambda_2}_{m_2})^T 
\left(R_{{\tilde{A}}^h_{\lambda}}(\tilde{\Delta}^h_{\tilde{\Omega}^h})^\dagger E_{\tilde{{A}}^h_{\lambda}}\right)
\Big([C_{\tilde{N}_2}]^{k_2}_{\lambda_2}~d_k^h~ ([C_{\tilde{N}_1}]^{k_1}_{\lambda_1})^T\Big)
[C_{\tilde{N}_1}]^{\lambda_1}_{m_1}
\end{split}
\end{equation}
for all $d_k^h\in\mcX(A_k^h,1)$.

With formula (\ref{invLaplLowCapLong}), the inverse laplacian of an image block $d_k^h$ can be evaluated with never needing more access memory than a constant times the biggest block. Furthermore, the summands are completely independent from each other and can be computed parallely without any communication between the threads (except distributing $d_k^h$ on the threads and collecting the result on the main thread to sum the result up).

To limit random access memory required per thread, the decompositions $(A^h_\kappa)_\kappa$, $(\tilde{A}^h_\lambda)_\lambda$ and $(B^h_\mu)_\mu$ should be chosen carefully.  In particular, it is reasonable to select them as uniformly sized as possible for all $k,m\in\{(1,1),...,(M_2,M_1)\}$, provided that condition~\eqref{conditionAB} remains satisfied.

The primary computational bottleneck lies in the evaluation of the individual blocks of the DCT, which could potentially be optimized using strategies analogous to those employed in fast implementations of the full Discrete Cosine Transform (Fast DCT).

\subsubsection*{Localized Algorithm}

Applying \cref{alg:tfsparinner} to the TFS step, with $\Lambda^h := \mcP[K^h]^h \vecdiv^h_{\tilde{\Omega}^h}$ and $\vec{f}^h:=\delta^{-1} \mathcal{P}^h_{K^h}\vec{\tau}^h_0$, we receive \cref{alg:tfsparinnertfs}.

\begin{algorithm}[htbp]
	\newcommand{\Break}{\State \textbf{break} }
	\caption{Discrete parallel DD: Inner loop for Tangent Field Smoothing}
	\label{alg:tfsparinnertfs}
	\begin{algorithmic}
		\Require $m\in\{(1,1),...,(M_2,M_1)\}, \vec{p}^{h,n},\vec{\hat{q}}_m^{h,n}\in\mcX(\tilde\Omega^h,2\times 2), \vec{f}^h\in\mcX(\tilde{\Omega}^h,2), t\in(0,\frac{1}{8}]$
			\State $\vec{v}^{h,0}\gets \vec{\hat{q}}_m^{h,n}$
			\State $\vec\omega^{h,0}\gets \vec{f}^h -\mathcal{P}^h_{K^h}\vecdiv^h_{\tilde{\Omega}^h}\sum\limits_{l\neq m}\theta^h_l\vec{p}^{h,n}$
			\For{$\nu = 0,1,...,max\_inner\_it$}
			\State $\vec\psi^{h,\nu} = (\vec\psi^{h,\nu}_1,\vec\psi^{h,\nu}_2) \gets \vecgrad^h_{\tilde{\Omega}^h}\left(\mathcal{P}^h_{K^h} \vecdiv^h_{\tilde{\Omega}^h}\vec{v}^{h,\nu}-\vec\omega^{h,0}\right)$
			\State
			$\vec{v}^{h,\nu+1} 
			= \left(\vec{v}_1^{h,\nu+1},\vec{v}_2^{h,\nu+1}\right)
			\gets \left(\dfrac{\theta_m^h \vec{v}_1^{h,\nu}+t\theta_m^h \vec\psi_1^{h,\nu}}{\theta_m^h+t|\vec\psi_1^{h,\nu}|},
			\dfrac{\theta_m^h\vec{v}_2^{h,\nu}+
			t\theta_m^h\vec\psi_2^{h,\nu}}{\theta_m^h+t|\vec\psi_2^{h,\nu}|}
			\right)$
			\If{inner\_stop\_criteria}
			\State $\vec{\hat{q}}_m^{h,n+1} \gets \vec{v}^{h,\nu+1}$
			\Break
			\EndIf
			\EndFor
		\State\Return $\vec{\hat{q}}_m^{h,n+1}$
	\end{algorithmic}
\end{algorithm}
Note that we introduced $\vec\omega^{h,0}\in\mcX(\tilde\Omega^h, 2)$ to avoid repeatedly computing $\mathcal{P}^h_{K^h}\vecdiv^h_{\tilde{\Omega}^h}\sum\limits_{l\neq m}\theta^h_l\vec{p}^{h,n}$ in every inner iteration step. This will be particularly useful in the following, where we localize the inner iteration entirely.
Building on the discussion above, \cref{alg:tfsparinnertfs} can now be
reformulated to rely solely on local operations; see \cref{alg:tfsparinnertfsloc}. 
To do so, we examine how global quantities can be replaced by their localized counterparts without loss of correctness. When transitioning from $\vec{v}^{h,\nu+1}\in\mcX(\tilde\Omega^h, 2\times 2)$ to the localized version $\vec{v}_{\loc}^{h,\nu+1}\in\mcX(\tilde\Omega_m^h, 2\times 2)$, we exploit the fact that in each iteration $\vec{v}^{h,\nu+1}$ is multiplied by $\theta_m^h$, rendering all values of $\vec{v}^{h,\nu+1}$ and $\vec{\psi}^{h,\nu}$ outside $\tilde{\Omega}_m^h$ redundant. 
Consequently, we can initialize $\vec{v}^{h,0}$ such that $\supp(\vec{v}^{h,0})\subseteq\tilde\Omega^h_m$, which implies $\supp(\vec{v}^{h,\nu})\subseteq\tilde\Omega^h_m$ for all $\nu \in\N$. 
In \cref{alg:tfsparinnertfsloc}, this is achieved by setting $\vec{v}_{\loc}^{h,0} = R_{\tilde{\Omega}^h_m}\vec{\hat{q}}_m^{h,n}$ . Since $\supp(\vec{v}_{\loc}^{h,\nu})\subseteq\tilde\Omega^h_m$, the non-zero entries of $\vecdiv^h_{\tilde{\Omega}^h}\vec{v}_{\loc}^{h,\nu}$ lie entirely within the extended subdomain $\tilde\Omega^h_{m,+}$. 
The fact that
\begin{equation*}
R_{\tilde{\Omega}^h_m} \vecgrad^h_{\tilde{\Omega}^h} \vec{v} = R_{\tilde{\Omega}^h_m} E_{\tilde{\Omega}_{m,+}^h}\vecgrad^h_{\tilde{\Omega}^h_{m,+}}R_{\tilde{\Omega}^h_{m,+}} \vec{v}
\end{equation*}
for any $\vec{v}\in \mcX(\tilde{\Omega}^h,2)$, and
\begin{equation*}
\vecdiv^h_{\tilde{\Omega}^h} E_{\tilde{\Omega}^h_m} \vec{v}_{\loc} = \vecdiv^h_{\tilde{\Omega}^h_{m,+}}R_{\tilde{\Omega}^h_{m,+}}E_{\tilde{\Omega}^h_{m}}\vec{v}_{\loc}
\end{equation*}
for any $\vec{v}_{\loc} \in \mcX(\tilde{\Omega}^h_m,2\times 2)$, together with the local computation of $\mathcal{P}^h_{K^h}$ on $\tilde\Omega^h_{m,+}$, which corresponds to $R_{\tilde{\Omega}^h_{m,+}}\mathcal{P}^h_{K^h} E_{\tilde{\Omega}^h_{m,+}}$ and can be computed via formulas \eqref{locglobprojpart1} and \eqref{invLaplLowCapLong}, yields the localized update of $\vec\psi_{\loc}^{h,\nu}$ in \cref{alg:tfsparinnertfsloc}. 
Moreover, we only require the localized quantity $\vec\omega^{h,0}_{\loc}:= R_{\tilde{\Omega}^h_{m,+}}\vec\omega^{h,0}$ instead of the global $\vec\omega^{h,0}$, since $\vec\psi^{h,\nu}$ is multiplied by $\theta_m^h$ in each iteration, rendering all entries of $\vec\omega^{h,0}$ outside $\tilde\Omega_m^h$ redundant.

\begin{algorithm}[htbp]
	\newcommand{\Break}{\State \textbf{break} }
	\caption{Discrete parallel DD: Localized inner loop for Tangent Field Smoothing}
	\label{alg:tfsparinnertfsloc}
	\begin{algorithmic}
		\Require $m\in\{(1,1),...,(M_2,M_1)\},~\vec{p}^{h,n},~ \vec{\hat{q}}_m^{h,n}\in\mcX(\tilde\Omega^h,2\times 2),~ \vec{f}^h\in\mcX(\tilde{\Omega}^h,2), ~t\in(0,\frac{1}{8}]$
		\State $\vec{v}_{\loc}^{h,0}\gets R_{\tilde{\Omega}^h_m}\vec{\hat{q}}_m^{h,n}$
		\State $\vec\omega^{h,0}_{\loc}\gets R_{\tilde{\Omega}^h_{m,+}}\left(\vec{f}^h -\mathcal{P}^h_{K^h}\vecdiv^h_{\tilde{\Omega}^h}\sum\limits_{l\neq m}\theta^h_l\vec{p}^{h,n}\right)$
		\For{$\nu = 0,1,...,max\_inner\_it$}
		\State $\vec\psi_{\loc}^{h,\nu} \gets R_{\tilde{\Omega}^h_m}E_{\tilde\Omega^h_{m,+}}\vecgrad^h_{\tilde{\Omega}^h_{m,+}}\left(R_{\tilde{\Omega}^h_{m,+}}\mathcal{P}^h_{K^h} E_{\tilde{\Omega}^h_{m,+}}\vecdiv^h_{\tilde{\Omega}^h_{m,+}}R_{\tilde{\Omega}^h_{m,+}}E_{\tilde{\Omega}^h_{m}}\vec{v}_{\loc}^{h,\nu}-\vec\omega_{\loc}^{h,0}\right)$, \\
		~~~~~~~~~~~~~~~where 
		$R_{\tilde{\Omega}^h_{m,+}}\mathcal{P}^h_{K^h} E_{\tilde{\Omega}^h_{m,+}}$ is computed via (\ref{locglobprojpart1}) and (\ref{invLaplLowCapLong})
		\State
		$\vec{v}_{\loc}^{h,\nu+1} 
		= \left(\vec{v}_{\loc,1}^{h,\nu+1},\vec{v}_{\loc,2}^{h,\nu+1}\right) 
		\gets \left(\dfrac{R_\theta\vec{v}_{\loc,1}^{h,\nu}+t(R_\theta) \vec\psi_{\loc,1}^{h,\nu}}{R_\theta+t|\vec\psi_{\loc,1}^{h,\nu}|},
		\dfrac{R_\theta\vec{v}_{\loc,2}^{h,\nu}+t R_\theta \vec\psi_{\loc,2}^{h,\nu}}{R_\theta+t|\vec\psi_{\loc,2}^{h,\nu}|}
		\right)$\\
		~~~~~~~~~~~~~~~where $R_\theta := R_{\tilde{\Omega}_m^h} \theta_m^h$
		\If{inner\_stop\_criteria}
		\State $\vec{\hat{q}}_m^{h,n+1} \gets E_{\tilde{\Omega}^h_m}\vec{v}_{\loc}^{h,\nu+1}$
		\Break
		\EndIf
		\EndFor
		\State\Return $\vec{\hat{q}}_m^{h,n+1}$
	\end{algorithmic}
\end{algorithm}

In summary, each operation within the inner loop is confined to the extended subdomain
 $\Omega^h_{m,+}$, implying that we have formulated a fully local algorithm, even for the TFS step in \cref{alg:disalgsubsp}.
However, it is worth noting that the initialization of $\vec\omega_{\loc}^{h,0}$ in \cref{alg:tfsparinnertfsloc} may still be computationally expensive, as it involves the global operator $\mcP[K^h]^h$. 
To preserve the locality of the overall algorithm, this projection is computed in 
localized pieces $R_{\Omega^h_m}\mathcal{P}^h_{K^h}E_{\Omega^h_k}$, using formulas \eqref{locglobprojpart1} and \eqref{invLaplLowCapLong}. 
These partial results can be communicated between threads and subsequently combined. Importantly, this global projection needs to be computed only once prior to the iteration process.

\subsection{Numerical Validation}\label{sec:DDNumVal}
We implemented our domain decomposition (DD) approach of the TV-Stokes model using multi-threading, where each thread processes only a local portion of the image. We evaluated this DD strategy for both image reconstruction variants (IRV1 and IRV2) and compared the results to the standard, non-decomposed TV-Stokes model (referred to as non-DD).
Our experiments were conducted on Image 10 (see \cref{fig:beach_denoising_comparison}), corrupted with medium-level Gaussian noise ($\sigma^2 = 0.01$) added to the ground truth. For the DD implementation, we employed a $3 \times 3$ domain decomposition ($M_1 = M_2 = 3$), with overlaps of 3 pixels in the $x$-direction and 4 pixels in the $y$-direction. We applied a coloring scheme across the overlapping domains which allowed us to use $\widehat{\alpha} = 0.25$ as weight parameter.

The regularization parameters were set as follows: $\delta = 0.15$ for TFS, $\mu = 0.1$ for IRV1, and $\alpha = 10.0$ for IRV2. In IRV1, the field $\vec\xi^h$ was computed from the tangent field $\vec\tau^h$ using $\epsilon = 0.001$. All optimization problems were solved using Chambolle's algorithm with a step size $t = 0.125$.

To verify numerical consistency between the DD and non-DD approaches, we used the non-DD TV-Stokes model with $10^6$ iterations per step (i.e., TFS, IRV1, and IRV2) to compute reference solutions and energies. 
In the DD setting, we performed 10 inner iterations ($max\_inner\_it$) and 5\;000 outer iterations ($max\_it$) per step, or terminated earlier if $\frac{1}{|\Gamma^h|}|\mcD^h(\vec p^{h,n})^2-\mcD^h(\vec p^{h,n+1})^2|<10^{-10}$, where $\Gamma^h$ is either $\tilde{\Omega}^h$ or $\Omega^h$, and $\mcD^h$ is either $\mcD^h_{\mathrm{TFS}}$, $\mcD^h_{\mathrm{IRV1}}$, or $\mcD^h_{\mathrm{IRV2}}$ depending on the step and variant. 
In \cref{fig:dd_energy}, we compared the energy progress of the reference solution and the DD-solution.
It can be seen that for TFS, the DD-algorithm clearly converges against the reference solution, demonstrating the suitability and correctness of our localized DD approach described in \cref{Sec:LocalSubspaceIteration}. 

When evaluating the second step of the TV-Stokes model, we recall that errors propagate from Step~1 to Step~2; cf.\ \cref{sec:ComparisonIR}. In particular, if we use the solution of the DD approach, which does not match the reference solution, the DD iterations of Step~2, i.e., IRV1 and IRV2, converge but not to the reference energy; see solid blue curve in \cref{fig:dd_energy} (b) and (c). This is because a perturbed objective is minimized as a result of error propagation.  
Despite this, when examining the reconstructions in \cref{fig:beach_denoising_comparison}, the DD results appear visually hardly distinguishable from the non-DD reconstructions.
On the other hand, if we use the reference tangent field in Step~2, then for both variants the DD-algorithm converges to the reference energy as expected; see the dashed orange curve in \cref{fig:dd_energy} (b) and (c).

\begin{figure}[htbp]
  \centering

  \subfloat[Ground Truth, Image~10\label{fig:beach_gt}]{
    \includegraphics[width=0.48\linewidth]{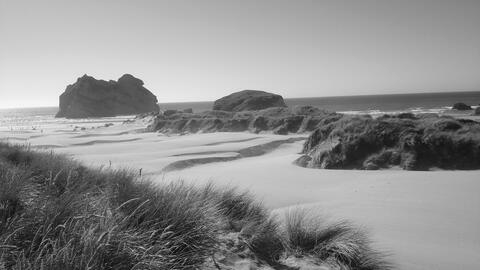}
  }\hfill
  \subfloat[Noisy image, $\sigma^2=0.01$\label{fig:beach_noisy_001}]{
    \includegraphics[width=0.48\linewidth]{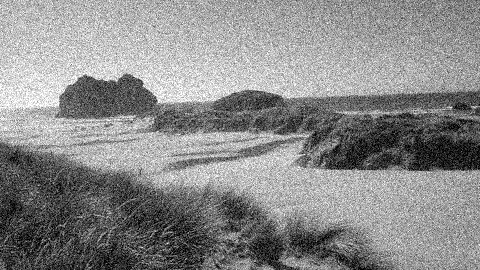}
  }\\[1ex]

  \subfloat[Denoised with IRV1, non-DD\label{fig:beach_denoised_v1}]{
    \includegraphics[width=0.48\linewidth]{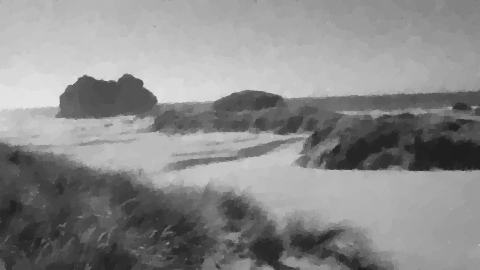}
  }\hfill
  \subfloat[Denoised with IRV1, DD\label{fig:beach_denoised_v1_dd}]{
    \includegraphics[width=0.48\linewidth]{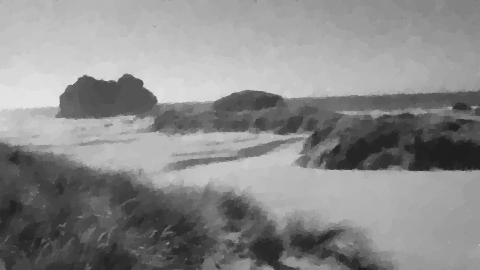}
  }\\[1ex]

  \subfloat[Denoised with IRV2, non-DD\label{fig:beach_denoised_v2}]{
    \includegraphics[width=0.48\linewidth]{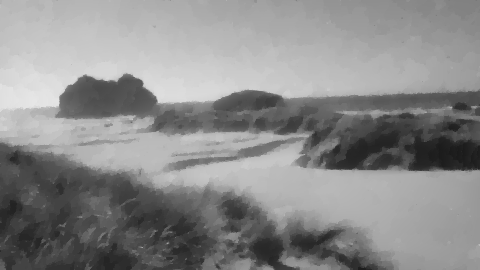}
  }\hfill
  \subfloat[Denoised with IRV2, DD\label{fig:beach_denoised_v2_dd}]{
    \includegraphics[width=0.48\linewidth]{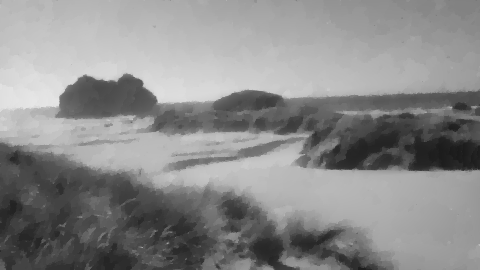}
  }

  \caption{Denoising comparison non-DD vs.\ DD for Image~10}
  \label{fig:beach_denoising_comparison}
\end{figure}

\begin{figure}[htbp]
  \centering

  \subfloat[TFS, Energy development DD, Image~10\label{fig:dd_tfs_energy}]{
    \includegraphics[width=0.48\linewidth]{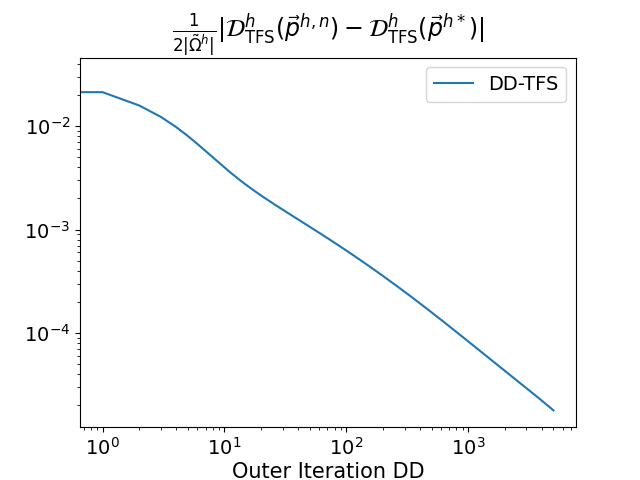}
  }\\[1ex]

  \subfloat[IRV1, Energy development DD, Image~10\label{fig:dd_irv1_energy}]{
    \includegraphics[width=0.48\linewidth]{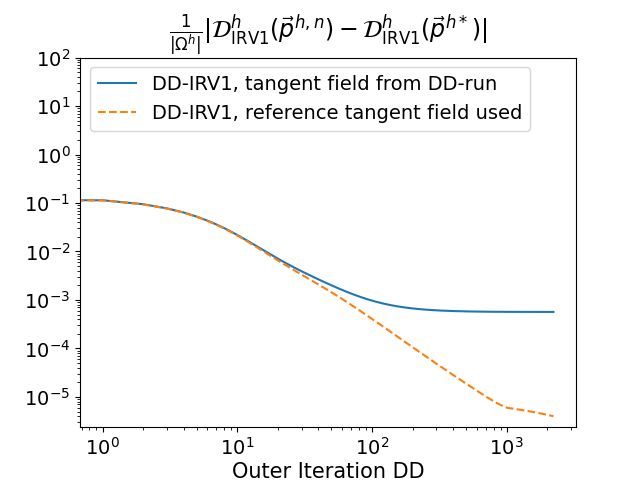}
  }\hfill
  \subfloat[IRV2, Energy development DD, Image~10\label{fig:dd_irv2_energy}]{
    \includegraphics[width=0.48\linewidth]{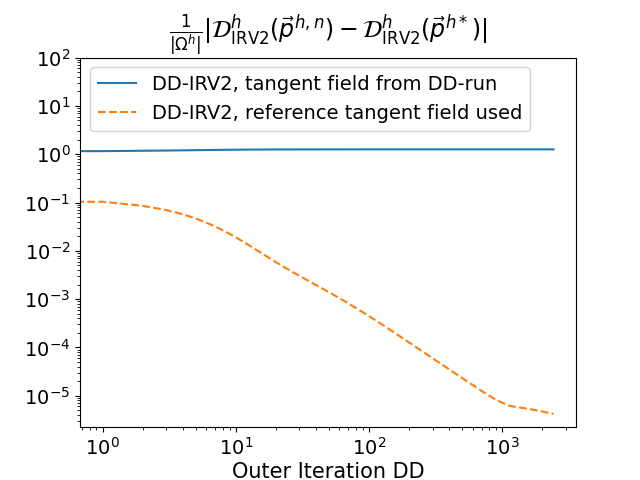}
  }

  \caption{Convergence of energies against reference energy for domain decomposition and Image~10, where $\vec p^{h*}$ denotes the respective reference solution}
  \label{fig:dd_energy}
\end{figure}

	
\section{Conclusions}
\label{sec:conclusions}

We provided a functional-analytic description of the TV-Stokes model by formulating both steps as precise variational problems, identifying the appropriate function spaces and dual formulations, and clarifying under which assumptions the divergence-free tangent field and the subsequent image reconstruction are mathematically compatible. Within this framework, we analyzed two variants of the image reconstruction step.

The analysis shows that these two variants differ in their variational coherence with the tangent field smoothing step. IRV1 can be formulated in a mathematically consistent way, but only under additional regularity assumptions and technical modifications. IRV2, by contrast, is directly compatible with the structure of the first step and leads to a standard TV problem with shifted data. From a variational perspective, IRV2 therefore provides a more natural and structurally aligned realization of the TV-Stokes model. At the same time, the numerical experiments indicate that both reconstruction variants yield visually similar results.

Finally, we derived an overlapping domain decomposition method for the TV-Stokes model. Although the continuous formulation involves global operators, the discrete setting admits a localized realization of all required computations. This allows the use of domain decomposition without modifying the underlying variational problems and enables parallel implementations that remain faithful to the analytical structure of the model.

%

\backmatter

%
%
%

\bmhead{Acknowledgements}

Tai is partially supported by the NORCE Kompetanseoppbygging program.

\begin{appendices}

\section{Dualization}\label{Sec:Dualization}
In this section we present the dualization of the TV-Stokes model. To this end we consider 
		\begin{align}\label{eq:quadraticTVproblem}
			\min\limits_{\vec{u}\in \bfK \cap BV(\Omega,\R^c)} \Big\{TV(\vec u) + F(\vec u)\Big\},
		\end{align}
where $c\in\N$, $\bfK \subseteq L^2(\Omega,\R^c)$ a closed subspace in $L^2(\Omega,\R^c)$, and $F(\vec u) = \frac{\beta}{2}\|\vec u + \frac{1}{\beta} \vec{f}\|_{L^2}^2$ with $\beta > 0$, $\vec{f}\in L^2(\Omega,\R^c)$. 
Note that a solution of \eqref{eq:quadraticTVproblem} is ensured \cite{ChambolleLions1997}.
\begin{example}
\begin{enumerate}[(i)] 
\item 
Set $c:=2$, $\bfK := K=\{ \vec{\tau} \in L^2(\Omega,\R^2)\colon \div\vec{\tau}=0\}$ denoting the null space of $\div: L^2(\Omega,\R^2) \to H^{-1}(\Omega,\R)$, $\vec{f} := -\frac{1}{\delta}\vec{\tau}_0$ and $\beta:=\frac{1}{\delta}$. 
Note that since $\div$ is a linear and bounded operator, $\Null(\div)$ is closed in $L^2(\Omega,\R^c)$ \cite[2.7-10 Corollary, p.98]{Kreyszig:1991}. 
Then we obtain
			\begin{align*}
				F(\vec{\tau})=\frac{1}{2\delta}\|\vec{\tau}-\vec{\tau}_0\|_{L^2}^2
			\end{align*}
and \eqref{eq:quadraticTVproblem} resembles \eqref{eq:cTFS}.
\item Setting $c:=1$, $\bfK:=L^2(\Omega,\R)$, $\vec{f}:=\div\vec{\xi}-\alpha d_0$ and $\beta:=\alpha$ yields
			\begin{align*}
			F(d) &= \frac{\alpha}{2}\|d + \frac{1}{\alpha}(\div\vec{\xi}-\alpha d_0)\|_{L^2}^2 = \frac{\alpha}{2} \|d - d_0\|_{L^2}^2 + \left\langle d- d_0, \div\vec\xi \right\rangle_{L^2} + \frac{\alpha}{2} \|\div\vec\xi\|_{L^2}^2
			\end{align*}
and hence \eqref{eq:quadraticTVproblem} resembles \eqref{IRPXi} with $\vec{\xi}$ as in 
\eqref{eq:defXi2}.
\end{enumerate}
\end{example}

		\begin{lemma}\label{Lemma:RCsupClosed}
Let $C\subseteq L^2(\Omega,\R^c)$ be a non-empty set and $R_C : L^2(\Omega,\R^c) \rightarrow\R\cup\{\infty\}$ be the support function of the set $C$ defined as
			\(
				R_C(u)~:=~\sup\limits_{w\in C}\langle u,w\rangle_{L^2},
			\)		
			then we have that $R_C(u)=R_{\overline{C}}(u)$ for all $u\in L^2(\Omega,\R^c)$. 
		\end{lemma}
		\begin{proof}
			It is clear that $R_{C}(u)\leq R_{\overline{C}}(u)$. For the other direction, fix $w\in \overline{{C}}$ and choose a sequence $(w_n)_n \subset C$ such that $w_n\to w$ in $L^2(\Omega,\R^c)$. Then we have
			\begin{align*}
			\big|\langle u,w\rangle_{L^2}-\langle u,w_n\rangle_{L^2}\big|
			=\big| \langle u,w-w_n\rangle_{L^2}\big|
			\leq\|u\|_{L^2}\|w-w_n\|_{L^2}\to 0,
			\end{align*}
			which implies
			$
			\langle u,w\rangle_{L^2}
			=\lim\limits_{n\to\infty}\langle u,w_n\rangle_{L^2}
			\leq R_{C}(u).
			$
			Since this holds for all $w\in \overline{C}$, it also holds for the supremum and we get $R_{\overline{{C}}}(u)\leq R_{C}(u)$. In conclusion, we get 
			$
			R_{C}(u)
			=R_{\overline{{C}}}(u)
			=\sup\limits_{u^{*}\in {\overline{{C}}}}\langle u,u^{*}\rangle_{L^2}
			$
			for all $u\in L^2(\Omega,\R^c)$.
		\end{proof}

\begin{lemma}\label{Lemma:ConvConjInd}
			 Let $C\subseteq L^2(\Omega,\R^c)$ be a non-empty closed convex set. Then for the convex conjugate (Legendre-Fenchel-Transform) $R^{*}_{C}$ of $R_{C}$ we have $R^{*}_{C}(u^{*})=\chi_{C}(u^{*})$ where 
			 \begin{align}\label{defind}
			 	\chi_{C}(\omega):=
			 	\begin{cases}
			 		0, &\text{if $\omega\in C$,}\\
			 		+\infty &\text{otherwise,}
			 	\end{cases}
			 \end{align}
			 denotes the indicator function.
		\end{lemma}
		\begin{proof}
 By \cite[Theorem 13.2]{rockefellar:1970}, the convex conjugate of the indicator function is the support function (defined in \cref{Lemma:RCsupClosed}), i.e.\
			\begin{align}\label{formConvConjInd}	\chi_{C}^{*}(\omega^{*})=R_{C}(\omega^{*}).
			\end{align}
Moreover, thanks to \cite[Theorem 12.2]{rockefellar:1970}, we have 
			\begin{align}\label{formConv2ConjCl}
			\cl\chi_{C}=\chi_{C}^{**}.
			\end{align}
			Then, since $\chi_{C}$ is closed, we get
			\begin{align*}
			R^{*}_{C}(u^{*})
			\underset{(\ref{formConvConjInd})}{=}\chi_{C}^{**}(u^{*})
			\underset{(\ref{formConv2ConjCl})}{=}\cl\chi_{C}(u^{*})
			=\chi_{C}(u^{*}).
			\end{align*}
		\end{proof}
		
Let $C^\bfK_0:=\mcP[\bfK] \vecdiv \B(\mathcal{C}_0^1(\Omega,\R^{2\times c})) =\{\mcP[\bfK] \vecdiv \vec{p} \in \bfK \colon \vec{p}\in \B(\mathcal{C}_0^1(\Omega,\R^{2\times c}))\}$, then for $\vec{u}\in \bfK$ we have that 
\begin{align*}
	TV(\vec{u})	&=	\sup\limits_{ \vec{p}\in \B(\mathcal{C}_0^1(\Omega,\R^{2\times c}))}\langle \vec{u},\vecdiv \vec{p}\rangle_{L^2}
			= \sup\limits_{ \vec{p}\in \B(\mathcal{C}_0^1(\Omega,\R^{2\times c}))}\langle \mcP[\bfK] \vec{u},\vecdiv \vec{p}\rangle_{L^2}\\
			&= \sup\limits_{ \vec{p}\in \B(\mathcal{C}_0^1(\Omega,\R^{2\times c}))}\langle \vec{u},\mcP[\bfK] \vecdiv \vec{p}\rangle_{L^2}
			= \sup\limits_{\vec{w}\in C^\bfK_0}\langle\vec{u},\vec{w}\rangle_{L^2}
			= R_{\overline{{C^\bfK_0}}}(\vec{u}),
			\end{align*}
where the latter equality follows from \cref{Lemma:RCsupClosed}. From these considerations we observe that 
\begin{align*}
	 R_{\overline{{C^\bfK_0}}}(\vec{u})= \begin{cases}
	\TV(\vec{u}) & \text{ if } \vec{u} \in \bfK,\\
	\TV(\mcP[\bfK]\vec{u}) & \text{ otherwise.}
	\end{cases}
\end{align*}
Thus \eqref{eq:quadraticTVproblem} takes the form
\begin{align*}
			\min\limits_{u\in \bfK\cap BV(\Omega,\R^c)} R_{\overline{{C^\bfK_0}}}(u)+F(u).
\end{align*}

\begin{theorem}\label{Thm:Dual}
Let $c\in\N$, $\bfK \subseteq L^2(\Omega,\R^c)$ a closed subspace in $L^2(\Omega,\R^c)$, and $F(u):=\frac{\beta}{2}\|\vec{u}+\frac{1}{\beta}\vec{f}\|_{L^2}^2$ 
with $\beta > 0$, 
$\vec{f}\in \bfK$. Then any solution $\vec{u}^*$ of
\begin{align}\label{eq:quadraticTVproblemR}
			\min\limits_{\vec{u}\in L^2(\Omega,\R^c)\cap BV(\Omega,\R^c)}  R_{\overline{{C^\bfK_0}}}(u)+F(u)
		\end{align}
is in $\bfK$, i.e.\ $\vec{u}^*\in \bfK$ and \eqref{eq:quadraticTVproblemR} is equivalent to \eqref{eq:quadraticTVproblem}. Moreover, $\vec{u}^* = -\frac{1}{\beta} \vec{f} - \frac{1}{\beta} \vec{s}^*$, where $\vec{s}^*$ 
\begin{equation}\label{eq:DualProblem}
\min\limits_{\vec{s}\in \overline{C^\bfK_0}} \| \vec{s} + \vec{f}\|_{L^2}^2 = \inf\limits_{\vec{p}\in \B(H_0^{\div}(\Omega,\R^{2\times c}))} \| \mcP[\bfK]\vecdiv\vec{p} + \vec{f}\|_{L^2}^2.
\end{equation}
\end{theorem}
\begin{proof}
Note that $R_{\overline{{C^\bfK_0}}}$ and $F$ are convex and lower semicontinuous functionals. Further there exists $\vec{u} \in \Dom(R_{\overline{{C^\bfK_0}}}) \cap \Dom(F)$ where $F$ is continuous, e.g.\ any constant function.
A minimizer $\vec{u}^*\in L^2(\Omega,\R^c)\cap BV(\Omega,\R^c)$ of \eqref{eq:quadraticTVproblemR} fulfills the necessary condition
			\begin{align}\label{Rumf1}
			0\in \partial \left(R_{\overline{{C^\bfK_0}}}(\vec{u}^*)+F(\vec{u}^*) \right)= \partial R_{\overline{{C^\bfK_0}}}(\vec{u}^*)+\partial F(\vec{u}^*) = \partial R_{\overline{{C^\bfK_0}}}(\vec{u}^*)+ \beta \vec{u}^* + \vec{f},
			\end{align}
where we used \cite[Prop. 5.6, p.26]{ekeland:1999} and the definition of $F$.
By \cite[Corollary 5.2, p.22]{ekeland:1999} we obtain
\begin{align*}
\vec{u}^*\in\partial R_{\overline{{C^\bfK_0}}}^{*}(-\beta \vec{u}^* - \vec{f})
\end{align*}
			which is equivalent to
\begin{align*}
0\in \frac{1}{\beta}\left(-\beta\vec{u}^* - \vec{f}\right) + \frac{1}{\beta} \vec{f} + \partial R_{\overline{{C^\bfK_0}}}^{*}(-\beta \vec{u}^* - \vec{f}).
\end{align*}
Setting $\vec{s}^*:=-\beta \vec{u}^* - \vec{f} \in L^2(\Omega,\R^c)$ we obtain
			\begin{align*}
			0\in \frac{1}{\beta} \left( \vec{s}^* + \vec{f}\right) + \partial R_{\overline{{C^\bfK_0}}}^{*}(\vec{s}^*)
			\end{align*}
and hence $\vec{s}^*$ is the minimizer of
			\begin{align}\label{eq:optwRkstar}
				\inf\limits_{\vec{s}\in L^2(\Omega,\R^c)}\frac{1}{2\beta} \| \vec{s} + \vec{f}\|_{L^2}^2 + R_{\overline{{C^\bfK_0}}}^{*}(\vec{s}).
			\end{align}
By \cref{Lemma:ConvConjInd} and the same arguments as in the proof of \cref{Lemma:RCsupClosed} we obtain that \eqref{eq:optwRkstar} can be rewritten as
\begin{align*}
\inf\limits_{\vec{s}\in \overline{{C^\bfK_0}}}\frac{1}{2\beta} \| \vec{s} + \vec{f}\|_{L^2}^2 =\inf\limits_{\vec{s}\in {C^\bfK_0}}\frac{1}{2\beta} \| \vec{s} + \vec{f}\|_{L^2}^2 = \inf\limits_{\vec{p}\in \B(\mathcal{C}_0^1(\Omega,\R^{2\times c}))}\frac{1}{2\beta} \| \mcP[\bfK]\vecdiv\vec{p} + \vec{f}\|_{L^2}^2
\end{align*}
Since $\B(C_0^{1}(\Omega,\R^{2\times c}))$ is dense in the sense of $H_0^{\div}(\Omega,\R^{2\times c})$ in $\B(H_0^{\div}(\Omega,\R^{2\times c}))$ \cite{HintermullerRautenberg:15, HiLaAl2023} we obtain
\begin{align}\label{eq:DualProblemProof}
 \inf\limits_{\vec{p}\in \B(\mathcal{C}_0^1(\Omega,\R^{2\times c}))}\frac{1}{2\beta} \| \mcP[\bfK]\vecdiv\vec{p} + \vec{f}\|_{L^2}^2  = \inf\limits_{\vec{p}\in \B(H_0^{\div}(\Omega,\R^{2\times c}))}\frac{1}{2\beta} \| \mcP[\bfK]\vecdiv\vec{p} + \vec{f}\|_{L^2}^2.
\end{align}

Let $\vec{p}^*$ be a solution of \eqref{eq:DualProblemProof}. Then from the relation $\vec{s}^*:=-\beta \vec{u}^* - \vec{f}$ we obtain $\mcP[\bfK]\vecdiv \vec{p}^* = -\beta \vec{u}^*-\vec{f}$. That is $\vec{u}^* = -\frac{1}{\beta} \vec{f} - \frac{1}{\beta}\mcP[\bfK]\vecdiv \vec{p}^*$ and since $\vec{f}\in\bfK$ we get $\vec{u}^* \in \bfK$.
\end{proof}

We call \eqref{eq:DualProblem} the dual formulation  of \eqref{eq:quadraticTVproblem}.
Note that $\inf\limits_{\vec{s}\in \overline{{C^\bfK_0}}}\frac{1}{2\beta} \| \vec{s} + \vec{f}\|_{L^2}^2$ has a unique solution for any $\vec{f}\in L^2(\Omega, \R^c)$, see, e.g., \cite[3.3-1 Thm., p.144]{Kreyszig:1991}. However, in general this cannot be ensured for 
$\inf\limits_{\vec{p}\in \B(H_0^{\div}(\Omega,\R^{2\times c}))} \| \mcP[\bfK]\vecdiv\vec{p} + \vec{f}\|_{L^2}^2$, as the orthogonal projection $\mcP[\bfK]$ annihilates components orthogonal to $\bfK$. 
While the latter is of course unpleasant, it is practically not relevant as long as we can find a minimal dual energy. Actually, in finite dimensions, utilizing a semi-implicit scheme, as in \cite{Chambolle:2004}, allows to generate a sequence $(\vec{p}^n)_n$ such that $\mcD(\vec{p}^n) \to \mcD(\vec{p})$ for $n\to \infty$, where $\mcP[\bfK]\vecdiv \vec{p}$ is a minimizer of \eqref{eq:DualProblem}, cf.\ \cite[Theorem 7.2]{HilbLanger2022}.

\section{Discrete Orthogonal Projection}\label{sec:disorthproj}
In this section, we show that $\mcPh[K^h]: \mcX(\tilde{\Omega}^h,2) \to K^h$ given by \eqref{formula:PKh} is the unique orthogonal projection onto $K^h$. 
\begin{lemma}\label{lmDiscLaplSym}
	Let $d_1,d_2\in \mcX(\tilde{\Omega}^h,1)$. Then we have that
	\begin{align*}
		\langle (\Delta_{\tilde{\Omega}^h}^h)^\dagger d_1, d_2\rangle_{\mcX(\Omega^h,1)} = \langle d_1, (\Delta_{\tilde{\Omega}^h}^h)^\dagger d_2\rangle_{\mcX(\tilde\Omega^h,1)}.
	\end{align*}
\end{lemma}
\begin{proof}
		We use representation (\ref{form:DiscInvLapl}) of $(\Delta_{\tilde{\Omega}^h}^h)^\dagger$, i.e., $
		(\Delta_{\tilde{\Omega}^h}^h)^\dagger=(\mathcal{C}_{\tilde{\Omega}^h}^h)^{-1}(\tilde\Delta_{\tilde{\Omega}^h}^h)^\dagger\mathcal{C}_{\tilde{\Omega}^h}^h.
		$	
		Let $d_1,d_2\in \mcX(\tilde\Omega^h,1)$. From (\ref{form:DiscInvLaplTild}), one can easily observe that
		\begin{align*}
			\langle (\tilde\Delta_{\tilde{\Omega}^h}^h)^\dagger d_1, d_2\rangle_{\mcX(\tilde\Omega^h,1)} = \langle d_1, (\tilde\Delta_{\tilde{\Omega}^h}^h)^\dagger d_2\rangle_{\mcX(\tilde\Omega^h,1)}.
		\end{align*}
		Furthermore, since the DCT is orthogonal, see e.g.\ \cite[p.\ 496]{Sauer:12}, we get
		\begin{align*}
			\langle\mathcal{C}_{\tilde{\Omega}^h}^h d_1,d_2\rangle_{\mcX(\Omega^h,1)}
			&=\langle C_{\tilde{N}_2}~d_1~C_{\tilde{N}_1}^T,d_2\rangle_{\mcX(\tilde\Omega^h,1)}
			=\langle C_{\tilde{N}_2}^T~d_2~C_{\tilde{N}_1},d_1\rangle_{\mcX(\tilde\Omega^h,1)}\\
			&=\langle d_1,(\mathcal{C}_{\tilde{\Omega}^h}^h)^{-1}d_2\rangle_{\mcX(\tilde\Omega^h,1)}.
		\end{align*}
		Bringing all this together, yields
		\begin{align*}
		&\langle (\Delta_{\tilde{\Omega}^h}^h)^\dagger d_1, d_2\rangle_{\mcX(\tilde{\Omega}^h,1)}
		= \langle (\mathcal{C}_{\tilde{\Omega}^h}^h)^{-1}(\tilde\Delta_{\tilde{\Omega}^h}^h)^\dagger\mathcal{C}_{\tilde{\Omega}^h}^h d_1, d_2\rangle_{\mcX(\tilde{\Omega}^h,1)}
		= \langle (\tilde\Delta_{\tilde{\Omega}^h}^h)^\dagger\mathcal{C}_{\tilde{\Omega}^h}^h d_1, \mathcal{C}_{\tilde{\Omega}^h}^h d_2\rangle_{\mcX(\tilde{\Omega}^h,1)} \\
		&= \langle \mathcal{C}_{\tilde{\Omega}^h}^h d_1, (\tilde\Delta_{\tilde{\Omega}^h}^h)^\dagger\mathcal{C}_{\tilde{\Omega}^h}^h d_2\rangle_{\mcX(\tilde{\Omega}^h,1)}
		= \langle d_1, (\mathcal{C}_{\tilde{\Omega}^h}^h)^{-1}(\tilde\Delta_{\tilde{\Omega}^h}^h)^\dagger\mathcal{C}_{\tilde{\Omega}^h}^h d_2\rangle_{\mcX(\tilde{\Omega}^h,1)}
		= \langle d_1, (\Delta_{\tilde{\Omega}^h}^h)^\dagger d_2\rangle_{\mcX(\tilde{\Omega}^h,1)}.
		\end{align*}
		{\ }
\end{proof}
\begin{proposition}\label{prop:DisProj}
	The operator $\mcPh[K^h]: \mcX(\tilde{\Omega}^h,2) \to K^h$ with
	\(
	\mcPh[K^h] = I^h - \grad_{\tilde{\Omega}^h}^h (\Delta_{\tilde{\Omega}^h}^h)^\dagger \div_{\tilde{\Omega}^h}, 
	\)
	as defined in (\ref{formula:PKh}), is the unique orthogonal projection onto $K^h$.
\end{proposition}
\begin{proof}
Due to the definition of orthogonal projection \cite[Definition 8.3]{hunter2001applied}, we need to show that  
$\big(\mcPh[K^h]\big)^2 = \mcPh[K^h]$ and 
$\langle \mcPh[K^h]\vec\tau^h_1,\vec\tau^h_2\rangle_{\mcX(\tilde{\Omega}^h,2)}
		=\langle \vec\tau^h_1,\mcPh[K^h]\vec\tau^h_2\rangle_{\mcX(\tilde{\Omega}^h,2)}$~for all $\vec\tau^h_1,\vec\tau^h_2\in\mcX(\tilde{\Omega}^h,2).$
The first equality follows, since the Moore-Penrose-Inverse fulfills $(\Delta^h_{\tilde{\Omega}^h})^\dagger=(\Delta^h_{\tilde{\Omega}^h})^\dagger\Delta^h_{\tilde{\Omega}^h}(\Delta^h_{\tilde{\Omega}^h})^\dagger$:
		\begin{align*}
		\big(\mcPh[K^h]\big)^2\vec\tau^h
		&=\mcPh[K^h]\vec\tau^h - \grad_{\tilde{\Omega}^h}^h(\Delta_{\tilde{\Omega}^h}^h)^\dagger\div_{\tilde{\Omega}^h}^h\mcPh[K^h]\vec\tau^h \\
		&=\mcPh[K^h]\vec\tau^h - \grad_{\tilde{\Omega}^h}^h(\Delta_{\tilde{\Omega}^h}^h)^\dagger\div_{\tilde{\Omega}^h}^h(\vec\tau^h - \grad_{\tilde{\Omega}^h}^h(\Delta_{\tilde{\Omega}^h}^h)^\dagger\div_{\tilde{\Omega}^h}^h\vec\tau^h) \\
		&=\mcPh[K^h]\vec\tau^h - \grad_{\tilde{\Omega}^h}^h(\Delta_{\tilde{\Omega}^h}^h)^\dagger\div_{\tilde{\Omega}^h}^h\vec\tau^h + \grad_{\tilde{\Omega}^h}^h(\Delta_{\tilde{\Omega}^h}^h)^\dagger\div_{\tilde{\Omega}^h}^h\grad_{\tilde{\Omega}^h}^h(\Delta_{\tilde{\Omega}^h}^h)^\dagger\div_{\tilde{\Omega}^h}^h\vec\tau^h \\
		&=\mcPh[K^h]\vec\tau^h - \grad_{\tilde{\Omega}^h}^h(\Delta_{\tilde{\Omega}^h}^h)^\dagger\div_{\tilde{\Omega}^h}^h\vec\tau^h + \grad_{\tilde{\Omega}^h}^h(\Delta_{\tilde{\Omega}^h}^h)^\dagger\Delta_{\tilde{\Omega}^h}^h(\Delta_{\tilde{\Omega}^h}^h)^\dagger\div_{\tilde{\Omega}^h}^h\vec\tau^h \\
		&=\mcPh[K^h]\vec\tau^h - \grad_{\tilde{\Omega}^h}^h(\Delta_{\tilde{\Omega}^h}^h)^\dagger\div_{\tilde{\Omega}^h}^h\vec\tau^h + \grad_{\tilde{\Omega}^h}^h(\Delta_{\tilde{\Omega}^h}^h)^\dagger\div_{\tilde{\Omega}^h}\vec\tau^h=\mcPh[K^h]\vec\tau^h.
		\end{align*}
		Using \cref{lmDiscLaplSym} and the fact that $(\div_{\tilde{\Omega}^h}^h)^*=-\grad_{\tilde{\Omega}^h}^h$, the second equality follows:
		\begin{align*}
		\langle \mcPh[K^h]\vec\tau^h_1,\vec\tau^h_2\rangle_{\mcX(\tilde{\Omega}^h,2)}
		&=\langle \vec\tau^h_1,\vec\tau^h_2\rangle_{\mcX(\tilde{\Omega}^h,2)}-\langle\grad_{\tilde{\Omega}^h}^h(\Delta_{\tilde{\Omega}^h}^h)^\dagger\div_{\tilde{\Omega}^h}^h\vec\tau^h_1,\vec\tau^h_2\rangle_{\mcX(\tilde{\Omega}^h,2)} \\
		&= \langle \vec\tau^h_1,\vec\tau^h_2\rangle_{\mcX(\tilde{\Omega}^h,2)}+\langle(\Delta_{\tilde{\Omega}^h}^h)^\dagger\div_{\tilde{\Omega}^h}^h\vec\tau^h_1,\div_{\tilde{\Omega}^h}^h\vec\tau^h_2\rangle_{\mcX(\tilde{\Omega}^h,1)}\\
		&= \langle \vec\tau^h_1,\vec\tau^h_2\rangle_{\mcX(\tilde{\Omega}^h,2)}+\langle\div_{\tilde{\Omega}^h}\vec\tau^h_1,(\Delta_{\tilde{\Omega}^h}^h)^\dagger\div_{\tilde{\Omega}^h}^h\vec\tau^h_2\rangle_{\mcX(\tilde{\Omega}^h,1)}\\
		&= \langle \vec\tau^h_1,\vec\tau^h_2\rangle_{\mcX(\tilde{\Omega}^h,2)}-\langle\vec\tau^h_1,\grad_{\tilde{\Omega}^h}^h(\Delta_{\tilde{\Omega}^h}^h)^\dagger\div_{\tilde{\Omega}^h}^h\vec\tau^h_2\rangle_{\mcX(\tilde{\Omega}^h,2)}
		=\langle \vec\tau^h_1,\mcPh[K^h]\vec\tau^h_2\rangle_{\mcX(\tilde{\Omega}^h,2)}
		\end{align*}
		for all $\vec\tau^h_1,\vec\tau^h_2\in\mcX(\tilde{\Omega}^h,2)$.
\end{proof}




\end{appendices}

\bibliography{ref_tvsd}

@book{Brezis:11,
	title={Functional analysis, Sobolev spaces and partial differential equations},
	author={Brezis, Haim},
	volume={2},
	number={3},
	year={2011},
	publisher={Springer},
	address={New York}
}

@book{Morrison:2011,
	title={Functional analysis: An introduction to Banach space theory},
	author={Morrison, Terry J},
	year={2011},
	publisher={John Wiley \& Sons},
	address={New York}
}

@book{rieder:2003,
	author = {Andreas Rieder},
	year = {Oktober 2003},
	title = {Keine Probleme mit Inversen Problemen},
	publisher = {Friedr. Vieweg \& Sohn Verlag/GWV Fachverlage GmbH,1. Auflage},
	address = {Wiesbaden}
}

@inproceedings{dualtvstokes:2009,
	title={A dual formulation of the {TV}-{S}tokes algorithm for image denoising},
	author={Elo, Christoffer A and Malyshev, Alexander and Rahman, Talal},
	booktitle={Scale Space and Variational Methods in Computer Vision: Second International Conference, SSVM 2009, Voss, Norway, June 1-5, 2009. Proceedings 2},
	pages={307--318},
	year={2009},
	organization={Springer}
}

@Article{Chambolle:2004,
	author	= {Chambolle, Antonin},
	title		= {An algorithm for total variation minimization and
	applications},
	journal	= {J. Math. Imaging Vision},
	year		= {2004},
	volume	= {20},
	number	= {1-2},
	pages		= {89--97},
	issn		= {0924-9907},
	note		= {Special issue on mathematics and image analysis},
	coden		= {JMIVEK},
	fjournal	= {Journal of Mathematical Imaging and Vision},
	mrclass	= {49M30 (65R32 68T45 68U10 90C30)},
	mrnumber	= {2049783 (2005m:49058)}
}

@book{ekeland:1999,
	author = {I. Ekeland and R. Témam},
	title = {Convex Analysis and Variational Problems},
	volume = {28},
	series = {Classics in Applied Mathematics},
	publisher = {Society for Industrial and Applied Mathematics (SIAM),
	Philadelphia, PA, English edition},
	year = {1999},
	note = {Translated from the French.},
	address = {Amsterdam}
}

@book{rockefellar:1970,
	author = {R. T. Rockafellar},
	title = {Convex Analysis},
	series = {Princeton Mathematical Series},
	number = {28},
	publisher = {Princeton University Press},
	year = {1970},
	address = {Princeton, New Jersey}
}

@book{ambrosio:2000,
	author = {L. Ambrosio and N. Fusco and D. Pallara},
	title = {Functions of Bounded Variation and Free Discontinuity
	Problems},
	publisher = {Oxford Mathematical Monographs. The Clarendon Press/Oxford University Press},
	year = {2000},
	address ={New York}
}

@article{HilbLanger2022,
	title={A general decomposition method for a convex problem related to total variation minimization},
	author={Hilb, Stephan and Langer, Andreas},
	journal={arXiv preprint arXiv:2211.00101},
	year={2022}
}

@Article{	  HintermullerKunisch:04,
	author	= {Hinterm{\"u}ller, Michael and Kunisch, Karl},
	title		= {Total Bounded Variation Regularization as a Bilaterally
	Constrained Optimization Problem},
	journal	= {SIAM Journal on Applied Mathematics},
	year		= {2004},
	volume	= {64},
	number	= {4},
	pages		= {1311--1333},
	publisher	= {SIAM},
	url		= {https://doi.org/10.1137/S0036139903422784}
}

@book{Kreyszig:1991,
	title={Introductory Functional Analysis with Applications},
	author={Kreyszig, Erwin},
	volume={17},
	year={1991},
	publisher={John Wiley \& Sons},
	address ={New York}
}

@Article{	  AcarVogel:94,
	author	= {Acar, R. and Vogel, C. R.},
	title		= {Analysis of bounded variation penalty methods for
	ill-posed problems},
	journal	= {Inverse Problems},
	year		= {1994},
	volume	= {10},
	pages		= {1217--1229},
	number	= {6},
	coden		= {INPEEY},
	fjournal	= {Inverse Problems. An International Journal on the Theory
	and Practice of Inverse Problems, Inverse Methods and
	Computerized Inversion of Data},
	issn		= {0266-5611},
	mrclass	= {65J20 (47A50 49J45)},
	mrnumber	= {1306801 (95i:65092)},
	url		= {http://stacks.iop.org/0266-5611/10/1217}
}

@article {LiRaTa2011,
	AUTHOR = {Litvinov, William G. and Rahman, Talal and Tai, Xue-Cheng},
	TITLE = {A modified {TV}-{S}tokes model for image processing},
	JOURNAL = {SIAM J. Sci. Comput.},
	FJOURNAL = {SIAM Journal on Scientific Computing},
	VOLUME = {33},
	YEAR = {2011},
	NUMBER = {4},
	PAGES = {1574--1597},
	ISSN = {1064-8275,1095-7197},
	MRCLASS = {65D18 (65K10 68U10 94A08)},
	MRNUMBER = {2821259},
	MRREVIEWER = {Pierluigi\ Maponi},
	DOI = {10.1137/080727506},
	URL = {https://doi.org/10.1137/080727506},
}

@inproceedings{RaTaOs2007,
	title={A {TV}-{S}tokes denoising algorithm},
	author={Rahman, Talal and Tai, Xue-Cheng and Osher, Stanley},
	booktitle={Scale Space and Variational Methods in Computer Vision: First International Conference, SSVM 2007, Ischia, Italy, May 30-June 2, 2007. Proceedings 1},
	pages={473--483},
	year={2007},
	organization={Springer}
}

@article{HiLaAl2023,
	title={A primal-dual finite element method for scalar and vectorial total variation minimization},
	author={Hilb, Stephan and Langer, Andreas and Alk{\"a}mper, Martin},
	journal={Journal of Scientific Computing},
	volume={96},
	number={1},
	pages={24},
	year={2023},
	publisher={Springer}
}

@article{	  HintermullerRautenberg:15,
	title		= {On the density of classes of closed convex sets with
	pointwise constraints in Sobolev spaces},
	author	= {Hinterm{\"u}ller, M and Rautenberg, CN},
	journal	= {Journal of Mathematical Analysis and Applications},
	volume	= {426},
	number	= {1},
	pages		= {585--593},
	year		= {2015},
	publisher	= {Elsevier}
}

@Article{ChambolleLions1997,
	author	= {Chambolle, Antonin and Lions, Pierre-Louis},
	title		= {Image recovery via total variation minimization and
	related problems},
	journal	= {Numer. Math.},
	year		= {1997},
	volume	= {76},
	pages		= {167--188},
	number	= {2},
	coden		= {NUMMA7},
	doi		= {10.1007/s002110050258},
	fjournal	= {Numerische Mathematik},
	issn		= {0029-599X},
	mrclass	= {65K10 (49J99 68U10)},
	mrnumber	= {1440119 (98c:65099)},
	mrreviewer	= {Tom{\'a}{\v{s}} Roub{\'{\i}}{\v{c}}ek},
	url		= {http://dx.doi.org/10.1007/s002110050258}
}

@article{	  GoStCr:12,
	title		= {The Natural Vectorial Total Variation
	Which Arises from Geometric Measure Theory},
	author	= {Goldluecke, Bastian and Strekalovskiy, Evegeny and Cremers, Daniel},
	doi = {10.1137/110823766},
	pages		= {537-563},
	year		= {2012},
	journal = {{SIAM} Journal on Imaging Sciences},
	volume = {5},
	number = {2}
}

@Article{Langer2017,
	author	= {Langer, Andreas},
	title		= {Automated Parameter Selection for Total Variation
	Minimization in Image Restoration},
	journal	= {Journal of Mathematical Imaging and Vision},
	year		= {2017},
	volume	= {57},
	number	= {2},
	pages		= {239--268},
	month		= {Feb},
	issn		= {1573-7683},
	day		= {01},
	doi		= {10.1007/s10851-016-0676-2},
	url		= {https://doi.org/10.1007/s10851-016-0676-2}
}

@article{Langer2017_2,
	title={Automated parameter selection in the {$L^1$-$L^2$-TV} model for removing {G}aussian plus impulse noise},
	author={Langer, Andreas},
	journal={Inverse Problems},
	volume={33},
	number={7},
	pages={074002},
	year={2017},
	publisher={IOP Publishing}
}

@Book{		  AtBuMi:14,
	title		= {Variational Analysis in {S}obolev and {BV} Spaces},
	publisher	= {Society for Industrial and Applied Mathematics (SIAM),
	Philadelphia, PA; Mathematical Optimization Society},
	year		= {2014},
	author	= {Attouch, Hedy and Buttazzo, Giuseppe and Michaille,
	G{\'e}rard},
	pages		= {xii+793},
	series	= {MOS-SIAM Series on Optimization},
	edition	= {Second},
	note		= {Applications to PDEs and optimization},
	doi		= {10.1137/1.9781611973488},
	isbn		= {978-1-611973-47-1},
	mrclass	= {49-02 (46E35 49J45 49J53 49K20 74G65)},
	mrnumber	= {3288271},
	mrreviewer	= {Luca Granieri},
	url		= {http://dx.doi.org/10.1137/1.9781611973488},
	address = {Philadelphia, PA}
}

@book{GirRav2012,
	title={Finite element methods for Navier-Stokes equations: theory and algorithms},
	author={Girault, Vivette and Raviart, Pierre-Arnaud},
	volume={5},
	year={2012},
	publisher={Springer Science \& Business Media},
	address = {Berlin, Heidelberg}
}

@article{ROF,
	title={Nonlinear total variation based noise removal algorithms},
	author={Rudin, Leonid I and Osher, Stanley and Fatemi, Emad},
	journal={Physica D: Nonlinear Phenomena},
	volume={60},
	number={1},
	pages={259--268},
	year={1992},
	publisher={Elsevier}
}

@article{LysOshTai2004,
	author    = {Marius Lysaker and Stanley Osher and Xue-Cheng Tai},
	title     = {Noise Removal Using Smoothed Normals and Surface Fitting},
	journal   = {IEEE Transactions on Image Processing},
	volume    = {13},
	number    = {10},
	pages     = {1345--1357},
	year      = {2004},
	doi       = {10.1109/TIP.2004.834662}
}

@article{Jalalzai2016,
	title={Some remarks on the staircasing phenomenon in total variation-based image denoising},
	author={Jalalzai, Khalid},
	journal	= {J. Math. Imaging Vision},
	fjournal	= {Journal of Mathematical Imaging and Vision},
	volume={54},
	pages={256--268},
	year={2016},
	publisher={Springer}
}

@article {BreKunPoc,
	AUTHOR = {Bredies, Kristian and Kunisch, Karl and Pock, Thomas},
	TITLE = {Total generalized variation},
	JOURNAL = {SIAM J. Imaging Sci.},
	FJOURNAL = {SIAM Journal on Imaging Sciences},
	VOLUME = {3},
	YEAR = {2010},
	NUMBER = {3},
	PAGES = {492--526},
	ISSN = {1936-4954},
	CODEN = {SJISBI},
	MRCLASS = {49J52 (49Q20 68U10)},
	MRNUMBER = {2736018 (2011k:49029)},
	MRREVIEWER = {Guy Jumarie},
	DOI = {10.1137/090769521},
	URL = {http://dx.doi.org/10.1137/090769521},
}

@Article{	  PapafitsorosSchonlieb:14,
	author	= {Papafitsoros, K. and Sch\"{o}nlieb, C.-B.},
	title		= {A combined first and second order variational approach for
	image reconstruction},
	journal	= {J. Math. Imaging Vision},
	fjournal	= {Journal of Mathematical Imaging and Vision},
	volume	= {48},
	year		= {2014},
	number	= {2},
	pages		= {308--338},
	issn		= {0924-9907},
	mrclass	= {94A08 (49N90)},
	mrnumber	= {3152107},
	doi		= {10.1007/s10851-013-0445-4},
	url		= {https://doi.org/10.1007/s10851-013-0445-4}
}

@article{ForKimLanSch,
	title={Wavelet Decomposition Method for {$L_2$/TV}-Image Deblurring},
	author={Fornasier, Massimo and Kim, Yunho and Langer, Andreas and Sch\"onlieb, C-B},
	JOURNAL = {SIAM J. Imaging Sci.},
	FJOURNAL = {SIAM Journal on Imaging Sciences},
	volume={5},
	number={3},
	pages={857--885},
	year={2012},
	publisher={SIAM}
}

@article{ForLanSch2010,
	title={A convergent overlapping domain decomposition method for total variation minimization},
	author={Fornasier, Massimo and Langer, Andreas and Sch{\"o}nlieb, Carola-Bibiane},
	JOURNAL = {Numer. Math.},
	FJOURNAL = {Numerische Mathematik},
	volume={116},
	number={4},
	pages={645--685},
	year={2010},
	publisher={Springer}
}

@article {ForSch,
	AUTHOR = {Fornasier, Massimo and Sch{\"o}nlieb, Carola-Bibiane},
	TITLE = {Subspace correction methods for total variation and {$l_1$}-minimization},
	JOURNAL = {SIAM J. Numer. Anal.},
	FJOURNAL = {SIAM Journal on Numerical Analysis},
	VOLUME = {47},
	YEAR = {2009},
	NUMBER = {5},
	PAGES = {3397--3428},
	ISSN = {0036-1429},
	MRCLASS = {65K10 (49M30 52A41 90C25)},
	MRNUMBER = {2551200 (2011d:65150)},
	MRREVIEWER = {Jonas Koko},
	DOI = {10.1137/070710779},
	URL = {http://dx.doi.org/10.1137/070710779}
}

@article{LanOshSch,
	title={Bregmanized domain decomposition for image restoration},
	author={Langer, Andreas and Osher, Stanley and Sch{\"o}nlieb, Carola-Bibiane},
	JOURNAL = {J. Sci. Comput.},
	FJOURNAL = {Journal of Scientific Computing},
	volume={54},
	number={2-3},
	pages={549--576},
	year={2013},
	publisher={Springer}
}

@incollection{Lan2021,
	title={Domain Decomposition for Non-smooth (in Particular {TV}) Minimization},
	author={Langer, Andreas},
	booktitle={Handbook of Mathematical Models and Algorithms in Computer Vision and Imaging: Mathematical Imaging and Vision},
	pages={1--47},
	year={2021},
	publisher={Springer},
	address={Cham},
	doi={10.1007/978-3-030-03009-4_104-1}
}

@article{LeeNam,
	title={Primal Domain Decomposition Methods for the Total Variation Minimization, Based on Dual Decomposition},
	author={Lee, Chang-Ock and Nam, Changmin},
	JOURNAL = {SIAM J. Sci. Comput.},
	FJOURNAL = {SIAM Journal on Scientific Computing},
	volume={39},
	number={2},
	pages={B403--B423},
	year={2017},
	publisher={SIAM}
}

@article {HinLan2013,
	AUTHOR = {Hinterm{\"u}ller, Michael and Langer, Andreas},
	TITLE = {Subspace correction methods for a class of nonsmooth and
	nonadditive convex variational problems with mixed {$L^1/L^2$} data-fidelity in image processing},
	JOURNAL = {SIAM J. Imaging Sci.},
	FJOURNAL = {SIAM Journal on Imaging Sciences},
	VOLUME = {6},
	YEAR = {2013},
	NUMBER = {4},
	PAGES = {2134--2173},
	ISSN = {1936-4954},
	MRCLASS = {94A08 (49M27 65K10 68U10 90C25 90C90)},
	MRNUMBER = {3121758},
	MRREVIEWER = {V. S. Sizikov},
	DOI = {10.1137/120894130},
	URL = {http://dx.doi.org/10.1137/120894130},
}

@incollection{HinLan2014,
	title={Surrogate Functional Based Subspace Correction Methods for Image Processing},
	author={Hinterm{\"u}ller, Michael and Langer, Andreas},
	booktitle={Domain Decomposition Methods in Science and Engineering XXI},
	pages={829--837},
	year={2014},
	publisher={Springer},
	address={Cham},
	doi={10.1007/978-3-319-05789-7_80}
}

@article{HinLan2015_1,
	title={Non-overlapping domain decomposition methods for dual total variation based image denoising},
	author={Hinterm{\"u}ller, Michael and Langer, Andreas},
	JOURNAL = {J. Sci. Comput.},
	FJOURNAL = {Journal of Scientific Computing},
	volume={62},
	number={2},
	pages={456--481},
	year={2015},
	publisher={Springer}
}

@article{ChaTaiWanYan,
	title={Convergence Rate of Overlapping Domain Decomposition Methods for the {R}udin--{O}sher--{F}atemi Model Based on a Dual Formulation},
	author={Chang, Huibin and Tai, Xue-Cheng and Wang, Li-Lian and Yang, Danping},
	JOURNAL = {SIAM J. Imaging Sci.},
	FJOURNAL = {SIAM Journal on Imaging Sciences},
	volume={8},
	number={1},
	pages={564--591},
	year={2015},
	publisher={SIAM}
}

@Article{LangerGaspoz:19,
	title		= {Overlapping Domain Decomposition Methods for Total
	Variation Denoising},
	author	= {Langer, Andreas and Gaspoz, Fernando},
	JOURNAL = {SIAM J. Numer. Anal.},
	FJOURNAL = {SIAM Journal on Numerical Analysis},
	volume	= {57},
	number	= {3},
	pages		= {1411--1444},
	year		= {2019},
	publisher	= {SIAM}
}

@article{LeeNamPark2019,
	title={Domain Decomposition Methods Using Dual Conversion for the Total Variation Minimization with {$L^{1}$} Fidelity Term},
	author={Lee, Chang-Ock and Nam, Changmin and Park, Jongho},
	JOURNAL = {J. Sci. Comput.},
	FJOURNAL = {Journal of Scientific Computing},
	volume={78},
	number={2},
	pages={951--970},
	year={2019},
	publisher={Springer}
}

@article{LeePark:20,
	title={Recent advances in domain decomposition methods for total variation minimization},
	author={Lee, Chang-Ock and Park, Jongho},
	JOURNAL = {J. Korean Soc. Ind. Appl. Math.},
	FJOURNAL = {Journal of the Korean Society for Industrial and Applied
	Mathematics},
	volume={24},
	number={2},
	pages={161--197},
	year={2020},
	publisher={Korean Society for Industrial and Applied Mathematics}
}

@article{LeePark2019,
	title={A finite element nonoverlapping domain decomposition method with {L}agrange multipliers for the dual total variation minimizations},
	author={Lee, Chang-Ock and Park, Jongho},
	JOURNAL = {J. Sci. Comput.},
	FJOURNAL = {Journal of Scientific Computing},
	volume={81},
	number={3},
	pages={2331--2355},
	year={2019},
	publisher={Springer}
}

@article{Lee2019fast,
	title={Fast Nonoverlapping Block {Jacobi} Method for the Dual {Rudin--Osher--Fatemi} Model},
	author={Lee, Chang-Ock and Park, Jongho},
	JOURNAL = {SIAM J. Imaging Sci.},
	FJOURNAL = {SIAM Journal on Imaging Sciences},
	volume={12},
	number={4},
	pages={2009--2034},
	year={2019},
	publisher={SIAM}
}

@article{LeeParkPark2019,
	title={A finite element approach for the dual {R}udin--{O}sher--{F}atemi model and its nonoverlapping domain decomposition methods},
	author={Lee, Chang-Ock and Park, Eun-Hee and Park, Jongho},
	JOURNAL = {SIAM J. Sci. Comput.},
	FJOURNAL = {SIAM Journal on Scientific Computing},
	volume={41},
	number={2},
	pages={B205--B228},
	year={2019},
	publisher={SIAM}
}

@article{LiZhangChangDuan2021,
	AUTHOR = {Li, Xue and Zhang, Zhenwei and Chang, Huibin and Duan, Yuping},
	TITLE = {Accelerated non-overlapping domain decomposition method for
	total variation minimization},
	JOURNAL = {Numer. Math. Theory Methods Appl.},
	FJOURNAL = {Numerical Mathematics. Theory, Methods and Applications},
	VOLUME = {14},
	YEAR = {2021},
	NUMBER = {4},
	PAGES = {1017--1041},
	ISSN = {1004-8979,2079-7338},
	MRCLASS = {65K05 (49J27 65K10 65N55 68U10 94A08)},
	MRNUMBER = {4313167},
	MRREVIEWER = {Yan\ Tang},
	DOI = {10.4208/nmtma.oa-2020-0146},
	URL = {https://doi.org/10.4208/nmtma.oa-2020-0146}
}

@article{park2020overlapping,
	title={An overlapping domain decomposition framework without dual formulation for variational imaging problems},
	author={Park, Jongho},
	JOURNAL = {Adv. Comput. Math.},
	FJOURNAL = {Advances in Computational Mathematics},
	volume={46},
	number={4},
	pages={1--29},
	year={2020},
	publisher={Springer}
}

@article{park2020additive,
	title={Additive {S}chwarz methods for convex optimization as gradient methods},
	author={Park, Jongho},
	JOURNAL = {SIAM J. Numer. Anal.},
	FJOURNAL = {SIAM Journal on Numerical Analysis},
	volume={58},
	number={3},
	pages={1495--1530},
	year={2020},
	publisher={SIAM}
}

@article{park2021accelerated,
	title={Accelerated additive {S}chwarz methods for convex optimization with adaptive restart},
	author={Park, Jongho},
	JOURNAL = {J. Sci. Comput.},
	FJOURNAL = {Journal of Scientific Computing},
	volume={89},
	number={3},
	pages={1--20},
	year={2021},
	publisher={Springer}
}

@article{Car,
	author =        {Carstensen, Carsten},
	journal =       {Numerical linear algebra with applications},
	number =        {3},
	pages =         {177--190},
	title =         {Domain decomposition for a non-smooth convex
	minimization problem and its application to
	plasticity},
	volume =        {4},
	year =          {1997},
}

@article{ChaMat,
	author =        {Chan, Tony F and Mathew, Tarek P},
	journal =       {Acta numerica},
	pages =         {61--143},
	publisher =     {Cambridge University Press},
	title =         {Domain decomposition algorithms},
	volume =        {3},
	year =          {1994},
}

@article{TaiTse,
	author =        {Tai, Xue-Cheng and Tseng, Paul},
	journal =       {Mathematics of Computation},
	number =        {239},
	pages =         {1105--1135},
	title =         {Convergence rate analysis of an asynchronous space
	decomposition method for convex minimization},
	volume =        {71},
	year =          {2002},
}

@article{TaiXu,
	author =        {Tai, Xue-Cheng and Xu, Jinchao},
	journal =       {Mathematics of Computation},
	number =        {237},
	pages =         {105--124},
	title =         {Global and uniform convergence of subspace correction
	methods for some convex optimization problems},
	volume =        {71},
	year =          {2002},
}

@article{CheTai,
	author =        {Chen, Ke and Tai, Xue-Cheng},
	journal =       {Journal of Scientific Computing},
	number =        {2},
	pages =         {115--138},
	publisher =     {Springer},
	title =         {A nonlinear multigrid method for total variation
	minimization from image restoration},
	volume =        {33},
	year =          {2007},
}

@article{XuTaiWan,
	author =        {Xu, Jing and Tai, Xue-Cheng and Wang, Li-Lian},
	journal =       {Inverse Problems \& Imaging},
	number =        {3},
	pages =         {523--545},
	title =         {A two-level domain decomposition method for image
	restoration},
	volume =        {4},
	year =          {2010},
}

@article{DuanTai2012,
	title={Domain decomposition methods with graph cuts algorithms for total variation minimization},
	author={Duan, Yuping and Tai, Xue-Cheng},
	journal={Advances in Computational Mathematics},
	volume={36},
	pages={175--199},
	year={2012},
	publisher={Springer}
}

@book{hunter2001applied,
  title={Applied analysis},
  author={Hunter, John K and Nachtergaele, Bruno},
  year={2001},
  publisher={World Scientific},
  address={Singapore}
}

@book{TosWid,
  title={Domain Decomposition Methods: Algorithms and Theory},
  author={Toselli, Andrea and Widlund, Olof B},
  volume={34},
  year={2005},
  publisher={Springer},
  address={Berlin, Heidelberg},
  doi={10.1007/b137868}
}

@book{Sauer:12,
  title={Numerical Analysis},
  author={Sauer, Timothy},
  year={2012},
  publisher={Pearson},
  address={Boston}
}

\end{document}